\documentclass[12pt]{amsart}
\usepackage{amsmath}
\usepackage{amssymb}
\usepackage{xcolor}
\usepackage{tikz}
\usetikzlibrary{cd}
\usetikzlibrary{knots}
\usetikzlibrary{calc}

\newcommand{\roundNbox}[6]{
	\draw[rounded corners=5pt, very thick, #1] ($#2+(-#3,-#3)+(-#4,0)$) rectangle ($#2+(#3,#3)+(#5,0)$);
	\coordinate (ZZa) at ($#2+(-#4,0)$);
	\coordinate (ZZb) at ($#2+(#5,0)$);
	\node at ($1/2*(ZZa)+1/2*(ZZb)$) {#6};
}

\tikzset{string/.style={ultra thick}}
\tikzset{smallstring/.style={thick,scale=0.75,every node/.style={transform shape}}}
\tikzset{
    triple/.style args={[#1] in [#2] in [#3]}{
        #1,preaction={preaction={draw,#3},draw,#2}
    }
}
\tikzset{
    quadruple/.style args={[#1] in [#2] in [#3] in [#4]}{
        #1,preaction={preaction={preaction={draw,#4},draw,#3}, draw,#2}
    }
}
\tikzset{
	super thick/.style={line width=3pt},
	more thick/.style={line width=1pt},
}

\usepackage{graphicx}
\usepackage{fullpage}
\usepackage[T1]{fontenc}
\usepackage[final]{microtype}
\usepackage{libertine}
\usepackage[libertine]{newtxmath}
\usepackage{ifthen}
\usepackage[all]{xy}
\usepackage[normalem]{ulem}

\usepackage{xcolor}
\definecolor{dark-red}{rgb}{0.7,0.25,0.25}
\definecolor{dark-blue}{rgb}{0.15,0.15,0.55}
\definecolor{medium-blue}{rgb}{0,0,0.65}
\definecolor{DarkGreen}{RGB}{0,150,0}

\usepackage[pdftex,plainpages=false,hypertexnames=false,pdfpagelabels,breaklinks]{hyperref}
\hypersetup{
   colorlinks, linkcolor={purple},
   citecolor={medium-blue}, urlcolor={medium-blue}
}
\usepackage[utf8]{inputenc}
\usepackage[T1]{fontenc}

\usepackage[backend=biber,style=alphabetic,doi=false,isbn=false,url=false,maxnames=6,maxalphanames=6]{biblatex}
\renewbibmacro{in:}{%
  \ifentrytype{article}{}{\printtext{\bibstring{in}\intitlepunct}}}
\renewbibmacro*{volume+number+eid}{%
  \printtext{vol.}
  \printfield{volume}%
  \setunit*{\addnbspace}% NEW (optional); there's also \addnbthinspace
  \printfield{number}%
  \setunit{\addcomma\space}%
  \printfield{eid}}
\DeclareFieldFormat[article]{number}{\mkbibparens{#1}}
\usepackage{silence}
\newcommand{\doi}[1]{\href{https://dx.doi.org/#1}{{\small DOI:#1}}}

\newcommand{\googlebooks}[1]{(preview at \href{https://books.google.com/books?id=#1}{google books})}

\newcommand{\numdam}[1]{}

\theoremstyle{plain}
\newtheorem{prop}{Proposition}[section]

\newtheorem{thm}[prop]{Theorem}
\newtheorem{lem}[prop]{Lemma}
\newtheorem{cor}[prop]{Corollary}
\numberwithin{equation}{section}

\theoremstyle{remark}
\newtheorem{example}[prop]{Example}

\newtheorem{remark}[prop]{Remark}

\theoremstyle{definition}
\newtheorem{defn}[prop]{Definition}         % numbered definition
\newtheorem{fact}[prop]{Fact}

\newtheorem{assume}[prop]{Assumption}

\newcommand{\sslash}{\mathbin{/\mkern-6mu/}}

\DeclareMathOperator{\Ad}{Ad}

\DeclareMathOperator{\id}{id}
\DeclareMathOperator{\im}{im}
\DeclareMathOperator{\Inv}{Inv}

\DeclareMathOperator{\core}{core}
\DeclareMathOperator{\Hom}{Hom}
\DeclareMathOperator{\End}{End}
\DeclareMathOperator{\Forget}{Forget}
\DeclareMathOperator{\FPdim}{FPdim}

\DeclareMathOperator{\Irr}{Irr}
\DeclareMathOperator{\Aut}{Aut}

\DeclareMathOperator{\op}{op}

\DeclareMathOperator{\Stab}{Stab}

\newcommand{\Path}{{\mathsf {Path}}}
\newcommand{\hoFib}{{\mathsf {hoFib}}}
\newcommand{\stFib}{{\mathsf {stFib}}}
\newcommand{\Fib}{{\mathsf {Fib}}}

\newcommand{\BrPic}{{\mathsf {BrPic}}}
\newcommand{\Ext}{{\mathsf {Ext}}}

\newcommand{\uBrPic}{\uline{\mathsf {BrPic}}}
\newcommand{\uuBrPic}{\uuline{\mathsf {BrPic}}}

\newcommand{\uuPic}{\uuline{\mathsf {Pic}}}
\newcommand{\uAut}{\uline{\Aut}}
\newcommand{\uStab}{\uline{\Stab}}
\newcommand{\uG}{\uline{G}}
\newcommand{\uuG}{\uuline{G}}
\newcommand{\urho}{\uline{\rho}}
\newcommand{\uurho}{\uuline{\rho}}
\newcommand{\uupi}{\uuline{\pi}}
\newcommand{\Cat}{{\mathsf {Cat}}}
\newcommand{\MonCat}{{\mathsf {MonCat}}}
\newcommand{\FusCat}{{\mathsf {FusCat}}}
\newcommand{\BrdMonCat}{{\mathsf {BrdMonCat}}}
\newcommand{\GrdFusCat}{{\mathsf {GrdFusCat}}}
\newcommand{\GrdMonCat}{{\mathsf {GrdMonCat}}}
\newcommand{\FibFusCat}{{\mathsf {FibFusCat}}}
\newcommand{\CrsBrd}{{\mathsf {CrsBrd}}}
\newcommand{\BrFus}{{\mathsf {BrdFus}}}
\renewcommand{\Vec}{{\mathsf {Vec}}}

\newcommand{\Rep}{{\mathsf {Rep}}}

\newcommand{\br}{{\rm br}}
\newcommand{\rev}{{\rm rev}}
\newcommand{\loc}{{\rm loc}}
\newcommand{\sZ}{{\scriptscriptstyle Z}}
\newcommand{\set}[2]{\left\{#1\middle|#2\right\}}

\def\semicolon{;}
\def\applytolist#1{
    \expandafter\def\csname multi#1\endcsname##1{
        \def\multiack{##1}\ifx\multiack\semicolon
            \def\next{\relax}
        \else
            \csname #1\endcsname{##1}
            \def\next{\csname multi#1\endcsname}
        \fi
        \next}
    \csname multi#1\endcsname}

\def\calc#1{\expandafter\def\csname c#1\endcsname{{\mathcal #1}}}
\applytolist{calc}QWERTYUIOPLKJHGFDSAZXCVBNM;
\def\bbc#1{\expandafter\def\csname bb#1\endcsname{{\mathbb #1}}}
\applytolist{bbc}QWERTYUIOPLKJHGFDSAZXCVBNM;
\def\bfc#1{\expandafter\def\csname bf#1\endcsname{{\mathbf #1}}}
\applytolist{bfc}QWERTYUIOPLKJHGFDSAZXCVBNM;

\makeatletter
\newlength{\L@UnitsRaiseDisplaystyle}
\newlength{\L@UnitsRaiseTextstyle}
\newlength{\L@UnitsRaiseScriptstyle}
\DeclareRobustCommand*{\@UnitsNiceFrac}[3][]{%
  \ifthenelse{\boolean{mmode}}{%
    \settoheight{\L@UnitsRaiseDisplaystyle}{%
      \ensuremath{\displaystyle#1{M}}%
    }%
    \settoheight{\L@UnitsRaiseTextstyle}{%
      \ensuremath{\textstyle#1{M}}%
    }%
    \settoheight{\L@UnitsRaiseScriptstyle}{%
      \ensuremath{\scriptstyle#1{M}}%
    }%
    \settoheight{\@tempdima}{%
      \ensuremath{\scriptscriptstyle#1{M}}%
    }%
    \addtolength{\L@UnitsRaiseDisplaystyle}{%
      -\L@UnitsRaiseScriptstyle%
    }%
    \addtolength{\L@UnitsRaiseTextstyle}{%
      -\L@UnitsRaiseScriptstyle%
    }%
    \addtolength{\L@UnitsRaiseScriptstyle}{-\@tempdima}%
    \mathchoice
      {%
        \raisebox{\L@UnitsRaiseDisplaystyle}{%
          \ensuremath{\scriptstyle#1{#2}}%
        }%
      }%
      {%
        \raisebox{\L@UnitsRaiseTextstyle}{%
          \ensuremath{\scriptstyle#1{#2}}%
        }%
      }%
      {%
        \raisebox{\L@UnitsRaiseScriptstyle}{%
          \ensuremath{\scriptscriptstyle#1{#2}}%
        }%
      }%
      {%
        \raisebox{\L@UnitsRaiseScriptstyle}{%
          \ensuremath{\scriptscriptstyle#1{#2}}%
        }%
      }%
    \mkern-2mu{\sslash}\mkern-1mu%
    \bgroup
      \mathchoice
        {\scriptstyle}%
        {\scriptstyle}%
        {\scriptscriptstyle}%
        {\scriptscriptstyle}%
      #1{#3}%
    \egroup
  }%
  {%
    \settoheight{\L@UnitsRaiseTextstyle}{#1{M}}%
    \settoheight{\@tempdima}{%
      \ensuremath{%
        \mbox{\fontsize\sf@size\z@\selectfont#1{M}}%
      }%
    }%
    \addtolength{\L@UnitsRaiseTextstyle}{-\@tempdima}%
    \raisebox{\L@UnitsRaiseTextstyle}{%
      \ensuremath{%
        \mbox{\fontsize\sf@size\z@\selectfont#1{#2}}%
      }%
    }%
    \ensuremath{\mkern-2mu}{\sslash}\ensuremath{\mkern-1mu}%
    \ensuremath{%
      \mbox{\fontsize\sf@size\z@\selectfont#1{#3}}%
    }%
  }%
}

\DeclareRobustCommand*{\nicefrac}{\@UnitsNiceFrac}%
\makeatother

\begin{document}
\title{Extension theory for braided-enriched fusion categories}
\author{Corey Jones}
  \address[Corey Jones]{
    Department of Mathematics,
    North Carolina State University
  }
  \email{\href{mailto:cmjones6@ncsu.edu}{cmjones6@ncsu.edu}}
\author{Scott Morrison}
 \address[Scott Morrison]{
   University of Sydney
%  	Mathematical Sciences Institute,
%  	Australian National University
  }
  \email{\href{mailto:scott@tqft.net}{scott@tqft.net}}
\author{David Penneys}
  \address[David Penneys]{
    Department of Mathematics,
    The Ohio State University
  }
  \email{\href{mailto:penneys.2@osu.edu}{penneys.2@osu.edu}}
\author{Julia Plavnik}
  \address[Julia Plavnik]{
    Department of Mathematics,
    Indiana University
  }
  \email{\href{mailto:jplavnik@iu.edu}{jplavnik@iu.edu}}
% \date{\today}
\begin{abstract}
For a braided fusion category $\mathcal{V}$, a $\mathcal{V}$-fusion category is a fusion category
$\mathcal{C}$ equipped with a braided monoidal functor $\mathcal{F}:\mathcal{V} \to Z(\mathcal{C})$.
Given a fixed $\mathcal{V}$-fusion category $(\mathcal{C}, \mathcal{F})$ and a fixed $G$-graded
extension $\mathcal{C}\subseteq \mathcal{D}$ as an ordinary fusion category, we characterize the
enrichments $\widetilde{\mathcal{F}}:\mathcal{V} \to Z(\mathcal{D})$ of $\mathcal{D}$ which are
compatible with the enrichment of $\mathcal{C}$.
%A $G$-graded extension of fusion categories $\mathcal{C} \subseteq \mathcal{D}$ gives a braided categorical $G$-action $\underline{G} \to \underline{\operatorname{Aut}}_\otimes^{\rm br}(Z(\mathcal{C}))$, and one such characterization of compatible enrichments is in terms of $G$-equivariant structures on $\mathcal{F}$.
%If the categorical $G$-action satisfies $\mathcal{F} \cong g \circ \mathcal{F}$ for all $g\in G$, we show there is a certain short exact sequence whose splittings are in bijective correspondence with these $G$-equivariant structures on $\mathcal{F}$.
We show that G-crossed extensions of a braided fusion category $\mathcal{C}$ are G-extensions of the
canonical enrichment of $\mathcal{C}$ over itself. As an application, we parameterize the set of $G$-crossed
braidings on a fixed $G$-graded fusion category in terms of certain subcategories of its center,
extending Nikshych's classification of the braidings on a fusion category.
%We give a characterization of this set of $G$-crossed braidings in terms of group theoretical data for various group theoretical fusion categories.
%This is the submitted version of \arxiv{???}
\end{abstract}
\maketitle

%%%%%%%%%%%%%%%%%%%%%%%%%%%%%%%%%%%%%%%%%%%%%%%%%%%%%%%%%%%%%%%%%%%%%%%%%%%%%%%%%%%%%%%%%%%%%%%%%%%%%%%%%%%%%%%%%%%%%%%%%%%%%%%%%%%%%%%%%%%%%%%%%%%%%%%%%%%%%%%%%%%%%%%%%%%%%%%%%%%%%%%%%%%%%%%%%%%%%%%%%%%%%%%%%%%%%%%%%%%%%%
\section{Introduction}

In previous articles \cite{MR3961709,1809.09782} we defined monoidal categories enriched in a
braided monoidal category $\cV$, and showed this notion was equivalent to an oplax, strongly unital,
braided monoidal functor from $\cV$ into the Drinfeld center of an ordinary monoidal category. When
the functor $\cF^\sZ : \cV \to Z(\cC)$ is strong monoidal, this coincides with the notion of a 1-morphism
$\cV \to \Vec$ in a suitable Morita 4-category \cite{MR4228258} (see also \S\ref{sec:4CatOfBraidedTensorCategories} below), and with the module tensor categories of
\cite{1607.06041}. 
Recent work of Kong and Zheng uses monoidal categories enriched in a braided
category to give a unified treatment of gapped and gapless edges for 2D topological orders
\cite{MR3763324,1903.12334,1905.04924}. 
Of particular importance is the case where $\cV$ is a
braided fusion category and $\cF^\sZ: \cV \to Z(\cC)$ is a braided strong monoidal functor into the
Drinfeld center of another fusion category $\cC$. 
We call such a pair $(\cC, \cF^\sZ)$ a
$\cV$-\emph{fusion category}.

The extension theory for fusion categories of \cite{MR2677836} has proven to be an immensely
important tool. 
Particular applications include the process of gauging a global symmetry on a
modular tensor category \cite{1410.4540,MR3555361}, permutation symmetries on modular tensor
categories \cite{MR3959559}, rank finiteness for ($G$-crossed) braided fusion categories
\cite{1902.06165}, and classification theorems for tensor categories generated by an object of small
dimension \cite{MR3952486,MR4076701}.

In this article, we define the notion of a $G$-graded extension of a $\cV$-fusion category. We begin
by proving that $G$-gradings on a fusion category $\cC$ are equivalent to
liftings of a fixed fiber functor $\Rep(G) \to \Vec = \langle 1_\cC\rangle \subseteq \cC$ to
$Z(\cC)$. Fixing such a $G$-grading $\cC = \bigoplus_{g\in G} \cC_g$, we see that an object $(c,
\sigma_{\cdot, c}) \in Z(\cC)$ satisfies $c\in \cC_e$ if and only if $(c, \sigma_{\cdot, c})$
lies in the M\"uger centralizer $\Rep(G)'$. 
Given this, we define a $G$-graded $\cV$-fusion category to be a $\cV$-fusion
category $(\cC,\cF^\sZ)$ such that the underlying fusion category $\cC = \bigoplus_{g\in G}\cC_g$ is
$G$-graded and $\cF^\sZ(\cV) \subseteq \Rep(G)' \subset Z(\cC) $.

\begin{thm}
\label{thm:FirstMainTheorem}
Fix a $G$-graded extension $\cC \subseteq \cD$ of ordinary fusion categories, and a $\cV$-fusion category structure $(\cC, \cF^\sZ)$ on $\cC$.
The following sets are in canonical bijection.
\begin{itemize}
\item
For all $v\in \cV$, extensions of the half-braiding for $\cF^{\scriptscriptstyle Z}(v)$ with $\cC$
to a half-braiding with all of $\cD$ coherently with respect to morphisms in $\cV$.
% \item
% Lifts $\widetilde{\cF}^{\scriptscriptstyle Z} : \cV \to \Rep(G)'\cap Z(\cD)$
% such that
% $\Forget_\cC  \circ \widetilde{\cF}^{\scriptscriptstyle Z} = i \circ \cF^{\scriptscriptstyle Z}$, where
% $\Forget_\cC : Z(\cD) \to Z_\cC(\cD)$ is the forgetful functor to the relative Drinfeld centre for $\cC \subseteq \cD$ and $i : Z(\cC) \to Z_\cC(\cD)$ is the inclusion.
\item
Lifts $\widetilde{\uupi}: \uuG \to \uuBrPic^\cV(\cC)$ of the monoidal 2-functor $\uupi:\uuG \to
\uuBrPic(\cC)$ afforded by the $G$-extension $\cD$ (up to homotopy), where $\uuBrPic^\cV(\cC)$ is our newly introduced $\cV$-enriched Brauer-Picard 2-groupoid (see Definition \ref{defn:VBrauerPicardGroupoid} below).
%the core of 
%the endomorphism monoidal 2-category $\End^{123}(\cC: \cV \to \Vec)$ where $\cC$ is considered as a 1-morphism from $\cV$ to $\Vec$ in a certain
%4-category of braided fusion categories.
\item
Lifts $\widetilde{\cF}^{\scriptscriptstyle Z} : \cV \to Z(\cC)^G$ such that $\Forget_G\circ
\widetilde{\cF}^{\scriptscriptstyle Z} = \cF^{\scriptscriptstyle Z}$ where the categorical
$G$-action $\urho:\uG \to \uAut_\otimes^\br(Z(\cC))$ comes from the $G$-extension $\cC \subseteq
\cD$ and $\Forget_G: Z(\cC)^G \to Z(\cC)$ forgets the $G$-equivariant structure.
% \item
% Lifts $\widetilde{\urho}: \uG \to \uAut(Z(\cC)|\cF^{\scriptscriptstyle Z})$ such that $\Forget_\cF \circ \widetilde{\urho}= \urho$, where
% $ \uAut(Z(\cC)|\cF^{\scriptscriptstyle Z})$ is the categorical group consisting of pairs $(\alpha, \lambda^\alpha)$ with $\alpha \in \uAut(Z(\cC))$ and $\lambda^\alpha : \cF^\sZ \Rightarrow \alpha \circ \cF^\sZ$ is a monoidal natural isomorphism, and $\Forget_\cF: \uAut(Z(\cC)|\cF^{\scriptscriptstyle Z})\to \uAut(Z(\cC))$ forgets $\lambda^\alpha$.
\end{itemize}
\end{thm}

This theorem characterises the possible enrichments $\widetilde{\cF}^\sZ: \cV \to Z(\cD)$ of $(\cC, \cF^\sZ)$ which are compatible with the
fixed $G$-graded extension $\cC \subseteq \cD$. The proof uses extension theory for fusion categories of \cite{MR2677836} together with the results of \cite{MR2587410}.

Observe that the structures listed in Theorem \ref{thm:FirstMainTheorem} are more naturally viewed as collections of objects in higher groupoids rather than sets, and it would be more natural to prove an equivalence of groupoids rather than construct a bijection between these sets.
However, each of these higher groupoids is in fact 0-\emph{truncated}, i.e., equivalent to a $0$-groupoid, which is a set.
We make this rigorous by showing homotopy fibers of certain forgetful functors are 0-truncated and equivalent to strict fibers.
We discuss these notions in detail in \S\ref{sec:TruncationOfHomotopyFibers} on homotopy fibers of forgetful functors.

Thus one of our canonical bijections is typically constructed as a composite of bijections
$$
\left\{
\parbox{.8cm}{set 1}
\right\}
\cong
\left\{
\parbox{2cm}{strict fiber 1}
\right\}
\cong
\tau_0
\left\{
\parbox{2.9cm}{homotopy fiber 1}
\right\}
\cong
\tau_0
\left\{
\parbox{2.9cm}{homotopy fiber 2}
\right\}
\cong
\left\{
\parbox{2cm}{strict fiber 2}
\right\}
\cong
\left\{
\parbox{.8cm}{set 2}
\right\}
$$
where $\tau_0$ denotes taking the $0$-truncation.
%of objects in the categories which moreover restricts to a bijection of strict homotopy fibers of the cores.
This strategy is also employed to construct the canonical bijections asserted in Theorem \ref{thm:ClassificationOfGCrossedBraidings}, 
%Theorem \ref{thm:VMonoidalCategoreyClassification}, 
Corollary \ref{cor:fiberedenrichments}, 
Theorem \ref{thm:LiftsToEnrichedBrPic}, 
Theorem \ref{thm:ClassificationOfVEnrichmentsAsEquivariantLifts}, 
Theorem \ref{thm:ClassificationOfLiftingsViaSplittings}, 
and 
Theorem \ref{thm:ClassificationOfGCrossedBraidingsUsingBrPicC}.
We would like to emphasize that these results prove equivalences of cores of various higher categories, which happen to be 0-truncated, by providing a bijection on the 0-truncations. 
It would be interesting to see if some of these canonical bijections could be lifted to functorial constructions on the cores, or even on the higher categories.

The third description of compatible enrichments in Theorem \ref{thm:FirstMainTheorem} bears many similarities to the classification
from \cite{MR3933137} of $G$-equivariant structures on a connected \'{e}tale algebra in a
nondegenerately braided fusion category. Adapting the arguments and techniques from
\cite{MR3933137}, we see that there are two \emph{obstructions} to lifting our $\cV$-enrichment.
First, for every $g\in G$, we must have that $\cF^\sZ\cong g\circ \cF^\sZ$ as monoidal functors
$\cV\to Z(\cC)$. We call the existence of such monoidal natural isomorphisms for $g\in G$ the \emph{first
obstruction} to the equivariant functor lifting problem. When such monoidal natural isomorphisms
exist, we say $\cD$ \emph{passes} the first obstruction, or that the first obstruction
\emph{vanishes}. In this case, similar to \cite{MR3933137}, we show that lifts $\widetilde{\urho}:
\uG \to \uAut(Z(\cC)|\cF^{\scriptscriptstyle Z})$ correspond to splittings of a certain exact
sequence.

\begin{thm}
There is a short exact sequence
\begin{equation}
\label{eq:SES intro}
\begin{tikzcd}
1
\ar[r]
&
\Aut_{\otimes}(\cF^{\scriptscriptstyle Z})
\ar[r]
&
\Aut_{\otimes}(I\circ \cF^{\scriptscriptstyle Z})
\ar[r]
&
G
\ar[r]
&
1
\end{tikzcd}
\end{equation}
where $I: Z(\cC) \to Z(\cC)^G$ is the induction functor adjoint to the forgetful functor $\Forget_G$.\footnote{Observe that while $I$ is only oplax monoidal as an adjoint of a monoidal functor, it still makes sense to talk about the (oplax) monoidal automorphisms $\Aut_\otimes(I\circ \cF^\sZ)$.}
Moreover, splittings of this exact sequence are in canonical bijection with lifts
$\widetilde{\urho}: \uG \to \uAut(Z(\cC)|\cF^{\scriptscriptstyle Z})$
as in the final case of Theorem \ref{thm:FirstMainTheorem}.
\end{thm}

We call the exact sequence \eqref{eq:SES intro} the \emph{second obstruction} to the equivariant
functor lifting problem. We say the second obstruction \emph{vanishes} when this short exact
sequence splits, and a splitting is a \emph{witness} of the vanishing of the second obstruction. In
\S\ref{sec:Examples}, we calculate the splittings of \eqref{eq:SES intro} for various examples.

In \S\ref{sec:GCrossedBraidings}, we give an application of our two main theorems above to extend
Nikshych's classification \cite{MR3943750} of braidings on a fixed fusion category, classifying
$G$-crossed braidings on a fixed $G$-graded fusion category in Theorem \ref{Gcrossedthm}. The main
tool is the following theorem, which extends \cite[Prop.~2.4]{1803.04949} in the unitary setting.

\begin{thm}
\label{thm:ClassificationOfGCrossedBraidings}
Let $\cV$ be a braided fusion category, and $\cC$ a $G$-graded extension of $\cV$ as fusion
categories. The set of extensions of the self enrichment $\cV \to Z(\cV)$ to $Z(\cC)$ characterized
in Theorem \ref{thm:FirstMainTheorem} is in bijective correspondence with equivalence classes of $G$-crossed braidings on
$\cC$.
\end{thm}

We then describe the equivalence classes of $G$-crossed braidings on group theoretical $G$-graded fusion categories,
e.g.,  $\Vec(H, \omega)$ and $\Rep(H)$ for appropriate groups $H$, in terms of group theoretical
data.

%%%%%%%%%%%%%%%%%%%%%%%%%%%%%%%%%%%%%%%%%%%%%%%%
\subsection*{Acknowledgements}

This project started at the 2018 BIRS workshop on Fusion Categories and Subfactors.
The authors would like to thank Theo Johnson-Freyd, Dmitri Nikshych, Noah Snyder, and Kevin Walker for helpful conversations.
Corey Jones was supported by NSF DMS grant 1901082/2100531,
Scott Morrison was supported by a Discovery Project 
`Low dimensional categories' DP160103479, and a Future Fellowship `Quantum symmetries' FT170100019
from the Australian Research Council,
David Penneys was supported by NSF DMS grant 1654159, and
Julia Plavnik was supported by NSF DMS grants 1802503 and 1917319.

%%%%%%%%%%%%%%%%%%%%%%%%%%%%%%%%%%%%%%%%%%%%%%%%%%%%%%%%%%%%%%%%%%%%%%%%%%%%%%%%%%%%%%%%%%%%%%%%%%%%%%%%%%%%%%%%%%%%%%%%%%%%%%%%%%%%%%%%%%%%%%%%%%%%%%%%%%%%%%%%%%%%%%%%%%%%%%%%%%%%%%%%%%%%%%%%%%%%%%%%%%%%%%%%%%%%%%%%%%%%%%
\textbf{}\textbf{}\section{Background}

In this article, we assume the reader is familiar with tensor categories, in particular the book
\cite{MR3242743}. We typically use their conventions. For example, the Drinfeld center of a tensor
category $\cC$ has objects $(c,\sigma_{\bullet,c})$ where $c\in\cC$ and
$\sigma_{\bullet,c}=\{\sigma_{a,c}: a\otimes c \to c\otimes a\}_{a\in \cC}$ is a family of
half-braidings. In this convention, the braiding on $Z(\cC)$ is given by
$\beta_{(c,\sigma_{\bullet,c}),(d,\tau_{\bullet,d})}:= \tau_{c,d}: c\otimes d\to d\otimes c$.
When $\cC$ is a monoidal subcategory of a monoidal category $\cD$, we use the notation $Z_\cC(\cD)$
for the relative Drinfeld center. This agrees with the  notation of \cite{MR2587410}, but is the
reverse of the notation of \cite{MR3663592}.

%%%%%%%%%%%%%%%%%%%%%%%%%%%%%%%%%%%%%%%%%%%%%%%%%%%%%%
\subsection{Braided enriched monoidal categories}
\label{sec:BraidedEnrichedMonoidalCategories}

Recall from \cite{MR2177301} that given a monoidal category $\cV$, a $\cV$-category $\cC$ has
objects together with hom objects $\cC(a\to b)\in \cV$ for all $a,b\in\cC$. For every $a,b,c\in
\cC$, we have a composition morphism $-\circ_\cC-\in \cV(\cC(a\to b)\cC(b\to c)\to \cC(a\to c))$
which satisfies an associativity axiom. For every $a\in \cC$, we have an identity element
$j_a \in \cV(1_\cV \to \cC(a\to a))$ which satisfies a unitality axiom.

There are also notions of $\cV$-functors and ($1_\cV$-graded) $\cV$-natural transformations.
We refer the reader to \cite{MR2177301} for more details.
(See also the pedestrian exposition in \cite[\S2]{MR3961709} or \cite[\S2]{1809.09782}.)

\begin{defn}
Given a $\cV$-category $\cC$, the \emph{underlying category} $\cC^\cV$ has the same objects as
$\cC$, and the hom-sets are given by $\cC^\cV(a\to b):= \cV(1_\cV \to \cC(a\to b))$. We leave the
reader to work out the definitions of composition and identity morphisms for $\cC^\cV$.
\end{defn}

\begin{defn}[\cite{MR649797,1809.09782}]
A $\cV$-category $\cC$ is called \emph{weakly tensored} if
every representable functor $\cC(a\to -): \cC^\cV \to \cV$ admits a left adjoint.

When $\cV$ is closed, we can form the self-enrichment $\widehat{\cV}$ of $\cV$ over itself
\cite[\S1.6]{MR2177301}. 
In this case, the representable functor $\cC(a\to -): \cC^\cV \to \cV$ can be promoted to a $\cV$-functor $\widehat{\cC}(a\to -):\cC \to \cV$.
A $\cV$-category $\cC$ is called \emph{tensored} if every
$\cV$-representable functor $\widehat{\cC}(a\to -): \cC \to \widehat{\cV}$ admits a left $\cV$-adjoint.
\end{defn}

We now assume $\cV$ is a braided monoidal category.

\begin{defn}
A (strict) $\cV$-\emph{monoidal category} is a $\cV$-category $\cC$ equipped with an associative
monoid structure on objects, denoted $ab$ for $a,b\in \cC$, whose unit object is denoted by $1_\cC$,
together with a tensor product morphism $-\otimes_\cC- \in \cV(\cC(a\to c)\cC(b\to d)\to \cC(ab\to
cd))$ for all $a,b,c,d\in \cC$ satisfying strict associativity and unitality axioms. The tensor
product and composition morphisms must further satisfy the \emph{braided interchange} relation
\begin{equation*}
\label{eq:BraidedInterchance}
\begin{tikzcd}
\cC(a\to b) \cC(d\to e)\cC(b\to c) \cC(e\to f)
\ar[rr, "(-\otimes_\cC-)(-\otimes_\cC-)"]
\ar[dd, "\id\beta_{\cC(d\to e),\cC(b\to c)}\id"]
&&
\cC(ad\to be)\cC(be\to cf)
\ar[d, "-\circ_\cC-"]
\\
&&
\cC(ad\to cf)
\\
\cC(a\to b) \cC(b\to c) \cC(d\to e)\cC(e\to f)
\ar[rr, "(-\circ_\cC-)(-\circ_\cC-)"]
&&
\cC(a\to c) \cC(d\to f)
\ar{u}[swap]{-\otimes_\cC-}
\end{tikzcd}
\end{equation*}
There are also notions of $\cV$-monoidal functors and ($1_\cV$-graded) $\cV$-monoidal natural
transformations. We refer the reader to \cite[\S2]{MR3961709} or \cite[\S6.1]{1809.09782} for more
details.
\end{defn}

\begin{defn}
A $\cV$-monoidal category $\cC$ is called 
\emph{rigid}
if its underlying monoidal category is rigid.
\end{defn}

\begin{remark}
When a $\cV$-monoidal category $\cC$ is rigid, $\cC$ is weakly tensored if and only if
$\cC(1_\cC\to -): \cC^\cV \to \cV$ admits a left adjoint \cite[Lem.~6.8]{1809.09782}. 
When in
addition $\cV$ is rigid (which implies $\cV$ is closed),
$\cC$ is tensored if and only if the $\cV$-functor $\widehat{\cC}(1_\cC \to -)$
admits a left $\cV$-adjoint \cite[Cor.~7.3]{1809.09782}.
\end{remark}

In \cite{MR3961709}, we proved a classification theorem for (weakly) tensored rigid $\cV$-monoidal
categories in terms of $\cV$-module tensor categories \cite{MR3578212}. 
The tensored case was
treated in \cite{1809.09782}.
We now restrict our focus to the tensored case for ease of exposition, as all our examples in this article are tensored.

\begin{defn}
\label{defn:VModuleTensorCategories}
A $\cV$-\emph{module tensor category} consists of a pair $(\cT,
\cF^{\scriptscriptstyle Z})$ with $\cT$ a monoidal category and $\cF^{\scriptscriptstyle Z}: \cV \to
Z(\cT)$ a braided strongly monoidal functor. 
We call a $\cV$-module tensor category:
\begin{itemize}
\item
\emph{rigid} if $\cT$ is rigid, and
%\item
%\emph{weakly tensored} if $\cF^{\scriptscriptstyle Z}$ is oplax braided strongly unital,
%and
%$\cF:=\Forget_Z\circ \cF^{\scriptscriptstyle Z}$
%admits a right adjoint, where $\Forget_Z: Z(\cT)\to \cT$ is the forgetful functor.
\item
\emph{tensored} if the strong monoidal functor $\cF:=\Forget_Z\circ \cF^{\scriptscriptstyle Z}$ admits a right adjoint.
\end{itemize}
Based on \cite[Def.~3.2]{1607.06041}, 
a 1-morphism $(\cT_1, \cF_1^\sZ)\to (\cT_2,\cF^\sZ_2)$ of $\cV$-module tensor categories consists of a pair $(G,\gamma)$ where $G:\cT_1\to \cT_2$ is a strong monoidal functor and $\gamma: \cF_2 \Rightarrow G\circ \cF_1$ is an \emph{action coherence} monoidal natural isomorphism which satisfies the following compatibility with half-braidings:
\begin{equation}
\label{eq:CoherenceForVModuleTensorCategoryHalfBraidings}
\begin{tikzcd}
G(t)\otimes \cF_2(v)
\arrow[r,"\id\otimes\gamma_v"]
\arrow[d,"\sigma_{G(t),\cF_2(v)}"]
&
G(t)\otimes G(\cF_1(v))
\arrow[r,"\cong"]
&
G(t\otimes \cF_1(v))
\arrow[d,"G(\sigma_{t,\cF_2(v)})"]
\\
\cF_2(v)\otimes G(t)
\arrow[r,"\gamma_v\otimes \id"]
&
G(\cF_1(v))\otimes G(t)
\arrow[r,"\cong"]
&
G(\cF_1(v)\otimes t)
\end{tikzcd}
.
\end{equation}
Based on \cite[Def.~3.3]{1607.06041}, a 2-morphism $\kappa: (G,\gamma) \Rightarrow (H,\eta)$ between 1-morphisms $(\cT_1, \cF_1^\sZ)\to (\cT_2,\cF^\sZ_2)$ is 
a monoidal natural transformation $\kappa: G\Rightarrow H$ such that for all $v\in \cV$, the following diagram commutes:
\begin{equation}
\label{eq:CoherenceForVModuleTensorCategory2Morphisms}
\begin{tikzcd}
\cF_2(v)
\arrow[rr,"\gamma_v"]
\arrow[dr,"\eta_v"]
&&
G(\cF_1(v))
\arrow[dl,"\kappa_{\cF_1(v)}"]
\\
&H(\cF_1(v)).
\end{tikzcd}
\end{equation}
%Based on \cite[Def.~3.2]{1607.06041} and \cite[Def.~7.1]{MR3961709}, one can define the notion of an
%equivalence for $\cV$-module tensor categories.
We call an invertible 1-morphism between $\cV$-module tensor categories an equivalence.
%An equivalence of $\cV$-module tensor categories is an invertible 1-morphism
\end{defn}

We have the following classification theorem, which has recently been extended to a 2-equivalence of
2-categories (pseudofunctor equivalence of bicategories) in \cite{Dell19}.

\begin{thm}[\cite{MR3961709,1809.09782}]
\label{thm:VMonoidalCategoreyClassification}
Let $\cV$ be a braided monoidal category.
There is a bijective correspondence between equivalence classes
\[
\left\{\,
\parbox{4.7cm}{\rm Tensored rigid $\cV$-monoidal categories $\cC$}\,\right\}
\,\,\cong\,\,
\left\{\,\parbox{6cm}{\rm Tensored rigid $\cV$-module tensor categories $(\cT,\cF^{\scriptscriptstyle Z})$}\,\right\}.
\]
\end{thm}

In light of Theorem \ref{thm:VMonoidalCategoreyClassification}
together with the results of \cite{1905.04924,MR3763324} in the fusion setting,
 we make the following definition.

\begin{defn}
A $\cV$-\emph{fusion category}, for $\cV$ a braided fusion category, consists of a fusion category $\cC$ together with a braided strong monoidal functor $\cF^{\scriptscriptstyle Z}: \cV \to Z(\cC)$.
Observe as $\cF^{\scriptscriptstyle Z}$ is a functor between fusion categories, it automatically admits a left adjoint, and hence $\cV$-fusion categories are tensored.
\end{defn}

We focus on the fusion setting in order to have access to the results of \cite{MR2677836} and \cite{MR2587410}.

%%%%%%%%%%%%%%%%%%%%%%%%%%%%%%%%%%%%%%%%%%%%%%%%%%%%%
\subsection{Extension theory for fusion categories}

We rapidly review the results of \cite{MR2677836} and \cite{MR2587410} on extension theory for fusion categories.

In \cite{MR2677836}, Etingof-Nikshych-Ostrik give a recipe for constructing $G$-extensions of a
fixed fusion category $\cC$ using cohomological obstruction theory.

\begin{defn}
Recall that a \emph{categorical $n$-group} is an $(n+1)$-category with one 0-morphism such that
every k-morphism is invertible up to a (k+1)-isomorphism for $k\leq n$, and all $(n+1)$-morphisms
are invertible. 
Typically, we indicate the categorical group number by adding that number of
underlines below. 
We denote the $k<n$ \emph{truncation} obtained by inductively identifying higher
isomorphism classes by simply removing underlines.
\end{defn}

\begin{example}
Given a group $G$, we view it as a category with one object where every morphism is invertible.
We get a categorical 1-group $\uG$ by adding only identity 2-morphisms, and we get a categorical 2-group $\uuG$ by only adding identity 2-morphisms to $\uG$.
\end{example}

\begin{example}
Given a fixed fusion category $\cC$, the 2-groupoid 
$\Ext(G,\cC)$
of $G$-extensions of $\cC$ is the categorical 2-group whose
unique object is $\cC$,
whose 1-morphisms are $G$-graded fusion categories $\cD=\bigoplus_{g\in G} \cD_g$ together with a fixed monoidal equivalence $I_\cD: \cC \to \cD_e$,
whose 2-morphisms are $G$-graded monoidal equivalences $F:\cD \to \cE$ together with a monoidal natural isomorphism $\alpha : I_\cE \Rightarrow F\circ I_\cD$,
and 
whose 2-morphisms 
$(F_1,\alpha_1)\Rightarrow (F_2,\alpha_2)$
are monoidal natural equivalences
$\gamma: F_1 \Rightarrow F_2$ such that
$$
\begin{tikzcd}
\cC\arrow[rr, "I_\cD"]
\arrow[ddrr, swap, "I_\cE"]
&&
\cD\arrow[dd,"F_2"]
\arrow[dl,Leftarrow,shorten <= 1em, shorten >= 1em, "\alpha_2"]
%    \arrow[d,bend left = 60, "B"]
\\
&\mbox{}&
\\
&&
\cE.
\end{tikzcd}
\qquad
=
\qquad
\begin{tikzcd}
\cC\arrow[rr, "I_\cD"]
\arrow[ddrr, swap, "I_\cE"]
&&
\cD\arrow[dd,swap,"F_1"]
\arrow[dl,Leftarrow,shorten <= 1em, shorten >= 1em, "\alpha_1"]
\arrow[dd,bend left = 90, "F_2"]
\\
&\mbox{}&
\arrow[r,Rightarrow,shorten >= 1em, "\!\!\!\!\gamma"]
&\mbox{}
\\
&&
\cE
\end{tikzcd}
\,.
$$
\end{example}

\begin{remark}
The 2-groupoid $\Ext(G,\cC)$ defined in the above example is equivalent to the one defined in \cite[Def.~8.2]{2006.08022}
where each $\cI_\cD: \cC \to \cD_e$ is $\id_\cC$ and each $\alpha: \id_\cC \Rightarrow F|_\cC\circ \id_\cC$ is the identity monoidal natural isomorphism.
\end{remark}

\begin{example}
\label{ex:BrPic}
Given a fixed fusion category $\cC$, its \emph{Brauer-Picard} groupoid
$\uuBrPic(\cC)$ is the categorical 2-group whose unique 0-morphism is $\cC$,
whose 1-morphisms are invertible $\cC-\cC$ bimodule categories,
whose 2-morphisms are $\cC-\cC$ bimodule equivalences,
and whose 3-morphisms are bimodule functor natural isomorphisms.
\end{example}

\begin{defn}
\label{defn:RelativeDeligneProduct}
In Example \ref{ex:BrPic} above, composition of $\cC-\cC$ bimodule categories is the relative Deligne tensor product.
In more detail, suppose $\cD$ is a fusion category,
$\cM_\cD$ is a right $\cD$-module category, and ${}_\cD\cN$ is a left $\cD$-module category.
The relative tensor product is a finitely semisimple category $\cM\boxtimes_\cD \cN$ together with a $\cD$-balanced functor 
$\boxtimes_\cD: \cM\boxtimes \cN \to \cM\boxtimes_\cD \cN$ satisfying the universal property 
that for every abelian category $\cP$ and any $\cD$-balanced functor $F: \cM\boxtimes \cN \to \cP$,
there exists a linear functor $\widetilde{F}: \cM \boxtimes_\cD \cN \to \cP$, unique up to unique natural isomorphism, such that the following diagram weakly commutes:
$$
\begin{tikzcd}
\cM \boxtimes \cN
\arrow[d,"\boxtimes_\cD"]
\arrow[dr,"F"]
\\
\cM\boxtimes_\cD \cN
\arrow[r,"\widetilde{F}"]
&
\cP
\end{tikzcd}
$$
When $\cM$ is a left $\cC$-module category and $\cN$ is a right $\cE$-module category, then $\cM\boxtimes_\cD \cN$ inherits the structure of a $\cC-\cE$ bimodule category.
We refer the reader to \cite[\S3]{MR2677836} for more details.
\end{defn}

The following theorem classifies $G$-extensions via monoidal 2-functors $\uuG \to \uuBrPic(\cC)$.
For the definition of a monoidal 2-functor, see \cite[Def.~2.8]{2006.08022}.

\begin{thm}[{\cite{MR2677836} and \cite[Thm.~8.5]{2006.08022}}]
\label{thm:ExtensionsAsMonoidal2Functors}
Let $\cC$ be a fusion category.
There is an equivalence of 2-groupoids
$\Ext(G,\cC) \cong \Hom(\uuG \to \uuBrPic(\cC))$.
%Equivalence classes of $G$-extensions $\cD$ of the fusion category $\cC$
%are in bijective correspondence with
%equivalence classes of monoidal 2-functors
%$\uuG \to \uuBrPic(\cC)$, or equivalently,
%with homotopy classes of maps between their classifying spaces:
%$BG \to B\uuBrPic(\cC)$.
\end{thm}

The main tool of \cite{MR2677836} gives a cohomological prescription for constructing $G$-graded extensions by lifting a group homomorphism, or \emph{symmetry action}, $\rho: G\rightarrow \BrPic(\cC)$ to $\uuG \to \uuBrPic(\cC)$.
We can lift $\rho$ to a categorical action
$\underline{\rho}:\underline{G}\rightarrow \uBrPic(\cC)$
if and only if the obstruction
$o_{3}(\rho)\in H^{3}(G, \Inv(Z(\cC)))$ vanishes.
In this case, the set of equivalence classes of liftings form a torsor over $ H^{2}(G, \Inv(Z(\cC)))$.
Given $\underline{\rho}: \underline{G}\rightarrow \uBrPic(\cC)$,
there is a lift
$\uurho:\uuG\rightarrow \uuBrPic(\cC)$
if and only if the obstruction $o_{4}(\underline{\rho})\in H^{4}(G, \mathbb{C}^{\times})$ vanishes.
In this case, the equivalence classes of liftings form a torsor over $H^{3}(G, \mathbb{C}^{\times})$.

We now recall the main results of \cite{MR2587410}. Suppose we have a $G$-extension $\cD =
\bigoplus_{g\in G} \cD_g$ of $\cC$. (Note that the convention $\cC \subseteq \cD$ is opposite to the
convention of \cite{MR2587410} which uses $\cD \subseteq \cC$.) The relative center $Z_\cC(\cD)$ is
canonically a $G$-\emph{crossed braided extension} \cite[\S8.24]{MR3242743} of $Z(\cC)$ whose
$G$-equivariantization \cite[\S4.15]{MR3242743} is equivalent to $Z(\cD)$. Moreover, the canonical
equivalence $Z(\cD) \cong Z_\cC(\cD)^G$ intertwines both forgetful functors to $Z_\cC(\cD)$, and
maps $\Rep(G)' \subset Z(\cD)$ to $Z(\cC)^G$ up to a canonical monoidal natural isomorphism.
\begin{equation}
\label{eq:RelativeCenterEquivariantization}
\begin{tikzcd}
Z(\cC)^G
\ar[d, hook]
\ar[rr, leftrightarrow, "\cong"]
&&
\Rep(G)' 
\ar[d, hook]
\\
Z_\cC(\cD)^G
\ar[rr, leftrightarrow, "\cong"]
\ar[dr, swap, "\Forget_\cC"]
&
\ar[d, phantom, near start, "\,\,\,\curvearrowright"]
&
Z(\cD)
\ar[dl,"\Forget_G"]
\\
&
Z_\cC(\cD)
\end{tikzcd}
\end{equation}

%%%%%%%%%%%%%%%%%%%%%%%%%%%%%%%%%%%%%%%%%%%%%%%%%%%%%%
\subsection{The 4-category of braided tensor categories}
\label{sec:4CatOfBraidedTensorCategories}

By \cite{MR3650080,MR3590516}, there is a 4-category of braided tensor categories
${\mathsf{BrTens}}$, and the sub-4-category $\BrFus$ of braided fusion categories is 4-dualizable by \cite[Thm.~1.19]{MR4228258}.

%\begin{defn}
Following \cite{MR4228258}, we now describe the $n$-morphisms and the composition operations of the 4-category $\BrFus$.
\begin{itemize}
\item
0-morphisms are braided fusion categories.
\item
1-morphisms $\BrFus_{1}(\cA\to \cB)$ are multifusion categories $\cC$ together with a braided monoidal functor $F_\cC:\cA\boxtimes \cB^{\rev}\rightarrow Z(\cC)$ called a \emph{central structure}. 
Sometimes we denote $\cC\in \BrFus_{1}(\cA\to \cB)$ by $_{\cA} \cC_{\cB}$.

The composite of ${}_{\cA_1} \cC_{\cA_2}$ and ${}_{\cA_2} \cD {}_{\cA_3}$ is defined as follows.
First, we look at the Deligne tensor product $\cC \boxtimes \cD$, which comes equipped with a
braided monoidal functor $F: \cA_2^{\rev} \boxtimes \cA_2 \to Z(\cC\boxtimes \cD)$. We define $\cC
\boxtimes_{\cA_2} \cD$ to be  $(\cC\boxtimes \cD)_L$, the category of left $L$-modules in
$\cC\boxtimes \cD$, where $L \in \cA_2^{\rev} \boxtimes \cA_2$ is the commutative algebra obtained by taking $I(1_{\cA_2})$, where $I$ is the left adjoint to the canonical tensor product functor $\otimes: \cA^{\rev}_{2}\boxtimes \cA_{2}\rightarrow \cA_{2}$, given by $\otimes (a\boxtimes b):=a\otimes b$ and using the braiding for the tensorator. This algebra is commutative since $\otimes$ is a central functor \cite[Lemma~3.5]{MR3039775}. If $\cA_{2}$ is nondegenerate, this algebra is identified with the canonical Lagrangian algebra under the standard equivalence $\cA^{\rev}_{2}\boxtimes \cA_{2}\cong Z(\cA_{2})$. To see that $\cC \boxtimes_{\cA_2} \cD$ has the structure of a 1-morphism
in $\BrFus_1(\cA_1 \to \cA_3)$, we observe that $Z((\cC\boxtimes \cD)_L) \cong Z(\cC \boxtimes
\cD)_L^{\loc}$, the $L$-local modules in $Z(\cC \boxtimes \cD)\cong Z(\cC)\boxtimes Z(\cD)$ by
\cite[Thm.~3.20]{MR3039775}. Since $\cA_1$ centralizes $F_{\cA_2^{\rev}}(\cA_2^{\rev})\boxtimes
Z(\cD)$ and $\cA_3^{\rev}$ centralizes $Z(\cC)\boxtimes F_{\cA_2}(\cA_2)$ in $Z(\cC)\boxtimes
Z(\cD)$, we get a braided monoidal functor $\cA_1 \boxtimes \cA_3^{\rev} \to Z(\cC \boxtimes
\cD)_L^{\loc}\cong Z((\cC\boxtimes \cD)_L)$. 

An explicit example calculation of the composite
$\Ad E_8 \boxtimes_\Fib \Ad E_8'$ appears in \cite{mitchell-thesis}.

\item
2-morphisms $\BrFus_{2}(\mathcal{C}, \mathcal{D})$ are finitely semisimple $\cC-\cD$ bimdodule
categories $\mathcal{M}$ together with natural isomorphisms $\eta_{a, m}: m\triangleleft
F_{\mathcal{D}}(a) \to  F_{\cC}(a)\triangleright m$ for $a\in \cA\boxtimes \cB^{\rev}$ and $m \in
\cM$ 
called a $\cA\boxtimes \cB^{\rev}$-\emph{centered structure} 
such that the following diagrams commute (here we suppress names of arrows):
\begin{equation}
\label{eq:BimoduleFunctorEquivalence1}
\begin{tikzcd}
  & F_{\cC}(a)\triangleright (c\triangleright m )  \arrow{r} & (c\triangleright m)\triangleleft F_{\mathcal{D}}(a) \arrow{dr}  &
\\
   (F_{\cC}(a)\otimes c)\triangleright m \arrow{dr} \arrow{ur}&  & & c\triangleright (m\triangleleft F_{\mathcal{D}}(a))
\\
   & (c\otimes F_{\cC}(a))\triangleright m \arrow{r} & c \triangleright (F_{\cC}(a)\triangleright m) \arrow{ur}  &
  \end{tikzcd}
\end{equation}

\begin{equation}
\label{eq:BimoduleFunctorEquivalence2}
\begin{tikzcd}
  & F_{\cC}(a)\triangleright (m\triangleleft d )  \arrow{r} & (m\triangleleft d)\triangleleft F_{\mathcal{D}}(a) \arrow{dr}  &
\\
   (F_{\cC}(a)\triangleright m)\triangleleft d \arrow{dr} \arrow{ur}&  & & m\triangleleft (d\otimes F_{\mathcal{D}}(a))
\\
   & (m\triangleleft F_{\mathcal{D}}(a))\triangleleft d \arrow{r} & m \triangleleft (F_{\mathcal{D}}(a) \otimes d) \arrow{ur}  &
  \end{tikzcd}
\end{equation}

\begin{equation}
\label{eq:BimoduleFunctorEquivalence3}
\begin{tikzcd}
F_{\cC}(a\otimes b)\triangleright m \arrow{d} \arrow{r}  
& 
m\triangleleft F_{\mathcal{D}}(a\otimes b) \arrow{r} 
& 
m\triangleleft (F_{\mathcal{D}}(a)\otimes F_{\mathcal{D}}(b)) 
\arrow{d}  
\\
(F_{\cC}(a)\otimes F_{\cC}(b))\triangleright m  
\ar{d}
&& 
(m\triangleleft F_{\mathcal{D}}(a))\triangleleft F_{\mathcal{D}}(b) 
\arrow{d}
\\
F_{\cC}(a)\triangleright (F_{\cC}(b)\triangleright m) 
\arrow{r}  
& 
F_{\cC}(a)\triangleright(m\triangleleft F_{\mathcal{D}}(b))
\arrow{r}
&
(F_{\cC}(a)\triangleright m)\triangleleft F_{\mathcal{D}}(b)
 \end{tikzcd}
\end{equation}

The definitions of horizontal and vertical composition of 2-morphisms are given in
\cite[p.~41-42]{MR4228258}. 
For our purposes, we need to know that vertical composition is the relative Deligne tensor product ${}_\cC\cM\boxtimes_\cD \cN_\cE$ discussed earlier in Definition \ref{defn:RelativeDeligneProduct}.
As described in \cite[Def.~Prop.~3.13]{MR4228258},
when $\cC,\cD,\cE$ are equipped with central structures $F_\cC,F_\cD,F_\cE$ respectively
and
$\cM, \cN$ are equipped with $\cA\boxtimes \cB^{\rev}$-centered structures $\eta^\cN, \eta^\cM$ satisfying \eqref{eq:BimoduleFunctorEquivalence1},
\eqref{eq:BimoduleFunctorEquivalence2},
\eqref{eq:BimoduleFunctorEquivalence3},
the $\cC-\cE$ bimodule category $\cM\boxtimes_\cD \cN$ is equipped with the $\cA\boxtimes \cB^{\rev}$-centered structure
% \nn{There are many models... use def.-prop. 3.13 in \cite{MR4228258}.}
% $\cM\boxtimes_\cD \cN = Z_\cD(\cM\boxtimes \cN)$ from \cite[Prop.~3.8]{MR2677836}. Here, we use the
% leg numbering convention as in Sweedler notation for the underlying objects in $\cM\boxtimes_\cD
% \cN$, i.e., $x_{(1)}\boxtimes x_{(2)}=\bigoplus_{i=1}^k m_i \boxtimes n_i$. We endow
% ${}_\cC\cM\boxtimes_\cD \cN_\cE$ with the isomorphisms $\eta$ which are given by the following
% composite:
\begin{equation}
\label{eq:VerticalCompositionEta}
m \boxtimes_\cD (n \vartriangleleft F_\cE(a))
\cong
m \boxtimes_\cD
(F_\cD(a) \vartriangleright n)
\cong
(m \vartriangleleft F_\cD(a))
\boxtimes_\cD n
\cong
(F_\cC(a) \vartriangleright m)
\boxtimes_\cD n.
\end{equation}

\item
Let $\mathcal{M}$ and $\mathcal{N}$ be two 2-morphisms with source $\cC$ and target $\mathcal{D}$.
Then a 3-morphism is a bimodule functor $G: \mathcal{M}\rightarrow \mathcal{N}$ such that the
following diagram commutes:

\begin{equation}
\label{eq:BimoduleFunctorProperty}
\begin{tikzcd}
    G(m\triangleleft F_{\mathcal{D}}(a)) \arrow{d}\arrow{r}{G(\eta_{a,m})}
    &
    G(F_{\cC}(a)\triangleright m)  \arrow{d}
\\
    G(m)\triangleleft F_{\mathcal{D}}(a)\arrow{r}{\eta_{a,G(m)}}
    &
    F_{\mathcal{C}}(a)\triangleright G(m)
 \end{tikzcd}
\end{equation}

 \item
 $4$ morphisms are bimodule natural transormations with no extra compatibility required!

\end{itemize}
%\end{defn}

\begin{remark}
Observe that we may consider a fusion category $\cC \in \BrFus_1(\Vec\to \Vec)$ where we suppress the obvious braided central functor $\cF^\sZ: \Vec \to Z(\cC)$.
Then $\uuBrPic(\cC)$ is exactly the core (consisting of only the invertible morphisms) of the endomorphism 3-category $\End^{123}(\cC)$ which has
\begin{itemize}
\item
a single 0-morphism $\cC$
\item
1-morphisms $\BrFus_2(\cC\to \cC)$
\item
2-morphisms the 3-morphisms in $\BrFus$,
and
\item
3-morphisms the 4-morphisms in $\BrFus$.
\end{itemize}
\end{remark}

\begin{remark}
Observe that given a $\cV \in \BrFus$, a 1-morphism $(\cC, \cF^{\sZ}) \in \BrFus_1(\cV \to \Vec)$ is exactly a $\cV$-fusion category.
\end{remark}

Recall that 
non-degenerate braided fusion categories $\cA,\cB$
are said to be \emph{Witt equivalent} \cite[Def.~5.1 and Rem.~5.2]{MR3039775} if there exist multifusion categories $\cC,\cD$ such that
$\cA \boxtimes Z(\cC) \cong \cB\boxtimes Z(\cD)$.
We conclude this section with the following observation.

\begin{thm}
\label{thm:WittEquivalence}
Suppose $\cA, \cB$ are non-degenerate braided fusion categories and $\cC\in \BrFus_1(\cA \to \cB)$.
The following statements are equivalent.
\begin{enumerate}
\item
$\cC$ is an invertible 1-morphism in $\BrFus$.
\item
$F_\cC : \cA \boxtimes \cB^{\rev} \to Z(\cC)$ is a braided equivalence.
%\item
%$\cC$ witnesses the Witt equivalence of $\cA$ and $\cB$.
\end{enumerate}
\end{thm}
Before proving the theorem, we observe that the existence of $\cC$ as in (2) above is equivalent to
the Witt equivalence of $\cA$ and $\cB$ by \cite[Rem.~5.2 and Cor.~5.8]{MR3039775}. 
\begin{proof}
%It suffices to prove (1) and (2) are equivalent, as the equivalence of (2) and (3) follows by 
Suppose $\cC$ is an invertible 1-morphism in $\BrFus(\cA \to \cB)$. First, since $\cA$ and
$\cB^{\rev}$ are non-degenerate, every braided tensor functor out of $\cA\boxtimes \cB^{\rev}$ is
fully faithful. Hence $Z(\cC) \cong \cA  \boxtimes \cD_1 \boxtimes \cB^{\rev}$ for some
non-degenerate braided fusion category $\cD_1$. Let $\cC^{-1}\in \BrFus(\cB \to \cA)$ be an inverse
for $\cC$ such that $\cA \cong (\cC\boxtimes \cC^{-1})_L$ as 1-morphisms in $\BrFus_1(\cA \to \cA)$,
where $L\in \cB^{\rev} \boxtimes \cB$ is the canonical Lagrangian algebra. By a similar argument as
before, $Z(\cC^{-1}) \cong \cB \boxtimes \cD_2 \boxtimes \cA^{\rev}$ for some non-degenerate braided
fusion category $\cD_2$. Observe now that
$$
Z(\cC\boxtimes \cC^{-1}) \cong Z(\cC)\boxtimes Z(\cC^{-1}) 
\cong 
\cA \boxtimes \cD_1 \boxtimes \cB^{\rev} \boxtimes \cB \boxtimes \cD_2 \boxtimes \cA^{\rev}.
$$
This means that by
$$
Z(\cC \boxtimes \cC^{-1})_L^{\loc}
\cong
\cA \boxtimes \cD_1 \boxtimes \cD_2 \boxtimes \cA^{\rev}.
$$
But since $Z(\cC \boxtimes \cC^{-1})_L^{\loc} \cong Z((\cC\boxtimes \cC^{-1})_L) \cong Z(\cA)\cong \cA \boxtimes \cA^{\rev}$ as $\cA$ is non-degenerate,
we must have $\cD_1$ and $\cD_2$ are trivial, and thus $Z(\cC) \cong \cA \boxtimes \cB^{\rev}$.

Conversely, if $Z(\cC) \cong \cA \boxtimes \cB^{\rev}$, then observe that $Z(\cC^{\mathrm{mp}})\cong
\cB \boxtimes \cA^{\rev}$, where $\cC^{\mathrm{mp}}$ is the monoidal opposite of $\cC$. For the  canonical Lagrangian algebra $L\in \cB^{\rev}\boxtimes \cB$,
$$
(\cC \boxtimes \cC^{\mathrm{mp}})_L 
\cong
(\cA \boxtimes \cB^{\rev} \boxtimes \cB \boxtimes \cA^{\rev})_L
\cong
\cA \boxtimes \cA^{\rev}
\cong
Z(\cA)
$$
and so $(\cC \boxtimes \cC^{\mathrm{mp}})_L \cong \cA$ as 1-morphisms in $\BrFus_1(\cA\to \cA)$.
Similarly, we have that $(\cC^{\mathrm{mp}}\boxtimes \cC)_{L'} \cong \cB$ as 1-morphisms in
$\BrFus_1(\cB \to \cB)$, where $L'$ is the canonical Lagrangian algebra in $\cA^{\rev} \boxtimes
\cA$.
\end{proof}

%\begin{proof}[Proof of Theorem \ref{thm:WittEquivalence}]
%Suppose $\cA,\cB$ are equivalent in $\BrFus$ via $\cC\in \BrFus_1(\cA \to \cB)$.
%By Lemma \ref{lem:Invertible1MorphismsBetweenNondegenerateBFCs}, 
%$\cA \boxtimes \cB^{\rev} \cong Z(\cC)$ so that
%$$
%\cA \boxtimes Z(\cB)
%\cong
%\cA \boxtimes \cB\boxtimes \cB^{\rev}
%\cong
%\cB \boxtimes \cA \boxtimes \cB^{\rev}
%\cong
%\cB\boxtimes Z(\cC).
%$$
%Conversely, suppose $\cA$ and $\cB$ are Witt equivalent.
%Then $\cA \boxtimes \cB^{\rev}$ is in the same Witt class as $\Vec$, so by \cite[Prop.~5.8]{MR3039775}, there exists a fusion category $\cC$ such that $\cA \boxtimes \cB^{\rev} \cong Z(\cC)$.
%Hence $\cC$ is an invertible 1-morphism in $\BrFus_1(\cA \to \cB)$ by Lemma \ref{lem:Invertible1MorphismsBetweenNondegenerateBFCs}.
%\end{proof}

%%%%%%%%%%%%%%%%%%%%%%%%%%%%%%%%%%%%%%%%%%%%%%%%%%%%%%
%%%%%%%%%%%%%%%%%%%%%%%%%%%%%%%%%%%%%%%%%%%%%%%%%%%%%%
%%%%%%%%%%%%%%%%%%%%%%%%%%%%%%%%%%%%%%%%%%%%%%%%%%%%%%
\section{Truncation of homotopy fibers and classification of \texorpdfstring{$G$}{G}-gradings on a fusion category}
\label{sec:TruncationOfHomotopyFibers}

In this article,
we will often discuss various notions which are somewhat evil from a categorical perspective, such as classifying lifts of a fixed functor or $G$-gradings on a fixed fusion category.
In this section, we discuss how to make these notion not evil by using the notion of truncation of a homotopy fiber.
In the cases we care most about, we can show that the homotopy fiber of a particular (2-)functor truncates to a set, and this set is in canonical bijection with a strict `set theoretic' fiber. 

As an example, in \S\ref{sec:G-gradings} below, we classify the \emph{set} of $G$-gradings on a fixed fusion category $\cT$ in terms of fully faithful $\Rep(G)$-fibered enrichments.

%%%%%%%%%%%%%%%%%%%%%%%%%%%%%%%%%%%%%%%%%%%%%%%%%%%%%%
\subsection{Homotopy fibers of forgetful functors}
\label{sec:HomotopyFibers}

In this section, we make sense of how various structures on a fixed monoidal category, like $G$-gradings for a fixed group $G$, or braidings, form a \emph{set}, and not a category.

Grothendieck's \emph{Homotopy Hypothesis} posits that homotopy $n$-types are equivalent to $n$-groupoids for all $n\in \bbN\cup\{\infty\}$ via the fundamental groupoid construction.
In this section, we use the term $n$-\emph{groupoid} as a synonym for \emph{homotopy $n$-type}, and \emph{weak $n$-functor} for \emph{a homotopy class of continuous maps}.

\begin{fact}
For weak $n$-categories, the homotopy hypothesis is known to hold for $n\leq 3$ \cite{MR2842929} to various degrees.
We will only use it for $n\leq 2$ for this article.
In more detail,
\begin{itemize}
\item
The strict 2-category of 
groupoids, functors, and natural transformations 
is equivalent to 
the 2-category of homotopy 1-types, continuous maps, and homotopy classes of homotopies.
\item
By \cite{MR1239560}, the homotopy category of strict 2-groupoids and strict 2-functors localized at the strict 2-equivalences is equivalent to the 1-category of homotopy 2-types and homotopy classes of continuous maps.
%From Algebraic classification of equivariant homotopy 2-types, I :
%Mac Lane and Whitehead proved
%that there is a similar equivalence between pointed connected CW-complexes
%(X, x,,) such that ri(X, .K,)) = 0 for i 2 3 (homotopy 2-types) and crossed modules
%(see [
%S. Mac Lane and J.H.C. Whitehead, On the 3-type of a complex, Proc. Nat. Acad. Sci. U.S.A. 36 (1950) 41-48. 
%J.H.C. Whitehead. Combinatorial homotopy II, Bull. Amer. Math. Sot. 55 (1949) 453-496. 
%]; cf. also Theorem 1.7 in [
%J.-L. Loday, Spaces with finitely many non-trivial homotopy groups, J. Pure Appl. Algebra 24 (1982) 179-202. 
%] for a more general result). 
\item
By \cite{MR2842929}, the homotopy category of $\mathsf{Gray}$-groupoids and $\mathsf{Gray}$-functors localized at the $\mathsf{Gray}$-equivalences is equivalent to the 1-category of homotopy 3-types and homotopy classes of continuous maps.
\end{itemize}
\end{fact}
%In this section, we use the term \emph{$n$-category} to mean weak $n$-category, and an \emph{$n$-functor} $F: \cC \to \cD$ between $n$-categories means a weak $n$-functor.
%In this section, we will be focused on the cases $n=2,3$.
%Thus \emph{2-category} means bicategory and \emph{2-functor} means pseudofunctor; \emph{3-category} means algebraic tricategory in the sense of \nn{}, and \emph{3-functor} means weak 3-functor in the sense of \nn{}.

\begin{defn}
Recall that the path space and homotopy fiber construction produces a fibration from any continuous map of spaces.
We now explain this in the language of $n$-groupoids.

Suppose $\cC, \cD$ are $n$-groupoids and $U: \cC \to \cD$ is an $n$-functor.
The \emph{path space} of $U$, denoted $\Path(U)$, has objects triples $(c,d, \psi)$ with $c\in \cC$, $d\in \cD$ and $\psi \in \cC(U(c)\to d)$ an isomorphism,
$(c_1,d_1,\psi_1)\to (c_2,d_2,\psi_2)$ are triples $(A,B,\alpha)$ with $A\in \cC(c_1\to c_2)$, $B\in \cD(d_1\to d_2)$, and $\alpha \in \cD(\psi_2\circ U(A)\Rightarrow B\circ \psi_1)$ is a 2-isomorphism,
and so forth, where $k$-morphisms consist of triples of a $k$-morphism in $\cC$, a $k$-morphism in $\cD$, and a $(k+1)$-isomorphism in $\cD$ compatible with lower structure.
Here, we interpret an $(n+1)$-isomorphism as an equality.

The \emph{homotopy fiber} of $U$ at $d\in \cD$, denoted $\hoFib_d(U)$, has objects
pairs $(c,\psi)$ with $c\in \cC$ and $\psi \in \cD(U(c)\to d)$ an isomorphism,
1-morphisms $(c_1,\psi_1)\to (c_2,\psi_2)$ are pairs $(A,\alpha)$ with $A\in \cC(c_1\to c_2)$ and $\alpha: \cD(\psi_1 \Rightarrow \psi_2 \circ U(A))$ is a 2-isomorphism, and so forth, where $k$-morphisms consist of pairs of a $k$-morphism in $\cC$ and a $(k+1)$-isomorphism in $\cD$ compatible with lower structure.
Here, we interpret an $(n+1)$-isomorphism as an equality.
\end{defn}

\begin{defn}
Suppose $\cC, \cD$ are $n$-groupoids and $U: \cC \to \cD$ is a weak $n$-functor.
We call $U$ $k$-\emph{truncated} or $(k+1)$-\emph{monic} \cite[\S5.5]{MR2664619} if 
\begin{itemize}
\item
$k=n$: no condition
\item
$k=n-1$: faithful on $n$-morphisms
\item
$k=n-2$: fully faithful on $n$-morphisms
\item
$-2\leq k< n-2$: fully faithful on $n$-morphisms and essentially surjective on $j$-morphisms for all $k+2\leq j\leq n-1$.
(Thus a $(-2)$-truncated $n$-functor is an equivalence.)
\end{itemize}
Under the homotopy hypothesis, 
%\nn{} which identifies $n$-groupoids with homotopy $n$-types and $n$-functors with continuous maps,
$U$ being $n$-truncated corresponds to $U_*: \pi_*(\cC) \to \pi_*(\cD)$ being 
injective on $\pi_{k+1}(\cC)$ and an isomorphism on $\pi_{j}(\cC)$ for all $j\geq k+2$ for all basepoints.
\end{defn}

\begin{prop}
\label{prop:TruncatedGroupoids}
Suppose $\cC,\cD$ are $n$-groupoids, and $U: \cC \to \cD$ is a weak $n$-functor.
For every $-2\leq k\leq n$,
$U$ is $k$-truncated if and only if 
at each object $d\in \cD$, the homotopy fiber $\hoFib_d(U)$ is $k$-truncated as an $n$-groupoid, i.e., a $k$-groupoid.\footnote{Here, we use `negative categorical thinking' \cite{MR2664619} when $k=-2,-1,0$.
That is, a $0$-groupoid is a set, a $(-1)$-groupoid is either a point or the empty set, and a $(-2)$-groupoid is a point.}
\end{prop}
\begin{proof}
Under the homotopy hypothesis, given a $d\in \cD$, we have a fibration $\hoFib_d(U) \to \Path(U) \to \cD$ which yields a long exact sequence in homotopy groups.
%\nn{or equivalently, a homotopy fiber sequence $\hoFib_d(U) \to \cC \to \cD$.}
Recall $\hoFib_d(U)$ is a $k$-groupoid if and only if $\pi_{j}(\hoFib_d(U))=0$ for all $j>k$.
Since $\Path(U)$ is homotopy equivalent to $\cC$, this happens if and only if $U_*: \pi_*(\cC) \to \pi_*(\cD)$ gives an injection $\pi_{k+1}(\cC)\hookrightarrow \pi_{k+1}(\cD)$ and an isomorphism on $\pi_{j}(\cC)\cong \pi_j(\cD)$ for all $j\geq k+2$.
The result now follows by quantifying over all objects $d\in \cD$.
\end{proof}

\begin{remark}
In this article, we will only every use the above proposition on $n$-functors between $n$-groupoids where $n\leq 2$, where the homotopy hypothesis is known to hold.
Furthermore, it is straightforward to give an explicit proof of Proposition \ref{prop:TruncatedGroupoids} for 2-groupoids using the formalism of bicategories 
which does not invoke the homotopy hypothesis.
%This is easier than a similar proof we supply below for 2-categories.
We leave these details to the interested reader.
\end{remark}

%\begin{prop}
%Suppose 
%$\cC,\cD$ are 2-groupoids \nn{or cats?}
%and
%$U:\cC \to \cD$ is a weak 2-functor.
%The homotopy fiber $\Fib_d(U)$ is
%\begin{itemize}
%\item 
%1-truncated, i.e.~equivalent to a 1-groupoid, if and only if $U$ is faithful on 2-morphisms,
%\item
%0-truncated, i.e.~equivalent to a set, if and only if $U$ is fully faithful on 2-morphisms
%\item
%(-1)-truncated, i.e.~equivalent to a point or empty, if and only if $U$ is fully faithful on 2-morphisms and essentially surjective on 1-morphisms, and
%\item
%(-2)-truncated, i.e.~equivalent to a point, if and only if $U$ is an equivalence.
%\end{itemize}
%\nn{not said quite correctly -- want $k$-truncated \emph{as a 2-category}}
%\end{prop}
%\begin{proof}
%\nn{todo}
%\end{proof}

%%%%%%%%%%%%%%%%%%%%%%%%%%%%%%%%%%%%%%%%%%%%%%%%%
\subsection{Strictification of fibers}

We now discuss for how to strictify the homotopy fiber of a 0-truncated functor $U: \cC \to \cD$ of $n$-groupoids for $n=1,2$.

\subsubsection{\texorpdfstring{$n=1$}{n=1}}
\label{sec:StrictifyFibersFor1Groupoids}
Suppose $\cC,\cD$ are groupoids and $U: \cC \to \cD$ is a 0-truncated functor, i.e., a faithful functor.
In the previous section, we saw that $U: \cC \to \cD$ gives a fibration $\Path(U) \to \cD$.
In order to strictify the homotopy fiber, it is necessary that $U: \cC \to \cD$ is a fibration in the canonical model structure on $\Cat$, i.e., an \emph{isofibration}, meaning every isomorphism in $\cD$ can be lifted to $\cC$.
Since $U$ was assumed to be faithful, every isomorphism can be lifted \emph{uniquely} subject to a fixed source and target.

However, an isofibration is not quite strong enough to strictify the homotopy fiber, as we may not have uniqueness of the source of the lifted isomorphism in $\cC$.
Hence we requre that
\begin{itemize}
\item
$U$ is a \emph{discrete fibration}, i.e., for every $c\in \cC$ and every morphism $g\in \cD(d\to U(c))$, there is a unique $b\in \cC$ and a unique morphism $f\in \cC(b\to c)$ such that $U(b)=d$ and $U(f)=g$.
\end{itemize}

\begin{defn}
We define the \emph{strict fiber} $\stFib_d(U)$ of $U$ at $d$ is the set of objects $c\in \cC$ such that $U(c)=d$.
\end{defn}

\begin{prop}
\label{prop:StrictFiber}
Suppose $U: \cC \to \cD$ is a 0-truncated functor of 1-groupoids which is a discrete fibration.
For every $d\in \cD$, there is a canonical bijection
$\stFib_d(U)\cong \tau_0(\hoFib_d(U)))$, the 0-truncation of the homotopy fiber of $U$ at $d$.
\end{prop}
\begin{proof}
Suppose $d\in \cD$.
If $d$ is not in the essential image of $U$, then both $\stFib_d(U)$ and $\hoFib_d(U)$ are empty.

Now assume $d$ is in the essential image of $U$ so that there is a $c\in \cC$ and a $g\in \cD(d\to U(c))$.
Since $U$ is a discrete fibration, there is a unique $b\in \cC$ and a unique $f\in \cC(b\to c)$ such that $U(f)=g$.
In particular, $U(b)=d$.

Denote the 0-truncation of $\hoFib_d(U)$ by $\tau_0(\hoFib_d(U))$.
Recall that elements of $\tau_0(\hoFib_d(U))$ are equivalence classes $[(c,g)]$ where $(c_1,g_1)\sim (c_2,g_2)$ if there exists an isomorphism $f: c_1 \to c_2$ such that $g_2\circ U(f)=g_1$.
Observe that if 
$(c,g)\in \hoFib_d$
and
$b\in \cC$ such that $U(b)=d$ and $f: b\to c$ such that $U(f)=g$, 
then $[(b,\id_b)]=[(c,g)]$.

Define $\Phi : \tau_0(\hoFib_d(U))\to \stFib(U)$ by $[(c,g)]\mapsto b$ where $U(b)=d$ and $U(f)=g$.
It is straightforward to verify that this map is well-defined.
Going the other way, define $\Psi: \stFib_d(U) \to \tau_0(\hoFib_d(U))$ by $b\mapsto [(b, \id_d)]$.
It is straightforward to show these maps are mutually inverse.
\end{proof}

\begin{example}
Let $\cC$ be a fusion category and $Z(\cC)$ its Drinfeld center.
Consider the forgetful tensor functor $\Forget_Z: Z(\cC) \to \cC$ which is faithful.
Its restriction to cores $\Forget_Z:\core(Z(\cC))\to \core(\cC)$ is also a discrete fibration.
Given a $c\in \cC$, the 0-truncation of the homotopy fiber $\tau_0(\hoFib_c(\Forget_Z))$ is in canonical bijection with the strict fiber $\stFib_c(\Forget_Z)$, which we view as the \emph{set of half-braidings on $c$}.
\end{example}

\begin{example}
Suppose $\cC\subset \cD$ is a (fully faithful) inclusion of fusion categories.
The forgetful functor $\Forget_Z:Z_\cC(\cD) \to \cD$ from the relative Drinfeld center to $\cD$ is fully faithful.
Moreover, its restriction to cores $\Forget_Z:\core(Z_\cC(\cD))\to \core(\cD)$ is also a discrete fibration.
Given a $d\in \cD$, the 0-truncation of the homotopy fiber $\tau_0(\hoFib_d(\Forget_Z))$ is in canonical bijection with the strict fiber $\stFib_d(\Forget_Z)$, which we view as the \emph{set of relative half-braidings on $d$ with $\cC$}.
\end{example}

\begin{example}
\label{ex:LiftToEquivariantization}
Let $\cC$ be a fusion category and $(\rho,\rho^1, \rho^2) : \uG \to \uAut_\otimes(\cC)$ a categorical $G$-action on $\cC$ by tensor automorphisms.
We write $\rho_g^1,\rho_g^2$ for the unitor and tensorator of $\rho_g$.
Recall from \cite[\S2.7]{MR3242743} that the \emph{equivariantization} $\cC^G$ 
has
\begin{itemize}
\item
objects are $c\in \cC$ together with a family of isomorphisms $\lambda_g \in \cC(g(c)\to c)$ for $g\in G$ such that the following diagram commutes:
$$
\begin{tikzcd}
\rho_g(\rho_h(c))
\arrow[r, "\rho_g(\lambda_h)"]
\arrow[d, "(\rho^2_{g,h})_c"]
&
\rho_g(c)
\arrow[d, "\lambda_g"]
\\
\rho_{gh}(c)
\arrow[r,"\lambda_{gh}"]
&
c.
\end{tikzcd}
$$
\item
morphisms are $f\in \cC((c,\lambda_g) \to (d, \kappa_g))$ such that the following diagram commutes for all $g\in G$:
$$
\begin{tikzcd}
\rho_g(c)
\arrow[r, "g(f)"]
\arrow[d, "\lambda_g"]
&
\rho_g(d)
\arrow[d, "\kappa_g"]
\\
c
\arrow[r,"f"]
&
c.
\end{tikzcd}
$$
\end{itemize}
The equivariantization tensor product given by
$$
(c, \lambda_g) \otimes (d, \kappa_g):=(c\otimes d, (\lambda_g\otimes \kappa_g)\circ (\rho_g^2)_{c,d}^{-1})
$$
and unit object $(1_\cC, \id_{1_\cC})$.

We have an obvious faithful forgetful tensor functor $\Forget_G: \cC^G\to \cC$ which forgets the $G$-equivariant structure. 
Its restriction to cores $\core(\cC^G) \to \core(\cC)$ is also a discrete fibration.
Given $c\in \cC$, the $0$-truncation of the homotopy fiber $\tau_0(\hoFib_c(\Forget_G))$ is in canonical bijection with the strict fiber $\stFib_c(\Forget_G)$, which we view as the \emph{set of $G$-equivariant structures on $c$}.
\end{example}

\begin{example}[Set of lifts of a (monoidal) functor]
\label{ex:LiftsOfFunctor}
Suppose $\cA,\cB,\cC$ are (monoidal) categories and $F: \cA \to \cC$ and $G: \cB \to \cC$ are (monoidal) functors.
The \emph{category of lifts of $F$} is the homotopy fiber $\hoFib_F(G\circ -)$ where $G\circ- : \core(\Hom(\cA \to \cB)) \to \core(\Hom(\cA \to \cC))$.
In more detail,
\begin{itemize}
\item 
objects are pairs $(\widetilde{F}, \alpha)$ with $\widetilde{F}: \cA \to \cB$ a (monoidal) functor and $\alpha: F \Rightarrow G\circ \widetilde{F}$ a (monoidal) equivalence, and
\item
1-morphisms $(\widetilde{F}_1, \alpha_1)\to (\widetilde{F}_2,\alpha_2)$ are (monoidal) natural isomorphisms $\eta: \widetilde{F}_1 \Rightarrow \widetilde{F}_2$ such that 
$$
\begin{tikzcd}
&&
\cB
\arrow[dd, "G"]
\\
&\mbox{}
\arrow[dr, Rightarrow, swap, shorten <= 1em, shorten >= 1em, "\alpha_1"]
&
\\
\cA
\arrow[uurr, "\widetilde{F}_1"]
\arrow[rr,"F"]
&&
\cC
\end{tikzcd}
\qquad
=
\qquad
\begin{tikzcd}
\mbox{}
\arrow[dr, Rightarrow, shorten <= 1.5em, shorten >= .5em, "\,\,\,\,\eta"]
&&
\cB
\arrow[dd, "G"]
\\
&\mbox{}
\arrow[dr, Rightarrow, swap, shorten <= 1em, shorten >= 1em, "\alpha_2"]
&
\\
\cA
\arrow[uurr, near end, "\widetilde{F}_2"]
\arrow[uurr,bend left = 90, "\widetilde{F}_1"]
\arrow[rr, "F"]
&&
\cC.
\end{tikzcd}
$$
\end{itemize} 
%When $\cA,\cB,\cC$ are 2-categories and $F: \cA \to \cB$ and $G: \cB \to \cC$ are 2-functors, 
%the 1-morphisms are equipped with extra data of an invertible 2-modification, and the 2-morphisms are compatible 2-modifications.
When $G$ is 0-truncated, $G\circ -$ is 0-truncated.
%Indeed, if $G$ is faithful and $\gamma,\delta: F_1\Rightarrow F_2$ are natural isomorphisms with equal whiskerings
%$G\circ F_1 \Rightarrow G\circ F_2$, then $G(\gamma_a)=G(\delta_a)$ for every $a\in \cA$, and thus $\gamma=\delta$.
If moreover the restriction $G:\core(\cB)\to \core(\cC)$ is a discrete fibration,
then so is the restriction $G\circ- : \core(\Hom(\cA \to \cB)) \to \core(\Hom(\cA \to \cC))$.
We leave the verification of this enjoyable exercise to the reader.
In this case, the 0-truncation of the homotopy fiber $\tau_0(\hoFib_F(G\circ -)$ is in canonical bijection with the strict fiber $\stFib_F(G\circ-)$, which we view as the \emph{set of lifts of $F$}.
These are exactly the functors $\widetilde{F}: \cA \to \cB$ such that $G\circ \widetilde{F}=F$ \emph{on the nose}.
\end{example}

\subsubsection{\texorpdfstring{$n=2$}{n=2}}
\label{sec:StrictifyFibersFor2Groupoids}
Suppose $\cC,\cD$ are 2-groupoids and $U: \cC \to \cD$ is a 0-truncated 2-functor, i.e., fully faithful on 2-morphisms.
Again, in order to strictify the homotopy fiber, it is necessary, but not sufficient, that $U$ is a fibration in the canonical model structure on bicategories \cite[\S2]{MR2138540}.

\begin{defn}
By a slight abuse of notation, we call such a 0-truncated 2-functor $U: \cC \to \cD$ a \emph{discrete fibration} if for each $c\in \cC$ and 1-morphism $g\in \cD(d\to U(c))$, there is a unique $b\in \cC$ and a unique 1-morphism $f\in \cC(b\to c)$ such that $U(b)=d$ and $U(f)=g$.
\end{defn}

\begin{defn}
We define the \emph{strict fiber} $\stFib_d(U)$ of $U$ at $d$ is the set of objects $c\in \cC$ such that $U(c)=d$.
\end{defn}

The proof of the following proposition is similar to Proposition \ref{prop:StrictFiber} and omitted.

\begin{prop}
Suppose $U: \cC \to \cD$ is a 0-truncated 2-functor of small 2-groupoids which is a discrete fibration.
For every $d\in \cD$, there is a canonical bijection $\stFib_d(U)\cong \tau_0(\hoFib_d(U))$.
\end{prop}

\begin{example}[Set of braidings]
The strict 2-category 
$\BrdMonCat$
of 
braided monoidal categories, braided monoidal functors, and monoidal natural transformations
admits a strict forgetful 2-functor 
$\Forget_{\br}$
to 
the strict 2-category 
$\MonCat$
of
monoidal categories, monoidal functors, and monoidal natural transformations
which is fully faithful on 2-morphisms, 
since every monoidal natural transformation of braided monoidal functors is compatible with the braidings.
Moreover, its restriction to cores is a discrete fibration; indeed, if $\cB$ is a braided monoidal category, $\cC$ is a monoidal category, and $F: \cC \to \Forget_{\br}(\cB)$ is any monoidal equivalence, there is a unique way to transport the braiding on $\cB$ to a braiding on $\cC$ such that $F$ is a braided equivalence.
Fixing a monoidal category $\cC\in \core(\MonCat)$, the homotopy fiber $\hoFib_\cC(\Forget_{\br})$ is 0-truncated, i.e., a set.
Moreover, its 0-truncation $\tau_0(\hoFib_\cC(\Forget_{\br}))$ is in canonical bijection with the strict fiber $\stFib_\cC(\Forget_{\br})$, which we view as the \emph{set of braidings} on $\cC$.
\end{example}

\begin{example}[Set of $G$-gradings]
\label{ex:SetOfGGradings}
The strict 2-category 
$G\GrdMonCat$
of 
$G$-graded monoidal categories, $G$-graded monoidal functors, and natural transformations 
admits a strict forgetful 2-functor 
$\Forget_{G}$
to 
the strict 2-category $\MonCat$
%monoidal categories, monoidal functors, and monoidal natural transformations
which is fully faithful on 2-morphisms, 
since every monoidal natural transformation of $G$-graded monoidal functors is compatible with the gradings.
Moreover, its restriction to cores $\Forget_G: \core(G\GrdMonCat) \to \core(\MonCat)$ is a discrete fibration.
Fixing a monoidal category $\cC$, the 0-truncation of the homotopy fiber $\tau_0(\hoFib_\cC(\Forget_{G}))$ is in canonical bijection with the strict fiber $\stFib_\cC(\Forget_{G})$.
We think of this set as the \emph{set of $G$-gradings} on $\cC$.
\end{example}

\begin{example}[Equivalence classes of $G$-crossed braidings]
\label{example:SetOfGCrossedBraidings}
The strict 2-category 
$G\CrsBrd$
of 
$G$-crossed braided fusion categories, $G$-crossed braided monoidal functors, and natural transformations 
admits a strict forgetful 2-functor 
$\Forget_{\beta}$
to 
the strict 2-groupoid 
$G\GrdFusCat$
of $G$-graded fusion categories, $G$-graded monoidal functors, and natural transformations 
which is fully faithful on 2-morphisms.
%Moreover, its restriction to cores $\Forget_{\beta}: \core(G\CrsBrd) \to \core(G\GrdMonCat)$ is a discrete fibration.
%Fixing a $G$-graded monoidal category $\cC$, the 0-truncation of the homotopy fiber $\tau_0(\hoFib_\cC(\Forget_{\beta}))$ is in canonical bijection with the strict fiber $\stFib_\cC(\Forget_{\beta})$.
%We think of this set as the \emph{set of $G$-crossed braidings} on $\cC$.

Unfortunately, the restriction to cores $\Forget_{\beta}: \core(G\CrsBrd) \to \core(G\GrdFusCat)$ is \emph{not} a discrete fibration, as given a $G$-crossed braided fusion category $\cC$ and a $G$-graded equivalence to another $G$-graded fusion category $\cD$, there is not a unique $G$-crossed braiding on $\cD$ such that $\cC,\cD$ are $G$-crossed braided equivalent.
This discrepancy arises as there is not a unique $G$-action on $\cD$ compatible with the equivalence $\cC\cong \cD$, but an equivalence class.

Suppose $\cD=\bigoplus_{g\in G} \cD_g$ is a $G$-graded fusion category and
$(\rho^1,\beta^1),(\rho^2,\beta^2)$ are two $G$-crossed braidings on $\cD$. 
We say
$(\rho^1,\beta^1),(\rho^2,\beta^2)$ are \emph{equivalent} if there is an equivalence $\eta : \rho^1
\Rightarrow \rho^2$ of monoidal functors $\uG \to \uAut^\otimes(\cD)$ such that for all $x_g \in
\cD_g$ and $y\in \cD$,
$$
(\eta^g_{y}\otimes \id_{x_g})\circ \beta^1_{x_g, y} = \beta^2_{x_g, y} : x_g \otimes y \to \rho^2_g(y) \otimes x_g.
$$
Observe that there is \emph{at most one} equivalence between any two $G$-crossed braidings as the
monoidal natural isomorphism $\eta$ is completely determined by $\beta^1,\beta^2$ if it exists.
(Indeed, $\beta^1_{x_g, y}$ is invertible, and $-\otimes \id_c$ is injective on hom spaces for every
fusion category using \cite[Lem.~A.5]{MR3578212}.)

Thus, fixing a $G$-graded fusion category $\cD$, the 0-truncation of the homotopy fiber $\tau_0(\hoFib_\cD(\Forget_{\beta}))$ 
of
$\Forget_{\beta}: \core(G\CrsBrd) \to \core(G\GrdFusCat)$
is in canonical bijection with the set of equivalence classes $G$-crossed braidings on $\cD$.
\end{example}

%\begin{remark}
%\nn{show this is the same set as above.}
%Suppose $\cE=\bigoplus_{g\in G} \cE_g$ is a faithfully $G$-graded fusion category and
%$(\rho^1,\beta^1),(\rho^2,\beta^2)$ are two $G$-crossed braidings on $\cE$. We say
%$(\rho^1,\beta^1),(\rho^2,\beta^2)$ are \emph{equivalent} if there is an equivalence $\eta : \rho^1
%\Rightarrow \rho^2$ of monoidal functors $\uG \to \uAut^\otimes(\cE)$ such that for all $x_g \in
%\cE_g$ and $y\in \cE$,
%$$
%(\eta^g_{y}\otimes \id_{x_g})\circ \beta^1_{x_g, y} = \beta^2_{x_g, y} : x_g \otimes y \to \rho^2_g(y) \otimes x_g.
%$$
%Observe that there is \emph{at most one} equivalence between any two $G$-crossed braidings as the
%monoidal natural isomorphism $\eta$ is completely determined by $\beta^1,\beta^2$ if it exists.
%(Indeed, $\beta^1_{x_g, y}$ is invertible, and $-\otimes \id_c$ is injective on hom spaces for every
%fusion category using \cite[Lem.~A.5]{MR3578212}.)
%\end{remark}

\begin{example}[Set of lifts of a (monoidal) 2-functor]
\label{ex:Lift2Functor}
Given two (monoidal) 2-categories $\cA,\cB$, there is a 2-category $\Hom(\cA \to \cB)$ whose 
objects are (monoidal) 
2-functors, 
1-morphisms are (monoidal) natural transformations, and 
2-morphisms are (monoidal) modifications.
Now suppose $G:\cB \to \cC$ is a (monoidal) 2-functor between (monoidal) 2-categories and $\cA$ is another (monoidal) 2-category.
We have a 2-functor $G\circ- : \Hom(\cA \to \cB) \to \Hom(\cA \to \cC)$.
As in Example \ref{ex:LiftsOfFunctor}, if $G$ is 0-truncated, then so is $G\circ -$.
Moreover, if $G$ is a discrete fibration, then so is $G\circ -$.
Fixing $F\in \Hom(\cA \to \cC)$, the 0-truncation of the homotopy fiber $\tau_0(\hoFib_F(G\circ -))$ is in canonical bijection with the strict fiber $\stFib_F(G\circ -)$ of (monoidal) 2-functors $\widetilde{F}: \cA \to \cB$ such that $G\circ \widetilde{F}=F$ \emph{on the nose}.
We view this set as the \emph{set of lifts} of $F$.
\end{example}

%%%%%%%%%%%%%%%%%%%%%%%%%%%%%%%%%%%%%%%%%%%%%%%%%%%%%%
\subsection{\texorpdfstring{$G$}{G}-gradings on fusion categories}
\label{sec:G-gradings}

Fix a finite group $G$. 
In this section, we explain how $G$-gradings on a fixed fusion category $\cT$ may be characterized in terms of braided-enriched structure.
In this section, our fusion categories are always over an algebraically closed field $\Bbbk$ of characteristic zero.
%and \emph{Cauchy complete}.
%For this section, we
%assume $\Bbbk$ is an algebraically closed field of characteristic zero. 
(Although it is not necessary, we would even be happy to assume further that $\Bbbk=\bbC$.)

%i.e., those monoidal categories which are ,
%In this section, we prove a classification theorem for $G$-gradings on
%linear monoidal categories $\cT$, i.e., those monoidal categories which are tensored over $\Vec$,
%the monoidal category of finite dimensional vector spaces over a field $\Bbbk$. 

\begin{defn}
\label{defn:CanonicalCopyOfRepGFromGrading}
%We call a linear monoidal category $\cT$ $G$-\emph{graded} if $\cT = \bigoplus_{g\in G} \cT_g$ as
%linear categories,
%\nn{One should assume that all $\cT_g$ are non-trivial for this
%to make sense.}
%and if $t_g \in \cT_g$ and $t_h \in \cT_h$, then $t_g \otimes t_h \in \cT_{gh}$.
Suppose $G$ is a group and $\cT=\bigoplus_{g\in G} \cT_g$ is a faithfully $G$-graded fusion category. 
In this case, by \cite[p.~12]{MR2587410} there is a canonical fully faithful strong monoidal functor
$\cI=\cI^\cT: \Rep(G) \to Z(\cT)$ defined as follows. 
For a representation $(K,\pi)\in \Rep(G)$, we
consider the object $\cI_\pi := K\otimes 1_{\cT} \in \cT$. 
Notice that both $\cI_\pi\otimes t$ and
$t\otimes \cI_\pi$ are canonically isomorphic to $K\otimes t$. 
Thus we can endow $\cI_\pi$ with the
half-braiding 
$$
\zeta_{t,\cI_\pi}
:=
\pi_g \otimes \id_t : t\otimes \cI_\pi \cong K\otimes t\longrightarrow  K\otimes t \cong \cI_\pi\otimes t
\qquad\qquad
t\in \cT_g.
$$
For a morphism $f : (K,\pi) \to (L,\rho)$, we get a morphism $\cI_f := f\otimes \id_{1_\cT}: \cI_\pi
\to \cI_\rho$. 
It is straightforward to verify:
\begin{itemize}
\item
$\cI$ is a fully faithful strong monoidal
functor (using the obvious tensorator/strength) since $\cT$ is tensored over $\Vec$,
\item
the forgetful functor $\Forget_Z: Z(\cT) \to \cT$ restricted to this copy of
$\Rep(G)\subseteq Z(\cT)$ is canonically monoidally naturally isomorphic 
to the canonical symmetric monoidal fiber functor
$$
\Forget_\Rep : \Rep(G) \to \Vec \cong \langle 1_\cT, \id_{1_\cT}\rangle \subseteq \cT.
$$
%\item
%equivalent $G$-gradings on $\cT$ induce equivalent braided monoidal functors $\Rep(G) \to Z(\cT)$.
\end{itemize}
\end{defn}

The following important lemma is essentially in \cite{MR2587410}.
Recall that given a braided fusion category $(\cV,\beta)$ and a symmetric subcategory $\cS \subset \cV$, the \emph{M\"uger centralizer} $\cS'$ of $\cS\subset \cV$ is the full subcategory of $\cV$ whose objects are \emph{transparent} to $\cS$, i.e., $\beta_{v,s}\circ \beta_{s,v} = \id_{s\otimes v}$ for all $s\in\cS$.

\begin{lem}
\label{lem:LandInTrivialComponent}
An object $(t, \sigma_{\bullet,t})\in \Rep(G)' \subseteq Z(\cT)$ if and only if
$\Forget_Z(t,\sigma_{\bullet,t})= t\in \cT_e$.
\end{lem}
\begin{proof}
It is clear that $t\in \cT_e$ implies $(t,\sigma_{\bullet,t})\in \Rep(G)'$.
Suppose $(t,\sigma_{\bullet,t})\in \Rep(G)'$.
If $t=\bigoplus_g t_g$, then for all $(H,\pi)\in\Rep(G)$,
$$
\id_{\cI_\pi\otimes t}
=
\zeta_{t,\cI_\pi} \circ \sigma_{\Forget_Z(\cI_\pi),t}
=
\zeta_{t,\cI_\pi} \circ \sigma_{\bigoplus 1_\cT,t}
=
\bigoplus_g
(\pi_g\otimes \id_{t_g}).
$$
Since $\cT$ is fusion, the above holds if and only if $t_g=0$ for all $g\neq e$.
\end{proof}

\begin{defn}
\label{defn:fibered-enrichment}
Suppose $(\cV,F: \cV \to \Vec)$ is a braided fusion category equipped with a fixed faithful strong monoidal fiber functor.
Given a fixed fusion category $\cT$, a $\cV$-\emph{fibered enrichment} of $\cT$ is a
braided strong monoidal functor $\cF^{\sZ}:\cV \to Z(\cT)$ 
together with a monoidal natural isomorphism 
$$
\begin{tikzcd}
\cV
\arrow[r, "\cF^\sZ"]
\arrow[d,"F"]
&
Z(\cT)
\arrow[d, "\Forget_Z"]
\\
\Vec
\arrow[r,hookrightarrow, "i_\cT"]
\arrow[ur, Rightarrow, shorten <= 1em, shorten >= 1em, "\alpha"]
&
\cT
\end{tikzcd}
\qquad\qquad
\alpha: \cF:= \Forget_Z\circ \cF^{\sZ} \Rightarrow i_{\cT}\circ F
$$
%$\alpha : \Forget_Z\circ \cF^\sZ \Rightarrow i_\cT\circ F$, 
where  $i_\cT:\Vec = \langle 1_\cT\rangle \hookrightarrow \cT$ is the inclusion $V\mapsto V\otimes 1_\cT$.

Similar to $\cV$-module tensor categories as in Definition \ref{defn:VModuleTensorCategories}, $\cV$-fibered enrichments form a 2-category $\cV\FibFusCat$.
A 1-morphism $(\cT_1,\cF_1^\sZ, \alpha^1)\to (\cT_2, \cF^\sZ_2, \alpha^2)$ is a 1-morphism $(H,\eta): (\cT_1,\cF_1^\sZ) \to (\cT_2,\cF_2^\sZ)$ of the underlying $\cV$-module tensor categories satisfying the extra compatibility with $\alpha^1,\alpha^2$:
\begin{equation}
\label{eq:FibFusCat1MorphismCoherence}
\begin{tikzcd}
&
\mbox{}
\arrow[d, Rightarrow, shorten <= .25em, shorten >= .25em, "\eta"]
\\
\cV
\arrow[r, "\cF_1^\sZ"]
\arrow[d,"F"]
\arrow[rr, bend left=60, "\cF_2^\sZ"]
&
Z(\cT_1)
\arrow[d, "\Forget_Z"]
&
Z(\cT_2)
\arrow[d, "\Forget_Z"]
\\
\Vec
\arrow[r,hookrightarrow, "i_{\cT_1}"]
\arrow[rr,hookrightarrow,bend right=60, swap, "i_{\cT_2}"]
\arrow[ur, Rightarrow, shorten <= 1em, shorten >= 1em, "\alpha^1"]
&
\cT_1
\arrow[r, "H"]
&
\cT_2
\\
&
\mbox{}
\arrow[u, Rightarrow, shorten <= .25em, shorten >= .25em, "H^1"]
\end{tikzcd}
=
\begin{tikzcd}
\cV
\arrow[r, "\cF_2^\sZ"]
\arrow[d,"F"]
&
Z(\cT_2)
\arrow[d, "\Forget_Z"]
\\
\Vec
\arrow[r,hookrightarrow, "i_{\cT_2}"]
\arrow[ur, Rightarrow, shorten <= 1em, shorten >= 1em, "\alpha^2"]
&
\cT_2
\end{tikzcd}
.
\end{equation}
Above, note that the unitality constraint 
$H^1: 1_{\cT_2}\to G(1_{\cT_1})$ 
determines a monoidal natural isomorphism still denoted $H^1:i_{\cT_2} \Rightarrow H\circ i_{\cT_1}$ of monoidal functors $\Vec \to \cT_2$.
Moreover, observe that $\eta$ is completely determined, as $\alpha^1,\alpha^2,H^1$ are all isomorphisms; indeed, \eqref{eq:FibFusCat1MorphismCoherence} above is equivalent to
which is equivalent to
\begin{equation}
\label{eq:FibFusCat1MorphismCoherence2}
\eta_v = H(\alpha^1_v) \circ (H^1)^{-1}_{F(v)}\circ (\alpha^2_v)^{-1}
\qquad\qquad
\forall\,v\in\cV.
\end{equation}

A 2-morphism $\kappa: (H_1,\eta^1) \Rightarrow (H_2,\eta^2)$ is an \emph{arbitrary} monoidal natural transformation $\kappa: H_1 \Rightarrow H_2$.
Observe that the the extra compatibility with the fiber functor $F$ amounts to unitality of the monoidal natural isomorphism $\kappa$, and the extra action coherence \eqref{eq:CoherenceForVModuleTensorCategory2Morphisms} with $\eta^1,\eta^2$ is automatically satisfied.
Indeed, setting $\cF_i:=\Forget_Z\circ \cF_i^\sZ$ for $i=1,2$, \eqref{eq:CoherenceForVModuleTensorCategory2Morphisms} automatically commutes by naturality and unitality of $\kappa$:
$$
\begin{tikzcd}
&&\mbox{}
\arrow[d,near start,phantom, "=\text{\scriptsize by \eqref{eq:FibFusCat1MorphismCoherence2}\hspace{2.5cm}}"]
\\
\cF_2(v)
\arrow[d, equals]
\arrow[r,"(\alpha^2_v)^{-1}"]
\arrow[rrr, bend left=30, "\eta^1_v"]
&
F(v)\otimes 1_{\cT_2}
\arrow[r,"(H_1^1)^{-1}_{F(v)}"]
\arrow[d, equals]
&
H_1(F(v)\otimes 1_{\cT_1})
\arrow[r,"H_1(\alpha^1_v)"]
\arrow[d,Rightarrow,"\kappa_{F(v)\otimes 1_{\cT_1}}"]
&
H_1(\cF_1(v))
\arrow[d,Rightarrow, "\kappa_{\cF_1(v)}"]
\\
\cF_2(v)
\arrow[r,"(\alpha^2_v)^{-1}"]
\arrow[rrr, bend right=30, "\eta^2_v"]
&
F(v)\otimes 1_{\cT_2}
\arrow[r,"(H_2^1)^{-1}_{F(v)}"]
&
H_2(F(v)\otimes 1_{\cT_1})
\arrow[r,"H_2(\alpha^1_v)"]
&
H_2(\cF_1(v))
\\
&&\mbox{}
\arrow[u,phantom, "=\text{\scriptsize by \eqref{eq:FibFusCat1MorphismCoherence2}\hspace{2.5cm}}"]
\end{tikzcd}
$$
\end{defn}

The following lemma allows us to work with an equivalent strict $\cV$-fibered enrichement.

\begin{lem}
\label{lem:StrictifyFiberedEnrichment}
The 2-category $\cV\FibFusCat$ is equivalent to the full 2-subcategory $\cV\FibFusCat^{\sf st}$ with objects $(\cT,\cF^\sZ,\id)$, i.e., $\Forget_Z\circ \cF^\sZ = i_\cT \circ F$ on the nose.
\end{lem}
\begin{proof}
Suppose $(\cT,\cF^\sZ, \alpha) \in \cV\FibFusCat$.
Define $(\cT, \cF'^\sZ,\id)\in \cV\FibFusCat$
by 
$\cF'(v) := F(v)\otimes 1_\cT$ and  $\sigma'_{t,\cF'(v)}:=(\alpha_v^{-1}\otimes \id_t)\circ \sigma_{\cF(z),t}\circ (\id_t\otimes \alpha_v)$ with tensorator $\mu'_{u,v}:= \alpha_{u\otimes v}^{-1}\circ \mu_{u,v}\circ (\alpha_u\otimes \alpha_v)$.
By definition, we have $\cF' :=\Forget_Z\circ \cF'^\sZ = i_\cT \circ F$ on the nose, so $(\cT,\cF'^\sZ, \id)\in\cV\FibFusCat^{\sf st}$.
We claim that $(H:=\id_\cT, \eta:=\alpha): (\cT,\cF^\sZ, \alpha) \to (\cT,\cF'^\sZ, \id)$ defines an invertible 1-morphism in $\cV\FibFusCat$.
it is clear that $\alpha^1=\alpha, \alpha^2=\id, \eta=\alpha$ satisfies \eqref{eq:FibFusCat1MorphismCoherence2}.
It remains to check that $\eta=\alpha$ satisfies the half-braiding coherence \eqref{eq:CoherenceForVModuleTensorCategoryHalfBraidings}, which expands to the formula $(\alpha_v\otimes \id_t)\circ \sigma'_{t,\cF'(v)}=\sigma_{t,\cF(z)}\circ (\id_t\otimes \alpha_v)$, which holds by definition.
\end{proof}

The following theorem shows that fully faithful $\Rep(G)$-fibered enriched fusion categories are the same as faithfully $G$-graded fusion categories.
We denote by $G\GrdFusCat_{\sf f}$ the 2-category of \emph{faithfully} $G$-graded fusion categories and by  $\Rep(G)\FibFusCat_{\sf ff}$ the 2-category of \emph{fully faithful} $\Rep(G)$-fibered enriched fusion categories, where we endow $\Rep(G)$ with the canonical symmetric fiber functor to $\Vec$.

\begin{thm}
\label{thm:GGraded2CatEquivalence}
There is a strict 2-equivalence $\Phi : G\GrdFusCat_{\sf f} \to \Rep(G)\FibFusCat_{\sf ff}$ such that the following triangle commutes:
\begin{equation}
\label{eq:GGradedTriangle}
\begin{tikzcd}
G\GrdFusCat_{\sf f}
\arrow[rr, "\Phi"]
\arrow[dr, swap, "\Forget_G"]
&&
\Rep(G)\FibFusCat_{\sf ff}
\arrow[dl, "\Forget_{\Rep(G)}"]
\\
&\FusCat
\end{tikzcd}
\end{equation}
Here, $\Forget_G: G\GrdFusCat\to \FusCat$ forgets the $G$-grading (cf.~Example \ref{ex:SetOfGGradings}), and $\Forget_{\Rep(G)}: \cV\FibFusCat \to \FusCat$ forgets the $\Rep(G)$-fibered enrichment.
\end{thm}
\begin{proof}
We already saw in Definition \ref{defn:CanonicalCopyOfRepGFromGrading} how to endow a faithfully $G$-graded fusion category $\cT$ with a fully faithful $\Rep(G)$-fibered enrichment.
It is straightforward to show that every $G$-graded monoidal functor $H:\cT_1\to \cT_2$ gives a 1-morphism in $\Rep(G)\FibFusCat$, since the coherence $\eta$ is completely determined by \eqref{eq:FibFusCat1MorphismCoherence}.
Indeed, the verification that this determined $\eta$ satisfies the compatibility \eqref{eq:CoherenceForVModuleTensorCategoryHalfBraidings} with the half-braidings amounts to the following commuting square for a homogeneous $t_g\in (\cT_1)_g$ and $(K,\pi)\in \Rep(G)$:
$$
\begin{tikzcd}
K \otimes H(t_g)
\arrow[d,"\pi_g \otimes \id_{H(t_g)}"]
\arrow[r,"\cong "]
&
H(K\otimes t_g)
\arrow[d,"H(\pi_g \otimes \id_{t_g})"]
\\
K\otimes H(t_g)
\arrow[r,"\cong"]
&
H(K\otimes t_g)
\end{tikzcd}.
$$
Moreover, the 2-morphisms of both $G\GrdFusCat_{\sf f}$ and $\Rep(G)\FibFusCat_{\sf ff}$ consist of \emph{all} monoidal natural transformations, so $\Phi$ is the identity on 2-morphisms.
We leave it to the reader to check that $\Phi$ is a strict 2-functor which is obviously fully faithful on 2-morphisms such that \eqref{eq:GGradedTriangle} commutes.

It remains to show essential surjectivity on objects and 1-morphisms.
By Lemma \ref{lem:StrictifyFiberedEnrichment}, we may restrict our attention to the 2-subcategory $\Rep(G)\FibFusCat_{\sf ff}^{\sf st}$ of fully faithful $\Rep(G)$-fibered enriched fusion categories $(\cT,
\cI^{\scriptscriptstyle Z})$ such that $\Forget_Z\circ \cI^\sZ = i_\cT \circ F$ on the nose.

Given $(\cT,\cI^{\scriptscriptstyle Z})\in \Rep(G)\FibFusCat_{\sf ff}^{\sf st}$, we claim there is a canonical faithful $G$-grading on $\cT$ which recovers our $\Rep(G)$-fibered enrichment.
We expect this result is known to experts, but we are unaware of its existence in the literature.

Recall $\cO(G)$ is the commutative algebra of $\Bbbk$-valued functions on $G$. Moreover, $\cO(G)$ is
a Hopf algebra with comultiplication given by $\Delta(\chi_g) := \sum_h \chi_{gh^{-1}}\otimes
\chi_h$ where $\chi_g$ denotes the indicator function at $g\in G$, antipode given by $S\chi_g :=
\chi_{g^{-1}}$, and counit given by $\epsilon(\chi_g) = \delta_{g=e}$. Let $\Irr(\Rep(G))$ be a set
of representatives for the simple objects of $\Rep(G)$. There is a unital isomorphism of Hopf
algebras
\begin{equation}
\label{eq:CoendIsO(G)}
\Phi:
\bigoplus_{(K,\pi) \in \Irr(\Rep(G))} (K,\pi)^* \otimes (K,\pi)
\cong
\cO(G)
\end{equation}
given on $w^*\otimes v \in (K,\pi)^* \otimes (K,\pi)$ by
$\Phi(w^* \otimes v)(g) :=  w^*(\pi_g(v))$.
Multiplication on the left hand side is given on
$w_i^*\otimes v_i \in K_i^*\otimes K_i$
for $i=1,2$
by
$$
(w_1^*\otimes v_1)(w_2^*\otimes v_2*)
=
\sum_{
\substack{
(L,\pi)\in \Irr(\Rep(G))
\\
\{\alpha\}
\subset
\Rep(G)(K_1 \otimes K_2 \to L)
}}
[(w_1^*\otimes w_2^*)\circ \alpha^*]
\otimes
[\alpha\circ (v_1\otimes v_2)]
\qquad\qquad
$$
where $\{\alpha\} \subseteq \Rep(G)(K_1 \otimes K_2 \to L)$ is a basis and $\{\alpha^*\} \subset
\Rep(G)(L \to K_1 \otimes K_2)$ is the dual basis under the pairing $\alpha' \circ \alpha^* =
\delta_{\alpha'=\alpha} \id_L$. The unit on the left hand side is exactly $1_\bbC^* \otimes
1_\bbC\in \bbC^*\otimes \bbC$ where $\bbC \in\Rep(G)$ is the trivial representation.
Comultiplication on $w^*\otimes v \in K^*\otimes K$ is given by
$$
\Delta(w^*\otimes v)
=
\sum_{i} (w^*\otimes e_i) \otimes (e_i^* \otimes v)
$$
where $\{e_i\}$ is a basis for $K$ and $\{e_i^*\}$ is the dual basis.
We will identify both sides of \eqref{eq:CoendIsO(G)} under the isomorphism $\Phi$ below.

Now given $t\in \cT$, we get a unital $\Bbbk$-algebra homomorphism
$\cO(G) \to \cT(t\to t)$ (whose image lies in $Z(\cT(t\to t))$)
whose image on $w^*\otimes v \in K^*\otimes K$ is given by
\begin{equation}
\label{eq:CanonicalO(G)Inclusion}
w^*\otimes v
\longmapsto
\begin{tikzpicture}[baseline=-.1cm, xscale=-1]
  \begin{knot}[clip width=3]
    \strand (0,-.7) -- (0,.7);
    \strand (-.4,-.2) -- (.4,.2);
  \end{knot}
	\roundNbox{fill=white}{(-.6,-.2)}{.25}{0}{0}{$v$}
	\roundNbox{fill=white}{(.6,.2)}{.25}{0}{0}{$w^*$}
	\node at (.2,-.5) {\scriptsize{$t$}};
	\node at (-.2,.2) {\scriptsize{$\cI_\pi$}};
\end{tikzpicture}
:=
(w^*\otimes \id_t)
\circ
\zeta_{I_\pi, t}
\circ
(\id_t \otimes v)
\end{equation}
where we identify elements $v\in K$ as morphisms $v:\Bbbk \to K$, which gives a map $v \in
\cT(1_{\cT} \to 1_\cT\otimes K = \cI_\pi),$ and similarly $w^* \in\cT(\cI_\pi = 1_\cT \otimes K \to
1_\cT)$. Now $\cO(G)\cong \Bbbk^{|G|}$ is an abelian $\Bbbk$-algebra, so for each $t\in \cT$ and
$g\in G$, we have a canonical projector $\chi_g^t \in \cT(t\to t)$. 
The proof of the following lemma
is straightforward.

\begin{lem}
\label{lem:CanonicalO(G)Projectors}
For $t\in \cT$, the projectors $\chi^t_g \in \cT(t\to t)$ satisfy the relations
\begin{itemize}
\item
(direct sum)
$\chi^t_g \circ \chi^t_h = \delta_{g=h} \chi^t_g$ and $\sum_{g\in G} \chi^t_g = \id_t$, and
\item
(compatibility with morphisms)
for all $s\in \cT$ with projectors $\chi^s_g\in \cT(s\to s)$ and all morphisms $f\in \cT(s\to t)$, we have
$\chi_g^s \circ f = f\circ \chi^t_g$.
\end{itemize}
\end{lem}

As $\cT$ is fusion and thus idempotent complete, for $g\in G$, we may define $t_g := \im(\chi^t_g)$.
By the direct sum relation in Lemma \ref{lem:CanonicalO(G)Projectors} we have $t = \bigoplus_{g\in
G} t_g$. Moreover, for all $f\in \cT(s\to t)$, we see that $\cT(s \to t) = \bigoplus_{g\in G}
\cT(s_g \to t_g)$. Thus defining $\cT_g$ to be the subcategory whose objects are of the form $t_g$
for $t\in \cT$, we have $\cT = \bigoplus_{g\in G} \cT_g$, i.e., $\cT$ is $G$-graded as a semisimple
category.

We now claim that this $G$-grading is compatible with the tensor product, i.e., if $s\in \cT_g$ and
$t \in \cT_h$, then $s \otimes t \in \cT_{gh}$. To show this, we observe that the map
\eqref{eq:CanonicalO(G)Inclusion} endows each hom space $\cT(s\to t)$ with an $\cO(G)$-\emph{action}
\begin{equation*}
%\label{eq:CanonicalO(G)Action}
(w^*\otimes v) \vartriangleright f :=
\begin{tikzpicture}[baseline=-.6cm, xscale=-1]
  \begin{knot}[clip width=3]
    \strand (0,-1.5) -- (0,.7);
    \strand (-.4,-.2) -- (.4,.2);
  \end{knot}
	\roundNbox{fill=white}{(0,-.8)}{.25}{0}{0}{$f$}
	\roundNbox{fill=white}{(-.6,-.2)}{.25}{0}{0}{$v$}
	\roundNbox{fill=white}{(.6,.2)}{.25}{0}{0}{$w^*$}
	\node at (.2,-1.3) {\scriptsize{$s$}};
	\node at (.2,-.3) {\scriptsize{$t$}};
	\node at (-.2,.2) {\scriptsize{$\cI_\pi$}};
\end{tikzpicture}
\end{equation*}
such that
\begin{equation}
\label{eq:HopfActionMultiplicationCompatibility}
(w_1^*\otimes v_1)(w_2^*\otimes v_2) \vartriangleright f
=
(w_1^*\otimes v_1)\vartriangleright(w_2^*\otimes v_2) \vartriangleright f
\qquad\qquad
\forall f\in \cT(s\to t).
\end{equation}
Since for all $s,t\in \cT$,
$$
\begin{tikzpicture}[baseline=-.1cm, xscale=-1]
  \begin{knot}[clip width=3]  
    \strand (-.1,-1) -- (-.1,1);
    \strand (.1,-1) -- (.1,1);
    \strand (-.6,-.4) -- (.6,.4);
  \end{knot}
	\roundNbox{fill=white}{(-.8,-.6)}{.25}{0}{0}{$v$}
	\roundNbox{fill=white}{(.8,.6)}{.25}{0}{0}{$w^*$}
	\node at (-.3,-.8) {\scriptsize{$s$}};
	\node at (.3,-.8) {\scriptsize{$t$}};
	\node at (-.6,-.1) {\scriptsize{$\cI_\pi$}};
\end{tikzpicture}
=
\sum_{i}
\begin{tikzpicture}[baseline=-.1cm, xscale=-1]
  \begin{knot}[clip width=3]  
    \strand (-.9,-1) -- (-.9,1);
    \strand (.9,-1) -- (.9,1);
    \strand (-1.4,-.6) -- (-.4,-.2);
  	\strand (1.4,.6) -- (.4,.2);
  \end{knot}
	\roundNbox{fill=white}{(-1.4,-.6)}{.25}{0}{0}{$v$}
	\roundNbox{fill=white}{(-.4,-.2)}{.25}{0}{0}{$e_i^*$}
	\roundNbox{fill=white}{(.4,.2)}{.25}{0}{0}{$e_i$}
	\roundNbox{fill=white}{(1.4,.6)}{.25}{0}{0}{$w^*$}
	\node at (-1.1,.4) {\scriptsize{$s$}};
	\node at (1.1,-.4) {\scriptsize{$t$}};
\end{tikzpicture}
,
$$
our $\cO(G)$-action satisfies
\begin{equation}
\label{eq:HopfActionComultiplicationCompatibility}
(-\otimes_\cT-)\circ\Delta(w^*\otimes v)\vartriangleright(f_1 \otimes_\Bbbk f_2)
=
(w^*\otimes v) \vartriangleright (f_1\otimes f_2)
\qquad
\forall
f_1\in \cT(s_1\to t_1),\,\,
f_2\in \cT(s_2\to t_2).
\end{equation}
This immediately implies that the idempotent $\chi^{st}_g \in \cT(st \to st)$ decomposes as
$$
\chi^{st}_g
=
\sum_{h\in G} \chi^s_{gh} \otimes \chi^t_{h^{-1}}
\qquad
\qquad
\Longrightarrow
\qquad\qquad
\chi^{st}_{gh} \circ (\chi^s_g \otimes \chi^t_h) =\chi^s_g \otimes \chi^t_h
\qquad
\forall g,h\in G.
$$
Thus the $G$-grading on $\cT$ respects the tensor product of $\cT$.
We leave it to the reader to verify this $G$-grading recovers an equivalent $\Rep(G)$-fibered enrichment on $\cT$.

Finally, we show essential surjectivity on 1-morphisms.
Suppose $(H,\eta): (\cT_1,\cI^\sZ_1)\to (\cT_2,\cI^\sZ_2)$ is a 1-morphism in $\cV\FibFusCat_{\sf ff}^{\sf st}$.
It suffices to prove that $H$ is $G$-graded, since $\eta$ is completely determined by $H$ by \eqref{eq:FibFusCat1MorphismCoherence2}.
By using the compatibility of $\eta$ with the half-braidings \eqref{eq:CoherenceForVModuleTensorCategoryHalfBraidings}, we see that $H$ intertwines the $\cO(G)$-actions \eqref{eq:CanonicalO(G)Inclusion} on $\cT_1(t\to t)$ and $\cT_2(H(t)\to H(t))$ for all $t\in \cT_1$.
Thus $H$ maps $\chi_g^t \in \cT_1(t\to t)$ to $\chi_g^{H(t)} \in \cT_2(H(t) \to H(t))$, and $H$ is $G$-graded.
\end{proof}

%\begin{remark}
%Given two 1-morphisms $(G,\gamma) , (H,\eta): (\cT_1,\cF_1^\sZ, \alpha_1)\to (\cT_2, \cF^\sZ_2, \alpha_2)$,
%every monoidal natural isomorphism $\kappa: G \Rightarrow H$ will satisfy the extra action coherence \eqref{eq:CoherenceForVModuleTensorCategory2Morphisms} with $\gamma$ and $\eta$.
%We leave this enjoyable exercise to the reader.
%This means that the forgetful 2-functor $\Forget_\cV$ from the 2-category of $\cV$-fibered enriched monoidal categories to the 2-category $\FusCat$ of linear monoidal categories is fully faithful at the level of 2-morphisms.
%This means that the homotopy fiber $\hoFib_\cT(\Forget_\cV)$ over any monoidal category $\cT$ is 0-truncated, i.e., a equivalent to a set.

We now look at the equivalent 2-subcategory $\Rep(G)\FibFusCat^{\sf st}_{\sf ff}$ such that $\Forget_Z\circ \cI^\sZ=i_\cT\circ F$ on the nose.
On this 2-subcategory, the restriction to cores of the forgetful 2-functor $\Forget_{\Rep(G)}$ is a discrete fibration, since given any fully faithful $\Rep(G)$-fibered enriched fusion category $(\cT_1, \cI_1^\sZ)$ such that $\Forget_Z\circ \cI_1^\sZ =i_{\cT_1}\circ F$ and any monoidal equivalence $H: \cT_1 \to \cT_2$, there is a unique lift $\cI^\sZ_2: \Rep(G) \to Z(\cT_2)$ such that $\Forget_Z\circ \cI^\sZ_2 = i_{\cT_2}\circ F$, which completely determines the necessary action coherence morphism $\eta$ to make $(H,\eta)$ an invertible 1-morphism in $\Rep(G)\FibFusCat^{\sf st}_{\sf ff}$.
We conclude that the
0-truncation of the homotopy fiber $\tau_0(\hoFib_\cT(\Forget_{\Rep(G)}))$ of the forgetful 2-functor $\Forget_{\Rep(G)}:\core(\Rep(G)\FibFusCat^{\sf st}_{\sf ff})\to \core(\FusCat)$
is in canonical bijection with the strict fiber $\stFib_\cT(\Forget_{\Rep(G)})$, which we view as the \emph{set of $\cV$-fibered enrichments of $\cT$}.

%Recall from Example \ref{ex:SetOfGGradings} that the homotopy fiber of the 2-functor $\Forget_G : \core(G\GrdFusCat_{\sf f}) \to \core(\FusCat)$ over $\cT$ is in canonical bijection with the strict fiber over $\cT$, which is called the set of faithful $G$-gradings on $\cT$.
%Thus Theorem \ref{thm:GGraded2CatEquivalence} yields the following corollary.

\begin{cor}\label{cor:fiberedenrichments}
Fix a fusion category $\cT$ and a finite group $G$,
and denote by $F: \Rep(G) \to \Vec = \langle 1_\cT\rangle \subset \cT$ the canonical symmetric fiber functor.
There is a bijective correspondence between the sets of 
\begin{enumerate}
\item
Fully faithful $\Rep(G)$-fibered enrichments $\cI^{\sZ}: \Rep(G) \to Z(\cT)$ such that $\Forget_Z \circ \cI^{\sZ} = i_{\cT}\circ F$ on the nose, where $i_{\cT}: \Vec \hookrightarrow \cT$ is the canonical inclusion $V\mapsto V\otimes 1_\cT$, and
%\item
%$\cO(G)$ actions on $\cT$ satisfying \eqref{eq:HopfActionMultiplicationCompatibility} and \eqref{eq:HopfActionComultiplicationCompatibility} (and suppressed (co)unitality axioms), and
\item
the set of faithful $G$-gradings on $\cT$ (cf.~Example \ref{ex:SetOfGGradings}).
\end{enumerate}
\end{cor}
\begin{proof}
The equivalence $\Phi : G\GrdFusCat_{\sf f} \to \Rep(G)\FibFusCat_{\sf ff}$
such that the triangle \eqref{eq:GGradedTriangle} commutes
gives an equivalence when restricted to cores such that the obvious triangle of cores commtes.
This gives an equivalence of the 0-truncated homotopy fibers over $\cT$
of
$\Forget_G: \core(G\GrdFusCat_{\sf f}) \to \core(\FusCat)$
and
$\Forget_{\Rep(G)}:\core(\Rep(G)\FibFusCat^{\rm st}_{\sf ff})\to \core(\FusCat)$.
Since both $\Forget_G$ and $\Forget_{\Rep(G)}$ restricted to cores are fully faithful and discrete fibrations (the former by Example \ref{ex:SetOfGGradings} and the latter by the discussion before the corollary),
we get canonical bijections
$$
\stFib_\cT(\Forget_G)
\cong
\tau_0(\hoFib_\cT(\Forget_G) )
\cong
\tau_0(\hoFib_\cT(\Forget_{\Rep(G)}) )
\cong
\stFib_\cT(\Forget_{\Rep(G)}).
$$
This completes the proof.
%By the previous theorem, the equivalence $\Phi$ gives an equivalence of homotopy fibers of $\Forget_G$ and $\Forget_{\Rep(G)}$ over a fixed monoidal category $\cT$.
%Since $\Forget_G$ and $\Forget_\cV$ are both fully faithful at the level of 2-morphisms, the equivalent homotopy fibers $\hoFib_\cT(\Forget_G)$ and $\hoFib_\cT(\Forget_{\Rep(G)})$ over any monoidal category $\cT$ is 0-truncated, i.e., a equivalent to a set.
%We already saw that the strict fiber $\stFib_\cT(\Forget_G)$ is equivalent to the 0-truncation $\tau_0(\hoFib_\cT(\Forget_G))$ in Example \ref{ex:SetOfGGradings}, which is the set of faithful gradings on $\cT$.
\end{proof}

% \begin{lem}
% The $G$-grading on $\cT$ induced from the $\Rep(G)$-fibered enrichment $\cI^{\sZ}: \Rep(G) \to Z(\cT)$ is faithful if and only if $\cI^{\sZ}$ is full.
% \end{lem}
% \begin{proof}
% Expanding the definition of $\zeta_{\bullet,\cI_\pi}$, we have
% $$
% \End_{Z(\cT)}(\cI_\pi,\zeta_{\bullet,\cI_\pi}) = \set{f:\cI_\pi \to \cI_\pi\,}{\, f\circ \pi_g = \pi_g \circ f\quad  \forall g\in G \text{ such that }\exists t_g \neq 0}.
% $$
% The result follows.
% \end{proof}

\begin{remark}
By \cite[Cor.~3.6.6]{MR3242743} $G$-gradings on a fusion category are also classified by surjective group homomorphisms from the \emph{universal grading group} $U$ to $G$.
\end{remark}

With these results in hand, we make the following definition.

\begin{defn}
A faithfully $G$-\emph{graded} $\cV$-fusion category is a $\cV$-fusion category $(\cD, \cF^{\scriptscriptstyle Z}_\cD)$
such that $\cD$ is faithfully $G$-graded as an ordinary fusion category, and $\cF_\cD^{\sZ}(\cV)\subseteq \Rep(G)' \subset Z(\cD)$.

A $G$-\emph{extension} of a $\cV$-fusion category $(\cC, \cF^{\scriptscriptstyle Z}_\cC)$ is a faithfully $G$-graded $\cV$-fusion category $(\cD, \cF^{\scriptscriptstyle Z}_\cD)$ together with an equivalence of $\cV$-fusion categories $(\cC,\cF_\cC^\sZ) \cong (\cD_e, \cF_\cD^\sZ)$
(recall $(\Forget_Z\circ \cF_\cD^\sZ)(\cV) \subseteq \cD_e$ by Lemma \ref{lem:LandInTrivialComponent}).
\end{defn}

We close this section with the following observation about $\Rep(G)$-fibered enrichments.
Given a fully faithful braided tensor functor $\Rep(G) \to Z(\cC)$ where $\cC$ is a fusion category, it is not necessarily the case that $\cC$ is $G$-graded.
For example, taking $\cC=\Rep(G)$, the universal grading group of $\cC$ is $\widehat{Z(G)}$.
Note that this enrichment is as far as possible from a $\Rep(G)$-fibered enrichment, since postcomposing the enrichment with the forgetful functor yields an equivalence.
However, $\Rep(G)$ is Morita equivalent to $\Vec(G)$, the quintessential example of a $G$-graded fusion category.
Our next result shows this behavior is generic. 
The proposition below shows that any fusion category with a $\Rep(G)$ enrichment is Morita equivalent to a $G$-graded fusion category whose associated $\Rep(G)$ enrichment (obtained from the canonical equivalence of centers) is fibered. 
This can be interpreted as a partial converse to Corollary \ref{cor:fiberedenrichments}.

\begin{prop}
Suppose $\cC$ is a fusion category and $F: \Rep(G) \to Z(\cC)$ is a fully faithful tensor functor.
Then there exists a faithfully $G$-graded fusion category $\cD$ which is Morita equivalent to $\cC$
such that the associated enrichment $\Rep(G) \to Z(\cC) \cong Z(\cD)$ is a $\Rep(G)$-fibered
enrichment.
\end{prop}
\begin{proof}
Consider the image of $\cO(G)$ inside $Z(\cC)$, which is a connected \'{e}tale algebra, which we
will still denote by $\cO(G)$. Observe that $Z(\cC)_{\cO(G)}$ is a $G$-crossed braided extension of
$Z(\cC)_{\cO(G)}^{\loc}$ by \cite[Thm.~8.24.3]{MR3242743}. Now note that $Z(\cC)_{\cO(G)}^{\loc}
\cong Z(\cC_{\cO(G)})$ by \cite[Thm.~3.20]{MR3039775} where $\cC_{\cO(G)}$ is a multifusion
category, and every center of a multifusion category is also the center of an ordinary fusion
category \cite[Rem.~5.2]{MR3039775}. By \cite{MR2587410}, there is a bijective correspondence
between $G$-extensions of fusion categories $\cF$ and $G$-crossed braided extensions of $Z(\cF)$
which is established by taking the relative center. Thus there is a $G$-graded fusion category $\cD$
whose relative center with respect to its trivial component is $Z(\cC)_{\cO(G)}$. Furthermore, by
\cite{MR2587410}, $Z(\cD) \cong (Z(\cC)_{\cO(G)})^G \cong Z(\cC)$. Hence $\cD$ is Morita equivalent
to $\cC$. Since the forgetful functor $Z(\cD) \to \cD$ factors through $Z(\cC)_{\cO(G)}$, the
$\Rep(G)$-enrichment for $\cD$ is fibered.
\end{proof}

%%%%%%%%%%%%%%%%%%%%%%%%%%%%%%%%%%%%%%%%%%%%%%%%%%%%%%
%%%%%%%%%%%%%%%%%%%%%%%%%%%%%%%%%%%%%%%%%%%%%%%%%%%%%%
%%%%%%%%%%%%%%%%%%%%%%%%%%%%%%%%%%%%%%%%%%%%%%%%%%%%%%
\section{Lifting \texorpdfstring{$\cV$}{V}-enrichment to a fixed \texorpdfstring{$G$}{G}-extension}

For this section, we fix a braided fusion category $\cV$, a $\cV$-fusion category $(\cC, \cF^{\scriptscriptstyle Z})\in \BrFus_1(\cV \to \Vec)$.

\begin{defn}
A $\cV$-enriched $G$-extension of $(\cC, \cF^{\scriptscriptstyle Z})$ is a triple $(\cD, \widetilde{\cF}^{\scriptscriptstyle Z},\alpha)$
such that 
\begin{itemize}
\item 
$\cD = \bigoplus_{g\in G}\cD_g$ is an orginary $G$-extension of $\cC=\cD_e$,
\item
$\cF^{\scriptscriptstyle Z}: \cV \to \Rep(G)'\subset Z(\cD)$
is a $\cV$-enrichment of $\cD$ that lands in the M\"uger centralizer of the canonical copy of $\Rep(G)\subset Z(\cD)$, and
\item
$\alpha$ is a natural isomorphism
\begin{equation}
\label{eq:LiftingVEnrichment}
\begin{tikzcd}
\cV
\ar[d,swap, "\cF^{\scriptscriptstyle Z}"]
\ar[rr, dashed, "\widetilde{\cF}^{\scriptscriptstyle Z}"]
&&
\Rep(G)'%\cap Z(\cD)
\ar[d, hook]
\\
Z(\cC)
\ar[dr, hook, swap, "i"]
\ar[rr, Rightarrow, , shorten <= 3em, shorten >= 3em, "\alpha"]
&&
Z(\cD)
\ar[dl,"\Forget_\cC"]
\\
&
Z_\cC(\cD).
\end{tikzcd}
\end{equation}
where $\Forget_\cC: Z(\cD) \to Z_\cC(\cD)$ denotes the
forgetful functor.
\end{itemize}
Observe that $\cV$-enriched $G$-extensions 
of $(\cC, \cF^{\scriptscriptstyle Z})$
form a 2-groupoid which admits an obvious forgetful 2-functor to the 2-groupoid of ordinary $G$-extensions of $\cC$ as an ordinary fusion category.
This forgetful functor is fully faithful at the level of 2-morphisms and a discrete fibration.
Hence by similar arguments to those in \S\ref{sec:StrictifyFibersFor2Groupoids}, the homotopy fiber over a fixed ordinary $G$-extension $\cD = \bigoplus_{g\in G}\cD_g$ of $\cC$ is 0-truncated and in bijective correspondence with the strict fiber over $\cD$, i.e., set of tensor functors  
$\cF^{\scriptscriptstyle Z}: \cV \to \Rep(G)'\subset Z(\cD)$
such that 
$\Forget_\cC\circ \widetilde{\cF}^{\scriptscriptstyle Z}=i\circ \cF^{\scriptscriptstyle Z}$ on the nose.
%is \emph{equal}
% \footnote{
% While it may seem unnatural to ask for \emph{equality} $\Forget_\cC\circ \widetilde{\cF}^{\scriptscriptstyle Z}=i\circ \cF^{\scriptscriptstyle Z} $, %^{\scriptscriptstyle Z}_\cD$,
% this is exactly the condition which ensures that we have not changed the underlying objects of each $\cF^{\scriptscriptstyle Z}(v)$ when $\cC\subseteq \cD$ for $v\in \cV$, and that the obtained half-braiding for $\widetilde{\cF}^{\scriptscriptstyle Z}(v)$ with $\cD$ is indeed a lift of the half-braiding for $\cF^{\scriptscriptstyle Z}(v)$ with $\cC$.
% } 
%to $i\circ \cF^{\scriptscriptstyle Z}$.
\end{defn}

\begin{remark}
Given an ordinary $G$-extension $\cD= \bigoplus_{g\in G}\cD_g$ of our $\cV$-fusion category $(\cC, \cF^{\scriptscriptstyle Z})$,
choosing a functor
$\widetilde{\cF}^{\scriptscriptstyle Z}: \cV \to \Rep(G)'\subset Z(\cD)$ in the strict fiber over $\cD$ 
is equivalent to choosing for all $v\in \cV$ coherent lifts of the half-braidings for $\cF^{\scriptscriptstyle Z}(v)$ with $\cC$ to all of $\cD$.
\end{remark}

We now 
use the ENO 
extension theory for fusion categories \cite{MR2677836}, together with the results from \cite{MR2587410}
to give several equivalent characterizations of the set of compatible $\cV$-enrichments on a fixed ordinary $G$-extension $\cD$ of our $\cV$-fusion category $(\cC, \cF^{\scriptscriptstyle Z})$.

\subsection{Classification in terms of monoidal 2-functors}
\label{sec:ClassificationByMonoidal2Functors}

% The endomorphism 3-category $\End^{123}(\cC,\cF^{\sZ})$ has
% \begin{itemize}
% \item
% a single 0-morphism $\cC$
% \item
% 1-morphisms $\BrFus_2(\cC\to \cC)$
% \item
% 2-morphisms the 3-morphisms in $\BrFus$,
% and
% \item
% 3-morphisms the 4-morphisms in $\BrFus$.
% \end{itemize}

\begin{defn}
\label{defn:VBrauerPicardGroupoid}
The $\cV$-\emph{Brauer-Picard 2-groupoid} $\uuBrPic^{\mathcal{V}}(\cC,\cF^{\sZ})$ of the $\cV$-fusion
category $(\cC,\cF^{\sZ})$ 
%has objects
%\nn{give defn.}
%
%is the core of the 3-category $\End^{123}(\cC,\cF^{\sZ})$ consisting of
%only invertible morphisms.
%
%Observe that $\uuBrPic^{\mathcal{V}}(\cC,\cF^{\sZ})$ 
is obtained by taking the ordinary unenriched
Brauer-Picard 2-groupoid $\uuBrPic(\cC)$ and imposing extra structure.
\begin{itemize}
\item
objects in $\uuBrPic^{\mathcal{V}}(\cC,\cF^{\sZ})$ are invertible $\cC-\cC$ bimodules $\cM$
equipped with natural isomorphisms $\eta_{a,m}:m\triangleleft F_{\mathcal{C}}(a) \to
F_{\cC}(a)\triangleright m$ satisfying \eqref{eq:BimoduleFunctorEquivalence1},
\eqref{eq:BimoduleFunctorEquivalence2}, and \eqref{eq:BimoduleFunctorEquivalence3}.
\item
1-morphisms are bimodule equivalences $E : \cM \to \cN$ satisfying \eqref{eq:BimoduleFunctorProperty}.
\item
2-morphisms are all bimodule natural isomorphisms.
\end{itemize}

We now endow $\uuBrPic^\cV(\cC,\cF^{\sZ})$ with the structure of a categorical 2-group (3-group) by lifting the monoidal structure on $\uuBrPic(\cC)$.
%This forgetful 2-functor is fully faithful at the level of 2-morphisms. and \nn{closed on products of objects}.
Observe that there is an obvious forgetful 2-functor $\Forget_\cV: \uuBrPic^\cV(\cC,\cF^{\sZ}) \to \uuBrPic(\cC)$ that forgets the extra structure, and $\Forget_\cV$ is fully faithful at the level of 2-morphisms and a discrete fibration.

We now define a monoidal structure on objects in $\uuBrPic^\cV(\cC,\cF^{\sZ})$ as the relative Deligne tensor product in $\uuBrPic(\cC)$ from Definition \ref{defn:RelativeDeligneProduct}, together with the centering morphism defined from the vertical composition on 2-morphisms in $\BrFus$ from \eqref{eq:VerticalCompositionEta}.
The monoidal product of objects in $\uuBrPic(\cC)$ is defined by a universal property, so there is not just one composite; there is a contractible choice.
We observe that when these bimodules are in the image of the forgetful 2-functor $\Forget_\cV$, there is a choice of composite $\cM\boxtimes_\cD \cN$ which is also in the image of $\Forget_\cV$ by construction.
Following \cite{MR2678824}, we get an associator and pentagonator for $\uuBrPic(\cC)$ from the universal property defining the relative Deligne product.
By the universal property, these associators lift to $\uuBrPic^\cV(\cC,\cF^{\sZ})$, and since $\Forget_\cV$ is fully faithful on 2-morphisms, so does the pentagonator.
%We define an associator by taking the same associator from $\uuBrPic(\cC)$; indeed \nn{check it works}.
%As the forgetful functor is fully faithful on 2-morphisms, the pentagonator automatically lifts, and the monoidal structure is automatically coherent.
This also means the forgetful 2-functor $\Forget_\cV$ is automatically a monoidal 2-functor.
\end{defn}

\begin{remark}
We expect that $\uuBrPic^\cV(\cC,\cF^{\sZ})$ is monoidally 2-equivalent to the core of the endomorphism monoidal 2-category 
of the 1-morphism $(\cC,\cF^{\sZ}) \in \BrFus(\cV \to \Vec)$.
We leave this verification to the interested reader.
\end{remark}

Observe that there is a 2-groupoid of monoidal 2-functors $\Hom(\uuG \to \uuBrPic^\cV(\cC,\cF^\sZ))$, and this 2-groupoid admits a 2-functor $U:=(\Forget_\cV)_*$ to the 2-groupoid of monoidal 2-functors $\Hom(\uuG \to \uuBrPic(\cC))$.
By Theorem \ref{thm:ExtensionsAsMonoidal2Functors}, this latter 2-groupoid is equivalent to the 2-groupoid $\Ext(\cC,G)$ of $G$-extensions of $\cC$ as an ordinary fusion category.
Now fixing an ordinary $G$-extension $\cD$ of $\cC$, we get a corresponding monoidal 2-functor $\uupi: \uuG \to \uuBrPic(\cC)$.
Similar to Example \ref{ex:Lift2Functor}, the homotopy fiber $\hoFib_\pi(U)$ is 0-truncated, i.e., a set.
Moreover, since $U$ is a discrete fibration, this set is in bijection to the strict fiber $\stFib_\pi(U)$ whose elements are monoidal 2-functors $\uupi^\cV: \uuG \to \uuBrPic^\cV(\cC,\cF^\sZ)$ such that $\Forget_\cV \circ \uupi^\cV = \uupi$ \emph{on the nose}.
$$
\begin{tikzcd}
&\uuBrPic^\cV(\cC, \cF^\sZ)
\arrow[d,"\Forget_\cV"]
\\ 
\uuG
\arrow[ur, "\uupi^\cV"]
\arrow[r, "\uupi"]
&
\uuBrPic(\cC)
\end{tikzcd}
$$
We call this set the \emph{set of lifts} of $\uupi$ to $\uuBrPic^\cV(\cC,\cF^{\sZ})$.

We now prove a version Theorem \ref{thm:ExtensionsAsMonoidal2Functors} for $\cV$-fusion categories.

\begin{thm}
\label{thm:LiftsToEnrichedBrPic}
Let $(\cC,\cF^\sZ)$ be a $\cV$-fusion category and $\cD$ an ordinary $G$-extension of $\cC$.
Let $\uupi: \uuG \to \uuBrPic(\cC)$ be any monoidal 2-functor corresponding to $\cD$ under the equivalence of 2-groupoids $\Ext(\cC, G)\cong \Hom(\uuG \to \uuBrPic(\cC))$ afforded by Theorem \ref{thm:ExtensionsAsMonoidal2Functors}.
The set of $\cV$-enrichments
$\cF_\cD^{\sZ}: \cV \to Z(\cD)$ compatible with the enrichment $\cF^\sZ: \cV \to Z(\cC)$
is in bijective correspondence with
the set of lifts of $\uupi:\uuG \to \uuBrPic(\cC)$ to $\uuBrPic^\cV(\cC, \cF^\sZ)$.
\end{thm}
\begin{proof}
Suppose we can lift the $\cV$-enrichment of $\cC$ to $\cD$. We define morphisms $\eta_{v,m}:m
\triangleleft F_\cC(v)\rightarrow F_\cC(v)\triangleright m $ for each $m\in \cD_g$, where $F_\cC:
\cV \to Z(\cC) \to \cC$ as follows. A lift $\widetilde{\cF}: \cV \to Z(\cD)$ applied to a $v\in \cV$
can be viewed as $\widetilde{\cF}(v)=(F_{\cC}(v), \sigma_{\bullet,F_\cC(v)})$, where
$\sigma_{\bullet,F_\cC(v)}$ is a half-braiding for $F_\cC(v)$ with $d\in \cD$. We define
$\eta_{v,d}:=\sigma_{d,F_\cC(v)}:d\otimes F_{\cC}(v)\rightarrow F_{\cC}(v)\otimes d$. The fact that
$\widetilde{\cF}:\cV\rightarrow Z(\cD)$ is a braided monoidal functor ensures that $\eta_{v,d}$
makes the diagrams \eqref{eq:BimoduleFunctorEquivalence1}, \eqref{eq:BimoduleFunctorEquivalence2},
and \eqref{eq:BimoduleFunctorEquivalence3}   commute. This means we can lift the image of the
monoidal 2-functor $\uuG \to \uuBrPic(\cC)$ to $\uuBrPic^{\mathcal{V}}(\cC,\cF^\sZ)$ at the level of 1-morphisms. To
lift at the level of 2-morphisms, recall that $\otimes$ induces a bimodule equivalence
$\cD_{g}\boxtimes_{\cC} \cD_{h}\rightarrow \cD_{gh}$. We need to show that this bimodule equivalence
is a morphism in $\uuBrPic^{\mathcal{V}}(\cC,\cF^\sZ)$. Given objects $d_{g}\in \cD_{g}, d_{h}\in \cD_{h}$,
we need to check the following diagram commutes:

\begin{equation}
\label{eq:HalfBraidingEquivalentToBimoduleFunctorCompatibility}
\begin{tikzcd}
    \otimes\left((d_{g}\boxtimes_{\cC} d_{h})\triangleleft F_{\cC}(v)\right)
    =
    d_{g}\otimes (d_{h}\otimes F_{\cC}(v))
    \arrow{r} \arrow{d}
    &
    \otimes \left( F_{\cC}(v) \triangleright (d_{g}\boxtimes_{\cC} d_{h})\right)
    =
    (F_{\cC}(v)\otimes d_{g})\otimes d_{h})
    \arrow{d}
\\
   \otimes (d_{g}\boxtimes_{\cC} d_{h}) \triangleleft F_{\cC}(v)
    =
    ( d_{g}\otimes d_{h}) \otimes F_{\mathcal{C}}(v)
     \arrow{r}
    &
    F_{\mathcal{C}}(v)
    \triangleright
    \left(\otimes(d_{g}\boxtimes_{\cC} d_{h})\right)
    =
    F_{\mathcal{C}}(v)\otimes(d_{g}\otimes d_{h})
 \end{tikzcd}
\end{equation}
where the top isomorphism is that from \eqref{eq:VerticalCompositionEta}.
This now follows immediately from the associativity of a half-braiding.

Conversely,
given a $\uupi^\cV: \uuG \to \uuBrPic^\cV(\cC, \cF^{\sZ})$ such that $\Forget_\cV\circ \uupi^\cV = \uupi$,
we need to extend the half-braiding of $\cF^{\sZ}(v)$ with $\cC$ to all of $\cD$.
We simply use $\eta^g$ on $\cD_g$ as our half-braiding:
$$
\eta^g_{v,m_g} : m_g \triangleleft F_{\cD_g}(v) = m_g \otimes \cF(v) \to \cF(v) \otimes m_g = F_{\cD_g}(v) \triangleright m_g.
$$
Now
one uses the commutativity of
\eqref{eq:BimoduleFunctorEquivalence1}, \eqref{eq:BimoduleFunctorEquivalence2}, \eqref{eq:BimoduleFunctorEquivalence3}
and
\eqref{eq:HalfBraidingEquivalentToBimoduleFunctorCompatibility}
to verify that this is a well-defined half-braiding with all of $\cD$.

Finally, one verifies these two constructions are mutually inverse.
\end{proof}

%%%%%%%%%%%%%%%%%%%%%%%%%%%%%%%%%%%%%%%%%%%%%%%%%%%%%
\subsection{Classification in terms of \texorpdfstring{$G$}{G}-equivariant structures on \texorpdfstring{$\cF^\sZ$}{FZ}}

We now show that given a $\cV$-fusion category $(\cC,\cF^\sZ)$ and an ordinary $G$-extension $\cD$ of $\cC$,
the set of possible compatible 
$\cV$-enrichments on $\cD$ is in canonical bijection with
$G$-equivariant structures
on the classifying functor
$\cF^{\scriptscriptstyle Z}:\mathcal{V}\rightarrow Z(\cC)$
with respect to the categorical action
$\urho:\underline{G}\rightarrow \uAut^{\rm br}_{\otimes}(Z(\cC))$
induced from the $G$-extension $\cC \subseteq \cD= \bigoplus_g \cD_g$.
(Recall from Example \ref{ex:LiftToEquivariantization} that lifts of $\cF^\sZ: \cV \to Z(\cC)$ to $Z(\cC)^G$ naturally form a set.)

\begin{thm}
\label{thm:ClassificationOfVEnrichmentsAsEquivariantLifts}
The lifts $\widetilde{\cF}^{\scriptscriptstyle Z}: \cV \to Z(\cD)$ which are compatible with $\cF^{\scriptscriptstyle Z}: \cV \to Z(\cC)$ are in bijective correspondence with lifts $\widetilde{\cF}^{\scriptscriptstyle Z} : \cV \to Z(\cC)^G$ which satisfy
$\Forget_G\circ \widetilde{\cF}^{\scriptscriptstyle Z} = \cF^{\scriptscriptstyle Z}$, where $\Forget_G: Z(\cC)^G \to Z(\cC)$ forgets the $G$-equivariant structure.
\begin{equation}
\label{eq:EquivariantLiftsOfCentralFunctor}
\begin{tikzcd}
\cV
\ar[rr, dashed, "\widetilde{\cF}^{\scriptscriptstyle Z}"]
\ar[dr, swap, "{\cF}^{\scriptscriptstyle Z}"]
&&
Z(\cC)^G
\ar[dl,"\Forget_G"]
\\
&
Z(\cC)
\end{tikzcd}
\end{equation}
\end{thm}
\begin{proof}
We insert the commutative diagram \eqref{eq:RelativeCenterEquivariantization} based on
\cite{MR2587410} into \eqref{eq:LiftingVEnrichment} to get the following diagram:
\begin{equation}
\label{eq:LiftingVEnrichmentWithGNN}
\begin{tikzcd}
\cV
\ar[d, swap, "\cF^{\scriptscriptstyle Z}"]
\ar[rr, dashed, bend right=-30, "\widetilde{\cF}^{\scriptscriptstyle Z}"]
&
Z(\cC)^G
\ar[d, hook]
\ar[r, leftrightarrow, "\cong"]
\ar[dl, swap, "\Forget_G"]
&
\Rep(G)'%\cap Z(\cD)
\ar[d, hook]
\\
Z(\cC)
\ar[dr, hook, swap, "i"]
&
Z_\cC(\cD)^G
\ar[r, leftrightarrow, "\cong"]
\ar[d, "\Forget_G"]
&
Z(\cD)
\ar[dl,"\Forget_\cC"]
\\
&
Z_\cC(\cD)
\end{tikzcd}
\end{equation}
We see that the set of lifts  $\widetilde{\cF}^{\scriptscriptstyle Z}: \cV \to \Rep(G)'\cap Z(\cD)$
which are compatible with $\cF^{\scriptscriptstyle Z}$ are in bijective correspondence with lifts
$\widetilde{\cF}^{\scriptscriptstyle Z} : \cV \to Z(\cC)^G$ which satisfy $i\circ \Forget_G\circ
\widetilde{\cF}^{\scriptscriptstyle Z} = i\circ \cF^{\scriptscriptstyle Z}$. Since $i$ is faithful
on both objects and morphisms, we can cancel it from the left on both sides of the equation, and the
result follows.
\end{proof}

Thus to classify enriched extensions, we must solve the equivariant lifting  problem for the data given by the initial enrichment and the extension.
In other words, given an (oplax) braided (strongly unital) monoidal functor
$\cF^{\scriptscriptstyle Z}:\mathcal{V}\rightarrow Z(\cC)$
and a categorical action
$\urho:\underline{G}\rightarrow \uAut^{\rm br}_{\otimes}(Z(\cC))$,
we need to find all the $G$-equivariant structures on $\cF^{\scriptscriptstyle Z}$.
We will formalize this notion in Definition \ref{defn:GEquivariantStructureOnFunctor} in the next section.

%%%%%%%%%%%%%%%%%%%%%%%%%%%%%%%%%%%%%%%%%%%%%%%%%%%%%
%%%%%%%%%%%%%%%%%%%%%%%%%%%%%%%%%%%%%%%%%%%%%%%%%%%%%
%%%%%%%%%%%%%%%%%%%%%%%%%%%%%%%%%%%%%%%%%%%%%%%%%%%%%
\section{The equivariant functor lifting problem}

In this section, we study the equivariant functor lifting problem, showing lifts are in bijection
with splittings of a certain exact sequence. Our approach is similar to \cite[\S3]{MR3933137}. We do
so in greater generality than needed for \eqref{eq:EquivariantLiftsOfCentralFunctor} above, since
our results are significantly more general.

For this section, $\cV, \cW$ will denote linear monoidal categories (which are not necessarily braided!) and $(\cF,\varphi,\varepsilon): \cV \to \cW$ denotes an oplax monoidal functor (which need not be strongly unital!), where $\varphi = \{\varphi_{u,v}:\cF(uv) \to \cF(u)\cF(v)\}_{u,v\in\cV}$ is the oplaxitor and $\varepsilon: \cF(1_\cV) \to 1_\cW$ is the counit.

\begin{assume}
\label{assume:ConnectedCoalgebra}
Notice that $F(1_\cV)\in \cW$ is a coalgebra object with comultiplication $\Delta := \varphi_{1_\cV, 1_\cV}$ and counit $\varepsilon$.
For this section, we assume $\cF(1_\cV)$ is connected, i.e., $\cW(\cF(1_\cV)\to 1_\cW)= \bbC \varepsilon$.
\end{assume}

We further suppose
$(\underline{\rho},\mu): \underline{G} \to \uAut_\otimes(\cW)$ is a categorical action of the finite group $G$.
We write $g = \rho_g$ for notational simplicity, and we write $\psi^g$ for its tensorator.
Our convention for the tensorator $\mu$ for $\underline{\rho}$ is $\mu_{g,h}:g\circ h \Rightarrow gh$.

%%%%%%%%%%%%%%%%%%%%%%%%%%%%%%%%%%%%%%%%%%%%%%%%%%%%%%%%%%%%%%%%%%%
\subsection{The first obstruction}

\begin{defn}
We consider the following categorical groups.
\begin{itemize}
\item
$\uAut_{\otimes}(\mathcal{W})$ is the categorical group of (strong) monoidal auto-equivalences of $\mathcal{W}$.
Thought of as a monoidal category, objects are monoidal auto-equivalences of $\mathcal{W}$, and morphisms are monoidal natural isomorphisms.

\item
$\uAut_{\otimes}(\mathcal{W} | \cF)$ is the categorical group defined as follows: objects are
triples $(\alpha, \psi^\alpha, \lambda^\alpha)$, where $(\alpha,\psi^\alpha) \in \uAut_\otimes(\cW)$
is a monoidal auto-equivalence of $\mathcal{W}$ (here, $\psi^\alpha$ is the tensorator of $\alpha$), and $\lambda^\alpha: \cF\Rightarrow \alpha\circ
\cF$ is an (oplax) monoidal natural isomorphism. The 1-composition is strict and defined as
$$
(\alpha, \psi^\alpha, \lambda^\alpha)
\circ
(\beta, \psi^\beta, \lambda^\beta)
:=
(\alpha \circ \beta, \psi^\alpha \circ \alpha(\psi^\beta), \lambda^\alpha \circ \alpha(\lambda^\beta)).
$$
The 2-morphisms $\eta: (\alpha, \psi^\alpha, \lambda^\alpha) \Rightarrow (\beta, \psi^\beta,
\lambda^\beta)$ are all monoidal natural isomorphisms $\eta: (\alpha, \psi^\alpha) \Rightarrow
(\beta, \psi^\beta)$ such that $(\eta\circ \id_F)\circ \lambda^\alpha = \lambda^\beta$.

\item
$\uStab_{\otimes}(\cF)$ is the full categorical subgroup of $\uAut_{\otimes}(\mathcal{W})$ generated
by the image of $\uAut_{\otimes}(\mathcal{W} | \cF)$ under the forgetful functor
$(\alpha,\psi^\alpha, \lambda^\alpha)\mapsto (\alpha,\psi^\alpha)$.

\end{itemize}
\end{defn}

\begin{defn}
\label{defn:GEquivariantStructureOnFunctor}
Let $\rho: \underline{G}\rightarrow \uAut_{\otimes}(\mathcal{W})$, $g\mapsto \rho_g$ be a
categorical action, and $\cF:\mathcal{V}\rightarrow \mathcal{W}$ an oplax monoidal functor. A
$G$-\textit{equivariant structure} on $\cF$ is a lifting
\begin{equation}
\label{eq:GEquivariantLifts}
\begin{tikzcd}
&
\uAut_{\otimes}(\cW | \cF) \arrow[d,"\Forget_\cF"]
 \\
\uG
\arrow[ur, dashed, "\widetilde{\urho}"]
\arrow[r, swap,"\underline{\rho}"]
&
\uAut_{\otimes}(\mathcal{W})
\end{tikzcd}
\end{equation}
which satisfies $\Forget_\cF\circ \widetilde{\underline{\rho}}=\underline{\rho} $ on the nose.
\end{defn}

Hence in order to find a lifting $\widetilde{\urho}: \uG \to \uAut(\cW|\cF)$, it is necessary that
for each $g\in G$,  there exists a monoidal natural isomorphism $\lambda^g: \cF \Rightarrow g\circ
\cF$. We call the existence of such a $\lambda^g$ for each $g\in G$ the \emph{first obstruction} to
the equivariant functor lifting problem. We say \emph{the first obstruction vanishes} if such a
$\lambda^g$ exists for each $g\in G$.

%%%%%%%%%%%%%%%%%%%%%%%%%%%%%%%%%%%%%%%%%%%%%%%%%%%%%%%%%%%%%%%%%%%
\subsection{The second obstruction}

We now assume that the first obstruction to the equivariant lifting problem vanishes, i.e., for
every $g\in G$, there exists a monoidal natural isomorphism $\lambda^g: \cF \Rightarrow g\circ \cF$.
We now give a necessary and sufficient condition for the isomorphisms $(\lambda^g)_{g\in G}$ to
assemble to a lift $\widetilde{\urho}:\uG \to \uAut_\otimes(\cW|\cF)$. We call this condition the
\emph{second obstruction} to the equivariant functor lifting problem.

Recall that the adjoint to the forgetful functor $\Forget_G : \cW^G \to \cW$
is $I : \cW \to \cW^G$ by $w \mapsto \bigoplus g(w)$ and
$f\in \cW(w_1 \to w_2)$ maps to $I(f)_{g,h}:=\delta_{g,h}\cdot g(f)$.
Observe that given $w\in \cW$, $f: I(w) \to I(w)$ is $G$-equivariant if and only if
the following diagram commutes for all $g,h,k\in G$:
\begin{equation}
\label{eq:EquivariantMapI(w)}
\begin{tikzcd}
g(k(w))
\ar[d, "{g(f_{h,k})}"]
\ar[r,"{\mu^w_{g,k}}"]
&
(gk)(w)
\ar[d,"{f_{gh,gk}}"]
\\
g(h(w))
\ar[r,"{\mu^w_{g,h}}"]
&
(gh)(w)
\end{tikzcd}
\qquad\qquad
\forall g,h,k\in G
\end{equation}
where $f_{h,k}:k(w)\to h(w)$ is the $(h,k)$-component map of $f$.
The functor $I$ is endowed with an oplax monoidal structure $\nu^I_{w_1,w_2} \in \cW^G( I(w_1 \otimes w_2) \to I(w_1)\otimes I(w_2))$ given componentwise by
$$
\bigoplus_{g\in G}\psi^g_{w_1,w_2}
:
\bigoplus_{g\in G} g(w_1\otimes w_2)
\xrightarrow {\psi^g_{w_1,w_2}}
\bigoplus_{g\in G}g(w_1)\otimes g(w_2)
\subseteq
\bigoplus_{g,h\in G} g(w_1)\otimes h(w_2)
\cong
I(w_1)\otimes I(w_2).
$$

\begin{remark}
In addition to $\cF(1_\cV)$ being a coalgebra with comultiplication $\Delta$ (see Assumption
\ref{assume:ConnectedCoalgebra}), notice that $(I\circ \cF)(1_\cV)\in \cW^G$ is also a coalgebra
object with comultiplication given on components by
$$
\Lambda_k^{g,h} := \delta_{g=h}\delta_{g=k}\cdot \psi^g_{\cF(1_\cV), \cF(1_\cV)}\circ g(\Delta) : k(\cF(1_\cV)) \to g(\cF(1_\cV)) \otimes h(\cF(1_\cV))
$$
and counit given on components by $\varepsilon_g := g(\varepsilon^\cF) : g(\cF(1_\cV))\to 1_\cW$.
\end{remark}

We define $\iota:\Aut_\otimes(\cF) \to \Aut_{ \otimes}(I\circ \cF)$ by $\iota(f)^{v} := I(f^v) \in
\cW^G(I(\cF(v))\to I(\cF(v)))$. To verify that $\iota(f)$ is oplax monoidal, we see the outside
square of the following diagram commutes, as the inner squares both commute:
\[
\begin{tikzcd}
I(\cF(v_1\otimes v_2))
\ar[rr,"I(\varphi^{v_1,v_2})"]
\ar[d,"I(f^{v_1\otimes v_2})"]
&&
I(\cF(v_1)\otimes \cF(v_2))
\ar[rr,"\nu^{F(v_1),F(v_2)}"]
\ar[d,"I(f^{v_1}\otimes f^{v_2})"]
&&
I(\cF(v_1))\otimes I(\cF(v_2))
\ar[d,"I(f^{v_1})\otimes I(f^{v_2})"]
\\
I(\cF(v_1\otimes v_2))
\ar[rr,"I(\varphi^{v_1,v_2})"]
&&
I(\cF(v_1)\otimes \cF(v_2))
\ar[rr,"\nu^{\cF(v_1),\cF(v_2)}"]
&&
I(\cF(v_1))\otimes I(\cF(v_2)).
\end{tikzcd}
\]

The following lemma is similar to \cite[Lem.~3.2]{MR3933137}.
We provide a proof for completeness and convenience of the reader.

\begin{lem}
\label{lem:UniqueNonzeroComponent}
Suppose $\eta\in \Aut_\otimes(I\circ \cF)$.
\begin{enumerate}
\item
For $h,k\in G$, $\eta^{v}_{h,k}: k(\cF(v)) \to h(\cF(v))$ is equal to
$\eta^{v}_{h,k}
=
\mu_{k,k^{-1}h}^{\cF(v)}\circ k(\eta^{v}_{k^{-1}h,e})\circ (\mu_{k,e}^{\cF(v)})^{-1}$.
Hence $\eta^{v}$ is completely determined by its components
$\eta^{v}_{g,e}: g(\cF(v)) \to \cF(v)$ for $v\in \cV$.
\item
There is a unique $g\in G$ such that $\eta^{1_\cV}_{g,e}\neq 0$, and $\eta^{1_\cV}_{g,e}: \cF(1_\cV) \to g(\cF(1_\cV))$ is a coalgebra isomorphism.
\item
For every $h\in G$, there are unique $g,k\in G$ such that $\eta^{v}_{g,h}\neq 0\neq \eta^{v}_{h,k}$ for all $v\in \cV$.
These $g,k$ are independent of $v\in \cV$.
\end{enumerate}
\end{lem}
\begin{proof}
To prove (1), since $\eta^{v}: I(\cF(v))\Rightarrow I(\cF(v))$ is $G$-equivariant, replacing $h,k$
by $g^{-1}h, g^{-1}k$ respectively in \eqref{eq:EquivariantMapI(w)} for $f=\eta^{v}$ gives
$$
\eta^{v}_{h,k}\circ \mu_{g,g^{-1}k}^{\cF(v)}
=
\mu_{g,g^{-1}h}^{\cF(v)}\circ g(\eta^{v}_{g^{-1}h,g^{-1}k})
\qquad\qquad
\forall g,h,k\in G
$$
Now setting $g=k$ gives the desired formula.

To prove (2), we first note that for each $g\in G$, there is a scalar $\gamma_g \in \bbC$ such that
$g(\varepsilon) \circ \eta^{1_\cV}_{g,e} = \gamma_g \cdot \varepsilon  \in \cC(\cF(1_\cV)\to 1_\cW)
= \bbC \cdot \varepsilon$. Looking at the $e$-component of the counitality axiom
$$
\varepsilon^I\circ I(\varepsilon^\cF)
=
\varepsilon^I\circ I(\varepsilon^\cF) \circ \sigma^{1_\cV} \in \cW^G(I(\cF(1_\cV)) \to 1_{\cW})
$$
gives us the identity
$$
\varepsilon
=
\sum_{h\in G} h(\varepsilon^\cF)\circ \eta^{1_\cV}_{h,e}
=
\left(\sum_{h\in G} \gamma_h\right)\varepsilon,
$$
which implies $\sum_h \gamma_h = 1$.
Fix $h\in G$ such that $\gamma_h \neq 0$.
For $g\neq h$, looking at the component $\Lambda^{h,g}_e: \cF(1_\cV)\to h(\cF(1_\cV))\otimes g(\cF(1_\cV))$ yields the identity
$$
(\eta^{1_\cV}_{h,e}\otimes \eta^{1_\cV}_{g,e})\circ \psi^{\cF(1_\cV),\cF(1_\cV)}_{e} \circ \Delta
=
\delta_{h=g}
\psi^{\cF(1_\cV),\cF(1_\cV)}_g
\circ
g(\Delta)
\circ
\eta^{1_\cV}_{g,e}
=
0.
$$
Postcomposing with $h(\varepsilon)\otimes \id_{g(\cF(1_\cV))}$ yields
$$
0
=
((\eta^{1_\cV}_{h,e}\circ h(\varepsilon^\cF)) \otimes \eta^{1_\cV}_{g,e})
\circ
\Delta
=
\gamma_h
\cdot
(\varepsilon^\cF\otimes \eta^{1_\cV}_{g,e})\circ \Delta
=
\gamma_h\cdot
\eta^{1_\cV}_{g,e}.
$$
Since $\gamma_h \neq 0$, we conclude $\eta^{1_\cV}_{g,e}=0$ whenever $g\neq h$, proving (2). Notice
this also proves $\gamma_h = 1$. That $\eta^{1_\cV}_{g,e}: \cF(1_\cV)\to g(\cF(1_\cV))$ is a
coalgebra isomorphism follows immediately by looking at components as above.

Now (3) follows by (1) and (2) using monoidality of $\eta$. Indeed, for $v\in \cV$, we have $v =
1_\cV \otimes v$ (suppressing unitors), so the components of $\eta^v \in \End_{\cW^G}((I\circ
\cF)(v)=\bigoplus_g g(\cF(v)))$ satisfy the following commuting diagram below:
$$
\begin{tikzcd}
h(\cF(v))
\ar[rrr, "\psi^h_{\cF(1_\cV), \cF(v)}\circ h(\varphi_{1_\cV, v})"]
\ar[d, "\eta^v_{g,h}"]
&&&
h(\cF(1_\cV))\otimes h(\cF(v))
\ar[d, "\eta^{1_\cV}_{g,h}\otimes \eta^v_{g,h}"]
\\
g(\cF(v))
\ar[rrr, "\psi_{\cF(1_\cV), \cF(v)}^g\circ g(\varphi_{1_\cV, v})"]
&&&
g(\cF(1_\cV))\otimes g(\cF(v)).
\end{tikzcd}
$$
Notice that the map $\psi_{\cF(1_\cV), \cF(v)}^g\circ g(\varphi_{1_\cV, v})$ has a left inverse for
every $v\in \cV$, namely $(g(\varepsilon)\otimes \id_{\cF(v)})\circ(\psi_{\cF(1_\cV),
\cF(v)}^g)^{-1}$. This implies that $\eta^v_{g,h} = 0$ whenever $\eta^{1_\cV}_{g,h} = 0$.
\end{proof}

\begin{lem}
\label{lem:ConstructionOfPi}
The function
$\pi: \Aut_{ \otimes}(I\circ \cF) \to G$
given by
setting $\pi(\eta)$ to be the unique $g$ such that $\eta^{1_\cV}_{g^{-1},e}\neq 0$
gives a well-defined group homomorphism.
\end{lem}
\begin{proof}
Suppose $\eta, \xi\in \Aut^{ \otimes}_{\mathcal{W}^G}(I\circ F)$, and consider $\eta\circ \xi$.
Then $\pi(\eta \circ \xi)$ is the unique element $g\in G$ such that $(\eta\circ \xi)^{1_\cV}_{g^{-1},e}\neq 0$.
We calculate that
$$
(\eta \circ \xi)^{1_\cV}_{g^{-1},e}
=
\sum_{h\in G}\eta^{1_\cV}_{g^{-1},h}\circ \xi^{1_\cV}_{h,e}
=
\eta^{1_\cV}_{g,\pi(\xi)^{-1}}\circ \xi^{1_\cV}_{\pi(\xi)^{-1},e}.
$$
By \eqref{eq:EquivariantMapI(w)}, we see that $\eta^{1_\cV}_{g^{-1},\pi(\xi)^{-1}}\neq 0$ if and
only if $\eta^{1_\cV}_{\pi(\xi)g^{-1}, e}\neq 0$. Hence $(\pi(\xi)g^{-1})^{-1}=\pi(\eta)$, which
immediately implies $\pi(\eta\circ \xi)=g=\pi(\eta)\cdot \pi(\xi)$.
\end{proof}

\begin{lem}
\label{lem:PiPreimage}
For every $\eta \in \pi^{-1}(g^{-1})$, $\theta_v:= \eta_{g,e}^v :\cF(v) \to g(\cF(v))$ gives an
monoidal natural isomorphism $\theta:\cF \Rightarrow g\circ \cF$.
Moreover, every monoidal natural isomorphism $\cF \Rightarrow g\circ \cF$ arises in this way.
Hence $\pi^{-1}(g^{-1})$ is in bijective correspondence with monoidal natural isomorphisms $\theta:\cF \Rightarrow g\circ \cF$.
\end{lem}
\begin{proof}
First, if $\eta \in \pi^{-1}(g^{-1})$, then the following diagram commutes for all $g \in G$ as $\eta$ is an oplax monoidal automorphism of $I\circ \cF$:
$$
\begin{tikzcd}
\cF(uv)
\ar[rrr, "\varphi_{u, v}"]
\ar[d, "\eta^{uv}_{g,e}"]
&&&
\cF(u)\otimes \cF(v)
\ar[d, "\eta^{u}_{g,e}\otimes \eta^v_{g,e}"]
\\
g(\cF(uv))
\ar[rrr, "\psi_{\cF(u), \cF(v)}^g\circ g(\varphi_{u, v})"]
&&&
g(\cF(u))\otimes g(\cF(v)).
\end{tikzcd}
$$
Notice this is exactly the condition that $\theta: \cF \Rightarrow g\circ \cF$ is oplax monoidal.
Conversely, if $\theta: \cF \Rightarrow g\circ \cF$ is an monoidal natural isomorphism,
then defining
$$
\eta^v_{h,k}:=
\delta_{g=k^{-1}h}
\cdot
\mu_{k,g}^{\cF(v)}\circ k(\theta_v)\circ (\mu_{k,e}^{\cF(v)})^{-1}
$$
gives a well-defined $\eta \in \pi^{-1}(g^{-1})$ such that $\eta^v_{g,e}=\theta_v$ by construction.
\end{proof}

\begin{prop}
The following sequence is exact:
\begin{equation}
\label{eq:ExactSequenceEquivariantFunctor}
\begin{tikzcd}
1
\ar[rr]
&&
\Aut_\otimes(\cF)
\ar[rr, "\iota"]
&&
\Aut_\otimes(I\circ \cF)
\ar[rr, "\pi"]
&&
G
\ar[rr]
&&
1.
\end{tikzcd}
\end{equation}
\end{prop}
\begin{proof}
The map $\iota$ is injective by definition. The map $\pi$ is surjective by Lemma
\ref{lem:PiPreimage}. To see $\im(\iota) = \ker(\pi)$, if $\eta\in \ker(\pi)$, then each $\eta^v$ is
determined by $\theta^v:=\eta^v_{e,e}:\cF(v)\to \cF(v)$ by Lemma \ref{lem:UniqueNonzeroComponent},
and $\theta : \cF\Rightarrow \cF$ is a monoidal natural isomorphism such that $\iota(\theta)=\eta$.
\end{proof}

\begin{thm}
\label{thm:ClassificationOfLiftingsViaSplittings}
The set of $G$-equivariant structures on $\cF$ as in \eqref{eq:GEquivariantLifts}
is in bijective correspondence with splittings of the exact sequence \eqref{eq:ExactSequenceEquivariantFunctor}.
\end{thm}
\begin{proof}
Suppose $\widetilde{\underline{\rho}}$ is a lift of $\underline{\rho}$, and denote
$\widetilde{\underline{\rho}}(g)=(g, \lambda_g)$, where $\lambda^g : \cF \Rightarrow g \circ \cF$ is
a monoidal natural isomorphism. We get a splitting $\sigma: G \to \Aut_\otimes(I\circ \cF)$ by
mapping $g^{-1}$ to the element corresponding to $\lambda^g$. Conversely, given a splitting
$\sigma$, $\sigma(g^{-1}) \in \pi^{-1}(g^{-1})$ gives an monoidal natural isomorphism $\lambda^{g}:=
\sigma(g^{-1})_{g,e}: \cF \Rightarrow g\circ \cF$. One now verifies that
$\widetilde{\underline{\rho}}(g):=(g, \lambda_g)$ is the desired lift. These two constructions are
clearly mutually inverse.
\end{proof}

%%%%%%%%%%%%%%%%%%%%%%%%%%%%%%%%%%%%%%%%%%%%%%%%%%%%%%
\subsection{The braided case}

We now assume $\cV,\cW$ are braided monoidal categories and $\cF: \cV \to \cW$ is an oplax braided monoidal functor.
We again use Assumption \ref{assume:ConnectedCoalgebra} that $\cF(1_\cV)$ is a connected coalgebra in $\cW$.

\begin{defn}
We consider the categorical groups
\begin{itemize}
\item
$\uAut_{\otimes}^{\rm br}(\cW)$ is the full categorical subgroup of $\uAut_{\otimes}(\mathcal{W})$
whose objects are braided (strong) monoidal auto-equivalences of $\cW$.
Observe that if $(\alpha, \psi^\alpha)\in \uAut_{\otimes}^{\rm br}(\cW)$, $(\gamma, \psi^\gamma)\in \uAut_{\otimes}(\mathcal{W})$, and
$\eta:(\alpha, \psi^\alpha)\Rightarrow (\gamma, \psi^\gamma)$
is a  monoidal natural isomorphism,
then $(\gamma, \psi^\gamma) \in \uAut_{\otimes}^{\rm br}(\cW)$,
as the back face of the following diagram commutes.
$$
\begin{tikzcd}
&
\gamma(u\otimes v)
\ar[rr, "\psi^\gamma_{u,v}"]
\ar[dd, near end, swap, "\gamma(\beta^\cV_{u,v})"]
&&
\gamma(u)\otimes\gamma(v)
\ar[dd, "\beta^\cW_{\gamma(u),\gamma(v)}"]
\\
\alpha(u\otimes v)
\ar[rr, crossing over, near end, "\psi^\alpha_{u,v}"]
\ar[ur, "\eta_{uv}"]
\ar[dd, swap, "\alpha(\beta^\cV_{u,v})"]
&&
\alpha(u)\otimes\alpha(v)
\ar[ur, swap, "\eta_u\otimes\eta_v"]
\\
&
\gamma(v\otimes u)
\ar[rr, near end, "\psi^\gamma_{v,u}"]
&&
\gamma(v)\otimes\gamma(u)
\\
\alpha(v\otimes u)
\ar[ur, "\eta_{vu}"]
\ar[rr, "\psi^\alpha_{v,u}"]
&&
\alpha(v)\otimes\alpha(u)
\ar[from=uu, crossing over, near end, swap, "\beta^\cW_{\alpha(u),\alpha(v)}" ]
\ar[ur, swap, "\eta_v\otimes\eta_u"]
\end{tikzcd}
$$
Indeed, the left face commutes since $\eta$ is natural,
the right face commutes since $\beta^\cW$ is natural,
the top and bottom faces commute since $\eta$ is monoidal,
and the front face commutes since $\alpha$ is braided.
We conclude the back face must also commute.

\item
$\uAut_{\otimes}^{\rm br}(\mathcal{W} | \cF)$ is a the full categorical subgroup of
$\uAut_{\otimes}(\mathcal{W} | \cF)$ whose objects are
triples $(\alpha, \psi^\alpha, \lambda^\alpha)$, where
$(\alpha,\psi^\alpha) \in \uAut_\otimes^{\rm br}(\cW)$.

\item
$\uStab_{\otimes}^{\rm br}(\cF)$ is the full categorical subgroup of
$\uAut_{\otimes}^{\rm br}(\cW)$ generated by the image of $\uAut_{\otimes}^{\rm br}(\cW | \cF)$ under the forgetful functor
$(\alpha,\psi^\alpha, \lambda^\alpha)\mapsto (\alpha,\psi^\alpha)$.
\end{itemize}

\end{defn}

In this setting, we make the following definition.

\begin{defn}
\label{defn:GEquivariantBraidedStructureOnFunctor}
Let $\rho: \underline{G}\rightarrow \uAut_{\otimes}^{\rm br}(\mathcal{W})$ be a categorical action, and $\cF:\mathcal{V}\rightarrow \mathcal{W}$ an oplax braided monodal functor.
A $G$-\textit{equivariant structure} on $\cF$ is a lifting
\begin{equation}
\label{eq:GEquivariantBraidedLifts}
\begin{tikzcd}
&
\uAut_{\otimes}^{\rm br}(\cW | \cF) \arrow[d,"\Forget_\cF"]
 \\
\uG
\arrow[ur, dashed, "\widetilde{\urho}"]
\arrow[r, swap,"\underline{\rho}"]
&
\uAut_{\otimes}^{\rm br}(\mathcal{W})
\end{tikzcd}
\end{equation}
which satisfies $\Forget_\cF\circ \widetilde{\underline{\rho}}=\underline{\rho} $ on the nose.
\end{defn}

Since
$\pi_2(\uAut_{\otimes}^{\rm br}(\mathcal{W})) = \pi_2 (\uAut_{\otimes}(\mathcal{W}))$
and
$\pi_2(\uAut_{\otimes}^{\rm br}(\mathcal{W}|F)) = \pi_2 (\uAut_{\otimes}(\mathcal{W}|F))$,
$G$-equivariant lifts as in \eqref{eq:GEquivariantBraidedLifts} are again in bijective correspondence with splittings of the exact sequence \eqref{eq:ExactSequenceEquivariantFunctor}.

%%%%%%%%%%%%%%%%%%%%%%%%%%%%%%%%%%%%%%%%%%%%%%%%%%%%%%
%%%%%%%%%%%%%%%%%%%%%%%%%%%%%%%%%%%%%%%%%%%%%%%%%%%%%%
%%%%%%%%%%%%%%%%%%%%%%%%%%%%%%%%%%%%%%%%%%%%%%%%%%%%%%
\section{Examples}
\label{sec:Examples}

In this section, we work out examples of our main Theorems
\ref{thm:ClassificationOfVEnrichmentsAsEquivariantLifts} and
\ref{thm:ClassificationOfLiftingsViaSplittings} above in the $\cV$-fusion setting.

%\begin{lem}
%\label{lem:RestrictedAutsCentralizer}
%Suppose $\cV,\cW$ are braided fusion categories and $\cF : \cV \to \cW$ is a braided tensor functor.
%There is a canonical monoidal functor
%$$
%\uAut_\otimes^\br(\cW| \cF)
%\to
%\uAut_\otimes^\br(\cF(\cV)'\cap \cW)
%$$
%\end{lem}
%\begin{proof}
%\nn{todo}
%\end{proof}

%%%%%%%%%%%%%%%%%%%%%%%%%%%%%%%%%%%%%%%%%%%%%%%%%%%%%%
\subsection{Fully faithful enrichment}

Suppose $(\cC, \cF^{\scriptscriptstyle Z})$ is a $\cV$-fusion category such that
$\cF^{\scriptscriptstyle Z}$ is fully faithful. This type of example is particularly important,
since every enrichment can be ``pushed forward" to a fully faithful enrichment by considering the
enrichment over the full subcategory generated by the image of $\cV$ in $Z(\cC)$. We will see that
in the fully faithful setting, the $G$-action on the normal subgroup $\Aut_{\otimes}(\cF^\sZ)$ is
trivial, and thus splitting of the short exact sequence \eqref{eq:ExactSequenceEquivariantFunctor}
becomes a 2-cocycle obstruction.
%Suppose $\cV$ is non-degenerately braided fusion category and $(\cC, \cF^{\scriptscriptstyle Z})$ is a $\cV$-fusion category such that $\cF^{\scriptscriptstyle Z}$ is fully faithful.
%By \nn{}, we have $Z(\cC) \cong \cU \boxtimes \cV$ for some other non-denegerately braided fusion category $\cU$ such that $\cF^{\scriptscriptstyle Z}$ is identified with $v\mapsto 1\boxtimes v$.

Now suppose $\cD$ is any $G$-graded extension of $\cC$ as an ordinary fusion category, so we get a
categorical action $\urho: \uG \to \uAut_\otimes^{\br}(Z(\cC))$. Assume that $\urho$ passes the
first obstruction, so that for each $g\in G$, there exists a monoidal natural isomorphism
$\lambda^g: \cF \Rightarrow g\circ \cF$. By a direct computation, we see that
\begin{equation}
\label{eq:2CocycleForNondegenerateEnrichment}
\omega(g,h):=(\lambda^{gh})^{-1}\circ \mu_{g,h}^{\cF}\circ g(\lambda^h)\circ \lambda^g : \cF \Rightarrow \cF
\end{equation}
is an element of $\Aut_\otimes(\cF) \cong \Aut_\otimes(\id_\cV)$, which is in turn isomorphic to the
group $\widehat{\cU(\cV)}$ of characters on the universal grading group of $\cV$. In fact $\omega\in
Z^2(G,\widehat{\cU(\cV)})$. Any other choice of $\lambda^g$ for $g\in G$ will give a cohomologous
2-cocycle. We see directly that the second obstruction vanishes if and only if $[\omega]=0$ in
$H^2(G, \widehat{\cU(\cV)})$. Hence the exact sequence \eqref{eq:ExactSequenceEquivariantFunctor} is
exactly
$$
\begin{tikzcd}
1
\ar[rr]
&&
\widehat{\cU(\cV)}
\ar[rr, "\iota"]
&&
\widehat{\cU(\cV)}\times_\omega G
\ar[rr, "\pi"]
&&
G
\ar[rr]
&&
1,
\end{tikzcd}
$$
which splits if and only if $[\omega]=0$.

% \begin{remark}
% For any enrichment $\cF^{\scriptscriptstyle Z}:\cV \to Z(\cC)$ which sends
% simple objects to simple objects (but is not necessarily full, e.g., the forgetful functor
% $\Rep(A)\rightarrow \Vec$ for some abelian group $A$), the same result will hold. \nn{Could we add a little more explanation here? I don't think this is so obvious.}

% \nn{From Corey: I propose removing this remark since we don't do anything with it anyways}

% In other words,
% the $G$-action on $\Aut_\otimes(\cF^\sZ)$ will be trivial, and thus the splitting problem directly
% reduces to a 2-cohomology question.
% \end{remark}

Observe that when $\rho$ passes the first obstruction, the 2-cocycle $\omega$ in
\eqref{eq:2CocycleForNondegenerateEnrichment} automatically vanishes if $\widehat{\cU(\cV)}$ is
trivial, in which case there is a unique splitting.

\begin{cor}
Suppose $(\cC,\cF^{\scriptscriptstyle Z})$ is a $\cV$-fusion category with $\cF^{\scriptscriptstyle Z}$ fully faithful.
Let $\cD$ be an arbitrary $G$-graded extension of $\cC$ for which the first obstruction vanishes.
If $\widehat{\cU(\cV)}$ is trivial, then the $\cV$-enrichment has a unique lifting to $\cD$.
\end{cor}

\begin{example}
If $(\cC, \cF^\sZ)$ is a $\Fib$-fusion category and $\cD$ is a $G$-graded extension of $\cC$ for
which the first obstruction vanishes, then there is  unique lift of the $\Fib$ enrichment to $\cD$.
For an explicit example, one may consider $\cC=\Ad(E_8)$ and $\cD=E_8$.
\end{example}

%%%%%%%%%%%%%%%%%%%%%%%%%%%%%%%%%%%%%%%%%%%
\subsection{Zesting a trivial extension}

For convenience, we assume that $H^4(G,\bbC^\times)=(1)$. Recall a braided categorical action of $G$
on $Z(\cC)$ is called $G$-\emph{stable} if each  $g\in G$ acts by the identity functor. Such actions
are given by twisting the trivial action by a 2-cocycle $\omega \in H^2(G,
\Aut_\otimes(\id_{Z(\cC)})=H^2(G, \Inv(Z(\cC)))$ \cite{MR2677836}. Since $H^4(G,\bbC^\times)=(1)$,
we get a $G$-graded extension $\cD$ of $\cC$ called $\cC\boxtimes_{\omega} \Vec(G)$, which is
$\cC\boxtimes \Vec(G)$ as a linear category with the tensor product functor twisted by $\omega$.
Twisting the monoidal product by a 2-cocycle in this manner is sometimes called \textit{zesting}
c.f.~\cite{MR3641612}.

For such extensions, for \textit{any} enrichment $(\cC, \cF^{\scriptscriptstyle Z})$, the first
obstruction always vanishes, namely $g\circ \cF^{\scriptscriptstyle Z}\cong \cF^{\scriptscriptstyle
Z} $ since $g\cong \id_{Z(\cC)}$. If in addition $\cF^{\scriptscriptstyle Z}$ is fully faithful (or
more generally sends simple objects to simple objects), then we get a restriction map
$R:\Aut_\otimes(\id_{Z(\cC)})\cong \Inv(Z(\cC)) \to \Aut_\otimes(\id_{\cV})\cong
\widehat{\cU(\cV)}$, and the 2-cocycle \eqref{eq:2CocycleForNondegenerateEnrichment} corresponds to
the push forward of $R_{*}\omega\in H^{2}(G, \widehat{\cU(\cV)} )$. Thus we can extend the
enrichment $(\cC, \cF^{\scriptscriptstyle Z})$ if and only if $R_{*}\omega$ is trivial.

For a slightly more explicit example, when $\cV = \Rep(N)$ and $\cC=\Vec(N)$, we have $\Inv(Z(\cC))
\cong \widehat{N} \times Z(N)$ and $\widehat{\cU(\cV)}\cong Z(N)$. Then the push-forward map $R:
\widehat{N} \times Z(N) \to Z(N)$ is the canonical projection to the factor $Z(N)$. In particular,
for any group with $Z(N)=(1)$ and for any $\omega\in H^{2}(G, \widehat{N})$ (with the trivial
action of $G$ on $\widehat{N}$), we can lift the $\Rep(N)$ enrichment on $\Vec(N)$ to the zested extension
$\Vec(N)\boxtimes_{\omega} \Vec(G)$. 
In this case, the latter category is actually equivalent to
$\Vec(N\times G)$.
%, \omega^{\prime})$, where $\omega^{\prime}\in Z^{3}(N\times G, \mathbb{C}^{\times})$
%is possibly a non-trivial 
\begin{verbatim}

\end{verbatim}

\begin{verbatim}

\end{verbatim}

since the 3-cocyle obtained from $\omega$ via a connecting map in the
Lyndon-Hochschild-Serre spectral sequence associated to the short exact sequence 
$1\rightarrow N\rightarrow  N\times G\rightarrow G\rightarrow 1$ 
is trivial
(see \cite[Appendix]{MR2677836}).
Indeed, all differentials in the LHS spectral sequence are zero; we thank the referee for pointing this out.

%%%%%%%%%%%%%%%%%%%%%%%%%%%%%%%%%%%%%%%%%%%%%%%%%%%%%%
\subsection{Fibered enrichments and group theoretical extensions}

In this example, we focus on $\Rep(N)$-fibered enrichments (recall Definition \ref{defn:fibered-enrichment}) with $\cC=\Vec(N)$ and $\cD = \Vec(E)$ for some normal subgroup $N \leq E$ corresponding to a fixed exact sequence 
\begin{equation}
\label{eq:GroupSES}
\begin{tikzcd}
1
\ar[r]
&
N
\ar[r]
&
E
\ar[r]
&
G
\ar[r]
&
1.
\end{tikzcd}
\end{equation}
We now analyze when we can extend the $\Rep(N)$-fibered enrichment on $\Vec(N)$ to $\Vec(E)$. The first step will be to analyze the categorical action of $G$ on the center, and in particular how it restricts to the fibered enrichment.

First, from the extension above we directly define a braided categorical action on $\Rep(N)$. 
Pick a set theortical section $\lambda: G \to E$ of the quotient map $E\rightarrow G$ which we will denote $g\mapsto\lambda_{g}\in E$. 
Then we have $\lambda_{g}\lambda_{h}=\lambda_{gh}n_{g,h}$ for some $n_{g,h}\in N$.
For each $g\in G$, we define $\alpha_{g}\in \uAut^{\br}_{\otimes}(\Rep(N))$ by $\alpha_{g}(\pi, V):=(\pi(\lambda^{-1}_{g}\cdot \lambda_{g}), V)$ on objects, and we set $\alpha_{g}$ to be the identity on morphisms. This has the obvious structure of a (braided) monoidal functor.
We now define monoidal natural isomorphisms $\mu_{g,h}: \alpha_{g}\circ \alpha_{h}\rightarrow\alpha_{gh}$. 
For each $(\pi, V)\in \Rep(N)$, consider the linear map $\pi(n_{g,h})$ on the vector space $V$. Then we have 
\begin{align*}
\pi(n_{g,h})\pi (\lambda^{-1}_{h}\lambda^{-1}_{g} - \lambda_{g} \lambda_{h})
&=\pi(n_{g,h})\pi(n^{-1}_{g,h}\lambda^{-1}_{gh}- \lambda_{gh} n_{g,h})
=\pi(\lambda^{-1}_{gh}- \lambda_{gh}) \pi(n_{g,h}).
\qquad\qquad
\forall g,h\in G.
\end{align*}
Setting $\mu_{g,h}:=\{\mu^{(\pi, V)}_{g,h}:= \pi(n_{g,h})\}_{(\pi, V)\in \Rep(N)}$, we see $\mu_{g,h}:\alpha_{g}\circ \alpha_{h}\rightarrow \alpha_{gh}$ gives a monoidal natural isomorphism of functors.

\begin{lem}\label{lem:actiononrep}
The assignment $g\mapsto \alpha_{g}\in \uAut^{\br}_{\otimes}(\Rep(G))$ together with the monoidal natural isomorphisms $\mu_{g,h}:\alpha_{g}\circ \alpha_{h}\rightarrow \alpha_{gh}$ described above assembles into a categorical action $\alpha: \underline{G}\rightarrow \uAut^{\br}_{\otimes}(\Rep(N))$.
\end{lem}

\begin{proof}
A quick computation shows that the equation we need to verify for all $g,h,k\in G$ and all representations $(\pi, V)$ is the cocycle-type equation
\begin{equation}
\label{eq:CocycleTypeEquation}
\pi(n_{gh,k}\lambda^{-1}_{k} n_{g,h}\lambda_{k})=\pi(n_{g,hk}n_{h,k}).
\end{equation}
From the definition of $n_{g,h}$,
we have
$$
\lambda_{g}\lambda_{h} \lambda_{k}
=
\lambda_{gh}n_{g,h}\lambda_{k}=\lambda_{gh}\lambda_{k} \lambda^{-1}_{k} n_{g,h}\lambda_{k}=\lambda_{ghk} n_{gh,k}\lambda^{-1}_{k} n_{g,h}\lambda_{k}.
$$
On the other hand, we also have
$$
\lambda_{g}\lambda_{h} \lambda_{k}=\lambda_{g} \lambda_{hk} n_{h,k}=\lambda_{ghk} n_{g,hk}n_{h,k}.
$$
Comparing these two expressions, we see $n_{gh,k}\lambda^{-1}_{k} n_{g,h}\lambda_{k}=n_{g,hk}n_{h,k}$ in $N$, so \eqref{eq:CocycleTypeEquation} holds for any representation of $N$.
\end{proof}

Now, we consider $\Vec(E)$ as a $G$-extensions of $\Vec(N)$. This yields a braided categorical action which we denote $\tilde{\alpha}:\underline{G}\rightarrow \uAut^{\br}_{\otimes}(Z(\cC))$.

\begin{lem}
The categorical action $\tilde{\alpha}$ restricts on the canonical copy of $\Rep(N)\subseteq Z(\Vec(N))$ to $\alpha$ defined in Lemma \ref{lem:actiononrep}.
\end{lem}

\begin{proof}
Recall that as a $\Vec(N)$ bimodule, $\Vec(E)\cong \bigoplus_{g\in G} {}_g \Vec(N)$, where here ${}_g \Vec(N)$ can be viewed as the linear category of vector spaces graded by elements of the coset indexed by $g\in G$. 
Let us consider the section $G\ni g\mapsto \lambda_{g}\in E$ chosen above. 
We can identify the simple objects of ${}_{g} \Vec(N)$ as elements $\lambda_{g}n$ for $n\in N$. 
Furthermore $_{g}\Vec(N)\cong \Vec(N)$ as a right $N$ module, where $\lambda_{g} n^{\prime} \triangleleft n:=\lambda_{g}n^{\prime}n$, but the left action of $N$ on ${}_{g}\Vec(N)$ is given by $n\triangleright \lambda_{g} n^{\prime}=\lambda_{g} (\lambda^{-1}_{g} n \lambda_{g}) n^{\prime}$. 
In other words, the left action is twisted by the auto-equivalence $\lambda^{-1}_{g} \cdot \lambda_{g}\in \operatorname{Aut}(N)$.
From the definition of the categorical action $\tilde{\alpha}$ \cite[Eq.~24]{MR2677836} and the canonical copy of $\Rep(N)\subseteq Z(\Vec(N))$, the result follows.
\end{proof}

%\begin{cor}
%The canonical fibered enrichment of $\Vec(N)$ in $\Rep(N)$ extends to $\Vec(E)$ if and only if $E\cong N\times G$.
%\end{cor}

\begin{cor} 
The canonical $\Rep(N)$-fibered enrichment of $\Vec(N)$ extends to $\Vec(E)$ if and only if $E\cong N\times G $.
In this case, these extensions form a torsor over $H^{1}(G, Z(N))$.
\end{cor}
\begin{proof}
Since the canonical fibered enrichment is fully faithful, by the previous lemma we can lift the enrichment if and only if the categorical action $\alpha: \underline{G}\rightarrow \uAut^{\br}_{\otimes}(\Rep(G))$ is isomorphic to the trivial categorical action. 
This would imply, in particular, that each $\alpha_{g}$ is trivial, namely that $\pi(\lambda^{-1}_{g}n \lambda_{g})=\pi(n)$ for all $n\in N$, $g\in G$, and $(\pi, V)\in \Rep(N)$. 
Applying this to the regular representation implies $\lambda^{-1}_{g}n\lambda_{g}=n$, and thus we have a decomposition $E\cong N\times_{\omega} G$ for some 2-cocycle $\omega\in Z^{2}(G, Z(N))$, where the action on the latter coefficient module is trivial. 
Furthermore, we see this 2-cocycle $\omega_{g,h}$ is precisely the $n_{g,h}$ associated to our choice of $\lambda$. 
But since the tensorator for the action $\alpha$ is given by $\mu^{(\pi, V)}_{g,h}=\pi(n_{g,h})$ by definition, we see that the action $\alpha$ is precisely the trivial action twisted by $\omega$. 
Therefore, $\alpha$ is isomorphic to the trivial action precisely when $[\omega]$ is trivial in $H^{2}(G, Z(N))$, which happens precisely when $E$ splits as $N\times G$.
The final claim follows easily.
\end{proof}

\section{Application: classification of \texorpdfstring{$G$}{G}-crossed braidings}
\label{sec:GCrossedBraidings}

An interesting point of view we wish to advocate is that various sorts of structures on a $G$-graded
extension can be equivalent to extensions of an enrichment on the base category. In particular, a
braided fusion category can be canonically enriched over itself. 
In this section, our goal is to
show that (equivalence classes of) $G$-crossed braidings on a $G$-graded fusion category $\cD$ which
restrict on the trivial graded component $\cC$ to some fixed braiding are exactly classified by
extensions of the corresponding self enrichment of $\cC$ to $\cD$.

While this proof essentially boils down to results in \cite{MR2677836,MR3107567,2006.08022} using \cite{MR2678824}, we believe our
point of view sheds new light on $G$-crossed braidings while simultaneously providing intuition for
enriched extensions as being `something like a $G$-crossed braiding'. We then apply our earlier
results to give a classification of $G$-crossed braidings generalizing the results of Nikshych
\cite{MR3943750}. 
This allows us to classify $G$-crossed braidings on a $G$-graded fusion category
$\cD$ in terms of full subcategories of its Drinfeld center, satisfying some conditions.

%%%%%%%%%%%%%%%%%%%%%%%%%%%%%%%%
\subsection{The canonical self-enrichment and \texorpdfstring{$G$}{G}-crossed braidings}

Fix a braided fusion category $\cC$ with braiding $\beta$.

\begin{defn}
%Given a braided fusion category $\cC$ with braiding $\beta$,
The canonical \emph{self-enrichment} $\cC \to Z(\cC)$ is given by $c\mapsto (c,\beta_{c,-})$.
\end{defn}

In \S\ref{sec:ClassificationByMonoidal2Functors}, we defined a monoidal product on $\uuBrPic^\cC(\cC)$ via lifting the product on $\uuBrPic(\cC)$ determined by the universal property discussed in Definition \ref{defn:RelativeDeligneProduct} via the forgetful 2-functor $\Forget_\cC$, which automatically makes $\Forget_\cC$ a monoidal 2-functor.

Recall from \cite[\S4.4]{MR2677836}, \cite[\S2.8]{MR3107567}, or \cite[\S3.3 and 5.2]{2006.08022} that the monoidal 2-groupoid of invertible $\cC$-modules $\uuPic(\cC)$ is also endowed with a monoidal product by lifting the relative product from $\uuBrPic(\cC)$.
In more detail, there is a canonical inclusion 2-functor $\uuPic(\cC) \to \uuBrPic(\cC)$ which identifies the right action with the left action, and one lifts the monoidal product to make this canonical inclusion into a monoidal 2-functor.

Observe now that this monoidal 2-functor $\uuPic(\cC) \to \uuBrPic(\cC)$ natrually factors through $\uuBrPic^\cC(\cC)$!
Indeed, when one defines the right action on an invertible $\cC$-module $\cM$ as equal to the left action, we get an obvious $\cC$-centered structure $\eta^\cM$ given by the identity.
Since the monoidal products on $\uuPic(\cC)$ and $\uuBrPic^\cC(\cC)$ were both lifted from $\uuBrPic(\cC)$, we see we have a commuting triangle:
\begin{equation}
\label{eq:PicFactorsThroughBrPicCC}
\begin{tikzcd}[column sep=6em]
\uuPic(\cC)
\arrow[r,"{(-\vartriangleleft \cC:=\cC\vartriangleright - ,\eta)}"]
\arrow[dr, swap, "-\vartriangleleft \cC:=\cC\vartriangleright - "]
&
\uuBrPic^\cC(\cC)
\arrow[d, swap,"\Forget_{\cC}"']
\\
&
\uuBrPic(\cC)
\end{tikzcd}
\end{equation}
The horizontal arrow in \eqref{eq:PicFactorsThroughBrPicCC} is easily seen to be an equivalence of the underlying 2-groupoids, with inverse (up to equivalence) given by forgetting the right $\cC$-action (cf.~\cite[Def.~2.12 and Rem.~2.13]{MR3107567}).

We now fix a braided fusion category $\cC$ together with a $G$-extension $\cC \subseteq \cD$ as
ordinary fusion categories corresponding to a monoidal 2-functor $\uurho: \uuG \to \uuBrPic(\cC)$ from
\cite{MR2677836}. We are now ready to prove Theorem \ref{thm:ClassificationOfGCrossedBraidings},
which is $(1)\cong (4)$ of the following theorem.

\begin{thm}
\label{thm:ClassificationOfGCrossedBraidingsUsingBrPicC}
Fix a braided fusion category $\cC$ and a $G$-extension $\cC\subset \cD=\bigoplus_{g\in G} \cD_g$ as an ordinary fusion category.
Let $\uupi: \uuG \to \uuBrPic(\cC)$ be any monoidal 2-functor corresponding to $\cD$ under Theorem \ref{thm:ExtensionsAsMonoidal2Functors}.
The following sets are in canonical bijection:
\begin{enumerate}
\item
Lifts of the self $\cC$-enrichment $\cF^\sZ: \cC \to Z(\cC)$ to $\cD$, i.e.,
braided tensor functors
$\widetilde{\cF}^{\sZ}: \cC \to \Rep(G)'\subset Z(\cD)$
such that 
$\Forget_\cC\circ \widetilde{\cF}^{\sZ}=i\circ \cF^{\sZ}$ where $i: Z(\cC) \hookrightarrow Z_\cC(\cD)$ is the canonical inclusion.
\item 
Lifts of $\uupi$ to $\uupi^\cC: \uuG \to \uuBrPic^\cC(\cC)$ such that $\Forget_\cC\circ \uupi^\cC = \uupi$ on the nose.
\item
Lifts of $\cD$ to $Z_{\cC}(\cD)$ that agree with the reversed self enrichment $\cF^{Z}_{\rev}:\cC^{\rev}\rightarrow Z(\cC)$, i.e., tensor functors $F:\cD \to Z_{\cC}(\cD)$ such that $\Forget_Z\circ F = \id_\cD$ on the nose and $F|_{\mathcal{C}}=i\circ \cF^{Z}_{\rev}$.
\item
The equivalence classes of $G$-crossed braidings on $\cD$ (cf.~Example \ref{example:SetOfGCrossedBraidings}).
\end{enumerate}
\end{thm}
\begin{proof}
\item[$\underline{(1)\cong (2):}$]
This is a special case of Theorem \ref{thm:LiftsToEnrichedBrPic} with $\cV=\cC$ and $\cF^\sZ: \cC \to Z(\cC)$ the self-enrichment.

\item[$\underline{(1)\cong (3):}$]
Observe that lifts 
$F:\cD\to Z_\cC(\cD)$ 
such that $
\Forget_Z\circ F = \id_\cD$
are in bijection with 
lifts $\widetilde{\cF}^{\sZ}:\cC\to \Rep(G)'\subset Z(\cD)$ 
such that
$\Forget_\cC \circ \widetilde{\cF}^{\sZ} = i\circ \cF^\sZ$
by taking the inverse half-braiding.

\item[$\underline{(2)\cong (4):}$]
In Example \ref{ex:Lift2Functor} we saw that the set (2) of lifts of $\uupi : \uuG \to \uuBrPic(\cC)$ to $\uuBrPic^\cC(\cC)$ is the strict fiber $\stFib_{\uupi}((\Forget_\cC)_*)$ where $(\Forget_\cC)_* : \Hom(\uuG \to \uuBrPic^\cC(\cC)) \to \Hom(\uuG \to \uuBrPic(\cC))$ is post-composition with $\Forget_\cC: \uuBrPic^\cC(\cC) \to \uuBrPic(\cC)$.

By essentially the same argument as in Example \ref{example:SetOfGCrossedBraidings}, the equivalence classes (4) of $G$-crossed braidings is equivalent to the 0-truncation of the homotopy fiber over $\cD$
$\tau_0(\hoFib_\cD(\Forget_\beta))$
of the forgetful 2-functor $\Forget_\beta: \Ext_{\CrsBrd}(G,\cC) \to \Ext(G,\cC)$.

By Theorem \ref{thm:ExtensionsAsMonoidal2Functors}, we have an equivalence of 2-groupoids $\Hom(\uuG\to \uuBrPic(\cC)) \cong \Ext(G,\cC)$, and by \cite[Thm.~7.12]{MR2677836} (see also \cite[Prop.~8.11 and Thm.~8.13]{2006.08022}), there is an equivalence of 2-groupoids $\Hom(\uuG, \uuPic(\cC)) \cong \Ext_{\CrsBrd}(G,\cC)$, where the latter denotes the 2-groupoid of $G$-crossed braided extensions of $\cC$.
Since the horizontal arrow in \eqref{eq:PicFactorsThroughBrPicCC} above is a monoidal 2-equivalence, we see $\Hom(\uuG, \uuBrPic^\cC(\cC)) \cong \Ext_{\CrsBrd}(G,\cC)$.

Putting it all together, we have a (weakly) commuting square of 2-functors
$$
\begin{tikzcd}
\Hom(\uuG, \uuBrPic^\cC(\cC))
\arrow[d,"(\Forget_\cC)_*"]
\arrow[r,"\cong"]
&
\Ext_{\CrsBrd}(G,\cC)
\arrow[d,"\Forget_\beta"]
\\
\Hom(\uuG, \uuBrPic(\cC))
\arrow[r,"\cong"]
&
\Ext(G,\cC)
\end{tikzcd}
$$
Since $\uupi$ maps to $\cD$ under the lower horizontal arrow, the homotopy fibers at $\uupi$ and $\cD$ are equivalent.
We thus have canonical bijections
$$
\stFib_{\uupi}((\Forget_\cC)_*)
\cong
\tau_0(\hoFib_{\uupi}((\Forget_\cC)_*))
\cong
\tau_0(\hoFib_{\cD}(\Forget_\beta))
\cong
\left\{\,\parbox{4.2cm}{\rm Equivalence classes of $G$-crossed braidings on $\cD$}\,\right\}
%\{\text{eq. classes of $G$-crossed braidings}\}
%\stFib_\cD(\Forget_\beta).
$$
where $\tau_0$ denotes the 0-truncations of the homotopy fibers.
This completes the proof.
\end{proof}

\subsection{Classification of \texorpdfstring{$G$}{G}-crossed braidings on a fixed \texorpdfstring{$G$}{G}-graded fusion category}

We can use Theorem \ref{thm:ClassificationOfGCrossedBraidings} to obtain a classification of G-crossed braidings on a G-graded fusion category generalizing a similar style of classification by Nikshych of braidings on a fusion category \cite{MR3943750}. 
Recall that if $A\in \mathcal{C}$ is an algebra object, a subcategory $\mathcal{D}\subseteq \mathcal{C}$ is called \emph{transverse} to $A$ if for all objects $d\in \mathcal{D}$, $\mathcal{C}(d\to A)=\mathcal{C}(d\to 1)$.
Recall that a subcategory $\cD$ of a category $\cC$ is called \emph{replete} if for all triples $(c,d,f)$ with $c\in \cC$, $d\in \cD$, and $f: c \to d$ an isomorphism, we have $c\in \cD$ and $f\in \cD(c\to d)$.

\begin{thm}
\label{Gcrossedthm} 
Let $\cD=\bigoplus \cD_{g}$ be a faithfully $G$-graded fusion category and $\Rep(G)\subseteq Z(\cD)$ the canonical subcategory of the center. 
Then $G$-crossed braidings on $\cD$ are classified by full and replete fusion subcategories $\cA\subseteq Z(\cD)$ satisfying the following properties:
\begin{enumerate}
  \item
  $\cA\subseteq \Rep(G)^{\prime}$.
  \item
    $|G|\FPdim(\cA)=\FPdim(\cD)$.
    \item
    $\cA$ is transverse to $I(1)$, i.e. for any $a\in \cA$, $Z(\cD)(a\to I(1))=Z(\cD)(a\to 1)$, where $I$ is the right adjoint of the forgetful functor $\Forget_Z:Z(\cD) \to \cD$.
\end{enumerate}
\end{thm}

\begin{proof}
We have just shown that G-crossed braidings are classified by braidings on the trivial component
$\cD_{e}$, and a lift of this braided category to $Z(\cD)$. Given a braiding $\sigma$ on $\cD_{e}$
which lifts to the center of $\cD$, this defines a full subcategory $\cA\subseteq
\Rep(G)^{\prime}\cap Z(\cD)$, which is equivalent as a braided fusion category to $\cD_{e}$ with braiding $\sigma$. 
By construction $\FPdim(\cA)=\FPdim(\cD_{e})=\FPdim(\cD)/|G|$ as desired. 
Furthermore,
since the forgetful functor $\Forget_Z|_{\cA}$ is fully faithful, $\cA$ is transverse to $I(1)$.

Conversely, given a subcategory $\cA\subseteq \Rep(G)^{\prime}\cap Z(\cD)$ (condition (3)) which is
transverse to $I(1)$ (condition (2)), $\Forget_Z|_{\cA}$ is fully faithful. Since $\cA$ centralizes
$\Rep(G)$, $\Forget(\cA)\subseteq \cD_{e}$. If moreover condition (1) holds, 
$$
\FPdim(\cA)
=
\FPdim(\Forget_Z(\cA))
=
\frac{\FPdim(\cD)}{|G|}
=
\FPdim(\cD_{e}),
$$ 
and thus $\Forget_Z|_{\cA}$ is an equivalence. 
Thus we can transport the half-braidings induced from $\cA$ onto $\cD_{e}$, to obtain a braiding which lifts to the center.

It is clear these two constructions are mutually inverse.
\end{proof}

We note that a $G$-crossed braiding, by definition, is additional structure on a fusion category
consisting of an entire categorical action by $G$ and a family of natural isomorphisms satisfying complicated
coherences. In the following two subsections, we apply Theorem \ref{Gcrossedthm} to provide a
complete classification of $H$-crossed braidings on group theoretical categories of the form
$\Vec(G, \omega)$ and $\Rep(G)$.

%%%%%%%%%%%%%%%%%%%%%%%%%%%%%%%%%%%%%%%%%
\subsection{Example: \texorpdfstring{$\Vec(G, \omega)$}{Vec G omega}}

First, we consider pointed categories $\cD=\Vec(G, \omega)$ where $G$ is a finite group and $\omega\in Z^{3}(G, \mathbb{C}^{\times})$.

We recall the results of \cite{MR2552301}, which classifies fusion subcategories of $Z(\Vec(G, \omega))$.
To state these results, given a normalized 3-cocycle $\omega$, for any triple of elements $a,g,h\in G$, we define the function
$$
\beta_{a}(g,h):=\frac{\omega(a,g,h) \omega(g,h,h^{-1}g^{-1}agh)}{\omega(g,g^{-1}ag, h)}
$$
Letting $C_{G}(a)=\{g\in G\ : ga=ag\}$, then $\beta_{a}|_{C_{G}(a)\times C_{G}(a)}\in
Z^{2}(C_{G}(a), \mathbb{C}^{\times})$. Isomorphism classes of simple objects in $Z(\cC)$ are then
classified by pairs $(a, \chi)$, where $a\in G$ is a representative of a conjugacy class and $\chi$
is an irreducible $\beta_{a}$-projective representation of $C_{G}(a)$. \cite{MR1048699,MR1770077}

\begin{defn} Let $L,M \triangleleft G$ be commuting normal subgroups. 
A function $B:L\times M\rightarrow \mathbb{C}^{\times}$ is called an $\omega$-\emph{bicharacter} if
\begin{enumerate}
\item
$B(\ell,mn)=\beta^{-1}_{\ell}(m,n)B(\ell,m)B(\ell,n)$
for all
$\ell\in L$ and $m,n\in M$
\item
$B(k\ell,m)=\beta_{m}(k,\ell)B(k,m)B(\ell,m)$
for all
$k,\ell\in L$ and $m\in M$
\end{enumerate}
An $\omega$-bicharacter $B:L\times M\rightarrow \mathbb{C}^{\times}$ is called $G$-\emph{invariant} if moreover
\begin{enumerate}
\item[(3)]
$
B(g^{-1}\ell g, m)
=
\beta_{\ell}(g,m) \beta_{\ell}(gm,g^{-1})\beta^{-1}_{\ell}(g,g^{-1})B(\ell, gmg^{-1})
$
for all $g\in G$, $\ell\in L$, and $m\in M$.
\end{enumerate}
\end{defn}

We recall the following classification theorem.

\begin{thm}[{\cite[Thm.~5.11]{MR2552301}}]
Full and replete fusion subcategories of $Z(\Vec(G,\omega))$ are classified by the following data
\begin{itemize}
\item
a pair $L,M$ of commuting normal subgroups of $G$, and
\item
a $G$-invariant $\omega$-bicharacter $B:L\times M\rightarrow \mathbb{C}^{\times}$.
\end{itemize}
\end{thm}

Given such an abstract fusion subcategory $\cA$, the subgroup $L$ is determined by the normal
subgroup of $G$ generated by the image of the forgetful functor, while $M$ is determined by
$\Rep(G/M)=\cA\cap \Rep(G)$, where $\Rep(G)$ denotes the canonical copy of $\Rep(G)\subset
Z(\Vec(G,\omega))$. See \cite{MR2552301} for an explanation of the role of the bicharacter $B$.

We denote the subcategory associated to the above data as $\cS(L,M, B)$.
In this notation, the canonical subcategory $\Rep(G)$ is $\cS(1,1, 1)$, and the trivial subcategory $\Vec$ is $\cS(1,G, 1)$.
We further recall the following facts from \cite{MR2552301}.

\begin{itemize}
\item  $\FPdim(\cS(L,M,B))=|L|[G:M]$ \cite[Lem.~5.9]{MR2552301}.
\item  $\cS(L,M, B)^{\prime}=\cS(M,L, (B^{\op})^{-1})$ \cite[Lem.~5.10]{MR2552301}.
\item  $\cS(L,M, B)\subseteq \cS(L^{\prime}, M^{\prime}, B^{\prime})$ if and only if $L\subseteq L^{\prime}$, $M^{\prime}\subseteq M$ and $B|_{L\times M^{\prime}}=B^{\prime}|_{L\times M^{\prime}}$ \cite[Prop.~6.1]{MR2552301}.
\end{itemize}

\begin{prop}
Suppose we have a faithful $H$-grading 
on $\Vec(G, \omega)$ given by a surjective homomorphism $\pi:G\rightarrow H$. 
Then if $\Vec(G, \omega)$ admits an $H$-crossed braiding, $\ker(\pi)\subseteq Z(G)$. 
In this case, H-crossed braidings are classified by $G$-invariant $\omega$-bicharacters $B: \ker(\pi)\times G\rightarrow \mathbb{C}^{\times}$.
\end{prop}
\begin{proof}
It suffices to show that subcategories of $Z(\Vec(G,\omega))$ satisfying the conditions of
\ref{Gcrossedthm} are precisely those of the form $\cS(\ker(\pi), G, B)$, where $B$ is an arbitrary
$G$-invariant $\omega$-bicharacter. Note that $\cS(L,M,B)$ is transverse to $I(1)$ if and only if
$\Rep(G/M)=\cS(L,M,B)\cap \Rep(G)=\Vec$, since $I(1)=\cO(G)\in \Rep(G)$ contains all the irreducible
objects of $\Rep(G)$. Thus $M=G$. Note this implies $L\le Z(G)$, since $L$ must centralize $M$. Now
observe that $\cS(L,G,B)$ centralizes $\Rep(H)=\Rep(G/\ker(\pi))=\cS(1, \ker(\pi), 1)$ if and only
if $\cS(L,G,B)\le \cS(\ker(\pi), 1, 1)$, which can be restated as $L\le \ker(\pi)$, $1\le G$, and
$B|_{L\times 1}=1$, where the last follows automatically from the properties of
$\omega$-bicharacters. Finally, the third condition is $\FPdim(\cS(L,G,B))=|G|/|H|=|\ker(\pi)|$.
However, $\FPdim(\cS(L,G,B))=|L|[G:G]=|L|$, and thus we must have $|L|=|\ker(\pi)|$. But since $L\le
\ker(\pi)$, we must have equality, which concludes the proof.
\end{proof}

As a special case, we recover the following well known corollary.

\begin{cor}
There is a unique $G$-crossed braiding on $\Vec(G,\omega)$.
\end{cor}

\begin{remark}
Recall that a braiding on a $G$-graded fusion category $\cD$ can be viewed as a $G$-crossed braiding
together with an extra piece of data, namely a trivialization of the categorical action
$\uG\rightarrow \uAut_{\otimes}(\cD)$. For example, when $G$ is abelian, we have a unique
$G$-crossed braiding on $\Vec(G)$, where the $G$ action is by conjugation, and the $G$-braiding is
the identity. However, we have several different braidings on  $\Vec(G)$ which correspond to
distinct trivializations of the conjugation action, which it is easy to show correspond to
bicharacters on $G$.
\end{remark}

%%%%%%%%%%%%%%%%%%%%%%%%%%%%%%%%%%%%%%%%%
\subsection{Example: \texorpdfstring{$\Rep(G)$}{Rep(G)}}

Now we consider the case where $\cD=\Rep(G)$, and we consider its center in terms of $Z(\Vec(G))$,
where we can use the convenient description as above. In this case, the universal grading group is
the dual group $\widehat{Z(G)}$. The copy of $\Rep(\widehat{Z(G)})\cong \Vec(Z(G))$ sitting inside
$Z(\Vec(G))$ is identified with the objects which are direct sums of objects $(z,1)$ where $z\in
Z(G)$ represents a conjugacy class, and $1$ is the trivial representation of the centralizer
subgroup of $z$ (which is $G$).

Note that all (normal) subgroups of $\widehat{Z(G)}$ are of the form 
$$
H^{\perp}=\set{\gamma\in \widehat{Z(G)}}{\gamma(h)=1\ \forall h\in H}
$$
for some $H\leq Z(G)$.
Thus faithful grading groups are given by quotients $\widehat{Z(G)}/H^{\perp}$, and $\Rep(\widehat{Z(G)}/H^{\perp})\cong \Vec(H)\subseteq \Vec(Z(G))$. 

We have the following result:

\begin{prop}
For $H\le Z(G)$, faithful $\widehat{Z(G)}/H^{\perp}$-crossed braidings on $\Rep(G)$ are classified by triples $(L,M,B)$, such that
\begin{itemize}
\item
$M\triangleleft G$ is normal such that $H\leq M$ and $M/H$ is abelian,
\item
$L\triangleleft G$ is abelian and commutes with $M$, and
\item
$B:L\times M/H\rightarrow \mathbb{C}^{\times}$ is a non-degenerate G-invariant bicharacter.
\end{itemize}
\end{prop}

\begin{proof}
Consider the faithful $\widehat{Z(G)}/H^{\perp}$-crossed braiding
on $\Rep(G)$ corresponding to the 
fusion subcategory $\cS(L,M,B)\subseteq Z(\Vec(G))=Z(\Rep(G))$
under Theorem \ref{Gcrossedthm}.
%Consider the fusion subcategory $\cS(L,M,B)\subseteq Z(\Vec(G))=Z(\Rep(G))$. 
\item[\underline{Step 1:}]
The subgroups $L,M \leq G$ and the bicharacter $B$ satisfy:
\begin{itemize}
\item
$M\triangleleft G$ is normal such that $H\leq M$ and $[M:H]=|L|$.
\item
$L\triangleleft G$ commutes with $M$, and
%\item $L,M$ are commuting normal subgroups of $G$ with $H\le M$ and $[M:H]=L$.
\item
$B: L\times M\rightarrow \mathbb{C}^{\times}$ is a G-invariant bicharacter such that 
$B|_{L\times H}=1$ and the homorphism $\widehat{B}: L\rightarrow \widehat{M},\ l\mapsto B(l,\ \cdot\ )$ is injective.
\end{itemize}
\begin{proof}[Proof of Step 1]
The canonical copy of $\Rep(\widehat{Z(G)}/H^{\perp})\cong \Vec(H)$ is given by the subcategory
$\cS(H, G, 1)$, whose centralizer is $\cS(G, H, 1)$. Thus $\cS(L,M,B)\subseteq \cS(H, G,
1)^{\prime}$ if and only if $L\le G, H\le M$ and $B_{L\times H}=1$. Now, the FP dimension condition
is satisfied if and only if $\FPdim(\cS(L,M,B))=|L|[G:M]=[G:H]$, which happens if and only if
$|L|=\frac{|M|}{|H|}=[M:H]$. Finally, $\cS(L,M,B)$ is transverse to the Lagrangian algebra $I(1)$
(for $\Rep(G)$) if and only if the homorphism $\widehat{B}: L\rightarrow \widehat{M},\ l\mapsto
B(l,\ \cdot\ )$ is injective by \cite[Lem.~5.1]{MR3943750}. 
\end{proof}

\item[\underline{Step 2:}]
$L$ and $M/H$ are abelian, and $\widehat{B}:L\rightarrow \widehat{M/H}$ given by
$\widehat{B}(l):=B(l,\ \cdot\ )$ is an isomorphism, which gives non-degeneracy of the bicharacter.

\begin{proof}[Proof of Step 2]
%Now we will show that the conditions in the claim actually imply that $L$ and $M/H$ are abelian, and that $\widehat{B}:L\rightarrow \widehat{M/H},\ \widehat{B}(l):=B(l,\ \cdot\ )$ is an isomorphism, which gives non-degeneracy of the bicharacter.
By the first condition in Step 1, $|L|=|M/H|$.
By the third condition, $\widehat{B}: L\rightarrow \widehat{M/H}$ is an injection. 
Thus we have 
$$
|M/H|=|L|\le |\widehat{M/H}|=|(M/H)/ [M/H,M/H] |\le |M/H|.
$$ 
This forces the equality $|(M/H)/ [M/H,M/H] |= |M/H|$, and thus $M/H$ is abelian. 
Furthermore, this implies $|M/H|=|\widehat{M/H}|$, and thus the injective map $\widehat{B}$ is an isomorphism as claimed.
\end{proof}
\let\qed\relax
\end{proof}

We now consider some examples. 

\begin{example}
When $\widehat{Z(G)}/H^{\perp}=1$ so that $H=1\le Z(G)$, then we should recover braidings on
$\Rep(G)$, which have been classified by \cite{q-alg/9706007} and again by \cite{1801.06125}, and
indeed this is the case.
\end{example}

\begin{example}
Consider the case $H=Z(G)$, so that the grading on $\Rep(G)$ is the universal grading. Then choosing
$M=Z(G)$ and $L=1$ and $B=1$, we obtain the usual braiding on $\Rep(G)$, viewed as a $G$-crossed
braiding.
\end{example}

\renewcommand*{\bibfont}{\small}
\setlength{\bibitemsep}{0pt}
\raggedright
\printbibliography

\end{document}